\pdfoutput=1
\documentclass[11pt,twoside]{article}

\usepackage[margin=1in]{geometry} 
\usepackage{amsmath,amsthm,amssymb}
\usepackage{color} 
\usepackage{mathtools,mathptmx}
\usepackage{indentfirst}
\usepackage{mathptmx}
\usepackage{cancel}
\usepackage{amsbsy}  %% for bold face Greek Letters
\usepackage{amsfonts}
\usepackage{graphicx}
\usepackage{tikz,tkz-euclide}

\usepackage{array}
\usepackage{geometry}
\usepackage{cases}

\usepackage{hyperref}
\hypersetup{colorlinks=true, allcolors=blue}
\usepackage[square, numbers, comma, sort&compress]{natbib}  % Use the "Natbib" style for the references in the Bibliography
\usepackage[all]{xy}
\usepackage{blkarray}
\usepackage{datetime}
\usepackage{enumerate} 
\usepackage{enumitem}    
\usepackage{relsize}
\usepackage[super]{nth}
\usepackage{imakeidx}

\newtheorem{theorem}{Theorem}[section]
\newtheorem{proposition}[theorem]{Proposition}
\newtheorem{corollary}[theorem]{Corollary}
\newtheorem{lemma}[theorem]{Lemma}

\numberwithin{equation}{section}
\numberwithin{theorem}{section}

\newcommand{\dis}{\displaystyle}
\newcommand{\R}{\mathbb{R}}
\newcommand{\C}{\mathbb{C}}
\newcommand{\Z}{\mathbb{Z}}
\newcommand{\N}{\mathbb{N}}
\newcommand{\Q}{\mathbb{Q}}

\newcommand{\Real}{\operatorname{Re}}

\newcommand{\dd}{\mathrm{d}}
\newcommand{\dx}[1][x]{\mathrm{d}#1}
\newcommand{\DiffOp}[1][x]{\frac{\mathrm{d}}{\dx[#1]}}
\newcommand{\DiffOpHigherOrder}[2][x]{\frac{\mathrm{d}^{#2}}{\dx[#1]^{#2}}}

\newcommand{\pochhammer}[2][n]{\left(#2\right)_{#1}}

\newcommand{\Hypergeometric}[5][x]{{}_{#2} F_{#3} \left( \left.\begin{matrix} #4 \\ #5 \end{matrix} \, \right| \, #1\right)}
\newcommand{\HypergeometricOneLine}[5][x]{{}_{#2} F_{#3} \left(\left.#4;#5\,\right|#1\right)}

\newcommand{\MeijerG}[5][x]{G_{\,#3}^{\,#2}\left(\left.\begin{matrix} #4 \\ #5 \end{matrix} \, \right| \, #1\right)}

\newcommand{\ceil}[1]{\left\lceil #1 \right\rceil}
\newcommand{\seq}[2][n\in\N]{\left(#2\right)_{#1}}

\newcommand{\n}{\vec{n}}

\newcommand{\Functional}[2]{#1\left[#2\right]}

\newcommand{\mStieltjesRogersPoly}[3][m]{S_{#2}^{(#1)}\left(#3\right)}
\newcommand{\modifiedStieltjesRogersPoly}[4][m]{S_{#2}^{(#1;#3)}\left(#4\right)}
\newcommand{\generalisedStieltjesRogersPoly}[4][m]{S_{#2,#3}^{(#1)}\left(#4\right)}
\newcommand{\generalisedStieltjesRogersPolyTypeJ}[5][m]{S_{#2,#3}^{(#1;#4)}\left(#5\right)}

\setlist[itemize]{noitemsep, topsep=0pt}

\begin{document}

\renewcommand{\baselinestretch}{1.1}
\title{Multiple orthogonal polynomials associated with branched continued fractions for ratios of hypergeometric series}
\author{H\'elder Lima\thanks{Email address: helder.lima@kuleuven.be}\vspace*{0,25 cm}\\
\small{Department of Mathematics, KU Leuven, Celestijnenlaan 200B box 2400, 3001 Leuven, Belgium}}
\date{To appear in \textit{Advances in Applied Mathematics}}
\maketitle
\vspace*{-0.5 cm}
\thispagestyle{empty}
	
\begin{abstract}
The main objects of the investigation presented in this paper are branched-continued-fraction representations of ratios of contiguous hypergeometric series and type II multiple orthogonal polynomials on the step-line with respect to linear functionals or measures whose moments are ratios of products of Pochhammer symbols.
This is an interesting case study of the recently found connection between multiple orthogonal polynomials and branched continued fractions that gives a clear example of how this connection leads to considerable advances on both topics.

We start by obtaining new results about generating polynomials of lattice paths and total positivity of matrices and giving new contributions to the general theory of the connection between multiple orthogonal polynomials and branched continued fractions with emphasis on its application to the analysis of multiple orthogonal polynomials.
Then, we construct new branched continued fractions for ratios of contiguous hypergeometric series. 
We give conditions for positivity of the coefficients of these branched continued fractions and we show that the ratios of products of Pochhammer symbols are generating polynomials of lattice paths for a special case of the branched continued fractions under study.
Next, we introduce a family of type II multiple orthogonal polynomials on the step-line associated with those branched continued fractions.
We present a formula as terminating hypergeometric series for these polynomials, we study their differential properties, and we find an explicit recurrence relation satisfied by them.
Finally, we focus the analysis of the multiple orthogonal polynomials to the cases where the corresponding branched-continued-fraction coefficients are all positive.
In those cases, the orthogonality conditions can be written using measures on the positive real line involving Meijer G-functions and we obtain results about the location of the zeros and the asymptotic behaviour of the polynomials.
%
%Specialisations of the multiple orthogonal polynomials studied here include the classical Laguerre, Jacobi, and Bessel orthogonal polynomials, multiple orthogonal polynomials with respect to Nikishin systems of two measures involving modified Bessel functions, confluent hypergeometric functions, and Gauss’ hypergeometric function, and multiple orthogonal polynomials with respect to Meijer G-functions used to investigate the singular values of products of Ginibre random matrices as well as a $r$-orthogonal polynomial sequence with constant recurrence coefficients for any positive integer $r$ and particular instances of the Jacobi-Pi\~neiro polynomials.
\end{abstract}



%\newpage
\renewcommand{\baselinestretch}{1.0}\normalsize
\tableofcontents
\renewcommand{\baselinestretch}{1.1}\normalsize
%\setcounter{tocdepth}{1} 

%\newpage
\section{Introduction and motivation}
This paper gives a detailed investigation of a case study of the connection between two different corners of Mathematics: multiple orthogonal polynomials, studied by the special-functions community, and branched continued fractions, introduced by the enumerative-combinatorics community to solve total-positivity problems.
This connection was introduced and analysed in the recent paper \cite{AlanSokalMOPd-opProdMatBCF}.

The main objects of the investigation presented here are branched-continued-fraction representations for ratios of contiguous hypergeometric series ${}_{r+1}F_s$ and multiple orthogonal polynomials with respect to $m=\max(r,s)$ linear functionals or measures whose moments are ratios of Pochhammer symbols
\begin{equation}
\label{moment sequence ratio of Pochhammers}
\seq{\frac{\pochhammer{a_1}\cdots\pochhammer{a_r}}{\pochhammer{b_1}\cdots\pochhammer{b_s}}} %\quad\text{for all } n\in\N,
\quad\text{with }(r,s)\in\N^2\backslash\{(0,0)\}.
\end{equation}
The Pochhammer symbol $\pochhammer{z}$, also known as the rising factorial, is defined by 
%$\pochhammer[0]{z}=1$ and $\pochhammer[n]{z}:=z(z+1)\cdots(z+n-1)$ for $n\in\Z^+$.
\begin{equation}
\pochhammer[0]{z}=1 
\quad\text{and}\quad
\pochhammer[n]{z}:=z(z+1)\cdots(z+n-1)
\quad\text{for  }n\geq 1.
\end{equation}
%\textbf{Terminology: Pochhammer symbol or rising factorial? Notation: $\pochhammer[n]{z}$ or $z^{\overline{n}}$?}

These branched continued fractions and multiple orthogonal polynomials are linked because the ordinary generating function of the moment sequence in \eqref{moment sequence ratio of Pochhammers} is
\begin{equation}
\label{generating function of a generic ratio of Pochhammer symbols}
\mathlarger{\sum}_{n=0}^{\infty}\frac{\pochhammer{a_1}\cdots\pochhammer{a_r}}{\pochhammer{b_1}\cdots\pochhammer{b_s}}\,t^n
=\Hypergeometric[t]{r+1}{s}{a_1,\cdots,a_r,1}{b_1,\cdots,b_s}.
%=\frac{\Hypergeometric[t]{r+1}{s}{a_1,\cdots,a_r,1}{b_1,\cdots,b_s}}{\Hypergeometric[t]{r+1}{s}{a_1,\cdots,a_r,0}{b_1,\cdots,b_s}}.
\end{equation}
The generating function \eqref{generating function of a generic ratio of Pochhammer symbols} is the special case $a_{r+1}=1$ of ratios of contiguous hypergeometric series for which branched-continued-fraction representations were introduced in \cite{AlanSokalM.PetroelleB.Zhu-LPandBCF1} and are generalised here.

There are several already known particular cases of multiple orthogonal polynomials with respect to measures whose moment sequences are of the form \eqref{moment sequence ratio of Pochhammers}.
The cases where $m=1$ are $(r,s)$ equal to $(1,0)$, $(1,1)$, and $(0,1)$, and they correspond to the well-known Laguerre, Jacobi, and Bessel classical orthogonal polynomials, respectively.
When $s\leq r=2$, we obtain multiple orthogonal polynomials with respect to Nikishin systems of $2$ measures, supported on the whole positive real line if $s<r$ (i.e., $s\in\{0,1\}$) or on the interval $(0,1)$ if $s=r$ (i.e., $s=2$).
%If $(r,s)=(2,0)$, the two orthogonality measures involve modified Bessel functions of the second kind, also known as Macdonald functions, and the corresponding multiple orthogonal polynomials were investigated in \cite{SemyonWalter,BenCheikhDouak,WalterCoussementMacdonaldFunctions}. 
%If $(r,s)=(2,1)$, the two orthogonality measures involve confluent hypergeometric functions of the second kind, also known as Tricomi functions, and the corresponding multiple orthogonal polynomials were investigated in  \cite{PaperTricomiWeights}.
%If $(r,s)=(2,2)$, the two orthogonality measures involve Gauss' hypergeometric function and the corresponding multiple orthogonal polynomials were investigated in  \cite{PaperHypergeometricWeights}.
The multiple orthogonal polynomials corresponding to the cases $(r,s)$ equal to $(2,0)$, $(2,1)$, and $(2,2)$ were investigated in \cite{SemyonWalter,BenCheikhDouak,WalterCoussementMacdonaldFunctions}, \cite{PaperTricomiWeights}, and \cite{PaperHypergeometricWeights}, respectively;
their orthogonality measures involve modified Bessel functions of the second kind, confluent hypergeometric functions of the second kind, and Gauss' hypergeometric function, respectively.
Finally, when $s=0$ and $r$ is an arbitrary positive integer, the moments in \eqref{moment sequence ratio of Pochhammers} reduce to products of Pochhammer symbols, which are associated with multiple orthogonal polynomials with respect to $r$ measures on the positive real line involving Meijer G-functions $G_{r,0}^{\,0,r}$ introduced in \cite{KuijlaarsZhang14} to investigate singular values of products of Ginibre random matrices.
The special cases $r=1$ and $r=2$ of the latter polynomials are, respectively, the classical Laguerre polynomials and the multiple orthogonal polynomials investigated in \cite{SemyonWalter,BenCheikhDouak,WalterCoussementMacdonaldFunctions}. 

When $m=1$, the branched continued fractions appearing in this paper reduce to classical continued fractions.
If $(r,s)=(1,1)$, we recover Gauss' continued fraction for ratios of contiguous ${}_2F_1$; 
if $(r,s)$ is equal to $(1,0)$ or $(0,1)$, we obtain known continued-fraction representations for ratios of contiguous ${}_2F_0$ and ${}_1F_1$, respectively, which can be obtained as limiting cases of Gauss' continued fraction.
These continued fractions are connected to the Jacobi, Laguerre, and Bessel polynomials, respectively.
These are particular instances of the well-known relation between continued fractions and orthogonal polynomials (see \cite[Ch.~III.4]{ChiharaBook} and \cite{JZengCombinatoricsOfOPs}).

%\newpage
We end this introductory section with an outline of the paper.

In Section \ref{Background}, we give some background on the main topics involved.
In particular, we give a brief introduction to multiple orthogonal polynomials and to branched continued fractions.

In Section \ref{Results about generalised and modified m-S.-R. poly}, we present new results about generalised and modified $m$-Stieltjes-Rogers polynomials (which are generating polynomials of lattice paths, more precisely partial $m$-Dyck paths) and total positivity. 
These results were motivated by their applications to the study of multiple orthogonal polynomials, which are made clear in Section \ref{Connection of MOP and BCF}, but they are worthy of interest on their own.

In Section \ref{Connection of MOP and BCF}, we revisit the connection between multiple orthogonal polynomials and branched continued fractions introduced in \cite{AlanSokalMOPd-opProdMatBCF} and give new contributions to the study of this connection, with emphasis on its application to the analysis of multiple orthogonal polynomials.

In Section \ref{BCF for ratios of hypergeometric series}, we construct new branched continued fractions for ratios of contiguous hypergeometric series, which generalise the branched continued fractions introduced in \cite[\S 14]{AlanSokalM.PetroelleB.Zhu-LPandBCF1}. 
In addition, we give conditions for non-negativity of the coefficients of our branched continued fractions and we show that the modified $m$-Stieltjes-Rogers polynomials linked to the special case $a_{r+1}=1$ of our branched continued fractions are ratios of products of Pochhammer symbols of the form in \eqref{moment sequence ratio of Pochhammers} multiplied by a binomial coefficient.

In Section \ref{Type II MOP w.r.t. linear functionals}, we analyse the type II multiple orthogonal polynomials on the step-line with respect to linear functionals whose moments are the ratios of products of Pochhammer symbols shown to be modified $m$-Stieltjes-Rogers polynomials in Section \ref{BCF for ratios of hypergeometric series}.
We obtain explicit expressions for these multiple orthogonal polynomials as terminating hypergeometric series, we derive differential properties satisfied by them, and we use their connection with the branched continued fractions constructed in Section \ref{BCF for ratios of hypergeometric series} to obtain expressions for their recurrence coefficients as combinations of branched-continued-fraction coefficients.

In Section \ref{Generalised m-S.-R. poly}, we present an explicit formula for the generalised $m$-Stieltjes-Rogers polynomials corresponding to the special case $a_{r+1}=1$ of the branched continued fractions constructed in Section \ref{BCF for ratios of hypergeometric series}. %as a consequence of the explicit expressions for the multiple orthogonal polynomials studied in Section \ref{Type II MOP w.r.t. linear functionals}.

In Section \ref{Type II MOP w.r.t. Meijer G-functions}, we focus the analysis of the multiple orthogonal polynomials introduced in Section \ref{Type II MOP w.r.t. linear functionals} to the cases where the corresponding branched-continued-fraction coefficients are all positive.
In those cases, the linear functionals of orthogonality are induced by measures on the positive real line whose densities are Meijer G-functions, we have positivity of the recurrence coefficients, we obtain results about the location of the zeros and the asymptotic behaviour of the multiple orthogonal polynomials, and we show that special cases of these polynomials include polynomial sequences with constant recurrence coefficients and particular instances of the Jacobi-Pi\~neiro polynomials.

We finish the paper with some final remarks about applications and future directions of investigation related with the work presented here.

In summary, this paper brings to light new contributions to the general theory of lattice paths and branched continued fractions (Section \ref{Results about generalised and modified m-S.-R. poly}) and to the general theory of multiple orthogonal polynomials via their connection with branched continued fractions (Section \ref{Connection of MOP and BCF}) as well as new results about its main objects of investigation: branched continued fractions for ratios of hypergeometric series (Sections \ref{BCF for ratios of hypergeometric series} and \ref{Generalised m-S.-R. poly}) and a new class of hypergeometric multiple orthogonal polynomials (Sections \ref{Type II MOP w.r.t. linear functionals} and \ref{Type II MOP w.r.t. Meijer G-functions}).

%\newpage
\section{Background}
\label{Background}
In this section we present some definitions and basic results about hypergeometric series, multiple orthogonal polynomials, lattice paths and branched continued fractions,  production matrices, and total positivity. %and oscillation matrices. %\textbf{And more???}

\subsection{Hypergeometric series}
%\textbf{Move the definition to the introduction and the other formulas to where they are used?}
For $p,q\in\N$, the (generalised) hypergeometric series (see \cite{AndrewsAskeyRoySpecialFunctions,LukeSpecialFunctionsVolI,DLMF}) is defined by
\begin{equation}
\label{generalised hypergeometric series}
\Hypergeometric[z]{p}{q}{a_1,\cdots,a_p}{b_1,\cdots,b_q}
=\sum_{n=0}^{\infty}\frac{\pochhammer{a_1}\cdots\pochhammer{a_p}}{\pochhammer{b_1}\cdots\pochhammer{b_q}}\,\frac{z^n}{n!}.
\end{equation}
We treat this expression as a formal power series. 
Convergence of non-terminating hypergeometric series plays no role here, except in Proposition \ref{Mehler-Heine formula r-OPS prop.}.

The hypergeometric series \eqref{generalised hypergeometric series} is a solution of the differential equation
\begin{equation}
\label{generalised hypergeometric differential equation}
\left[\prod_{j=1}^{q}\left(z\,\DiffOp[z]+b_j\right)\,\DiffOp[z]-\prod_{i=1}^{p}\left(z\,\DiffOp[z]+a_i\right)\right]F(z)=0.
\end{equation}	

The derivative of a hypergeometric series is equal to a shift in the parameters up to multiplication by a constant: %(see \cite[Eq.~16.3.1]{DLMF}):
\begin{equation}
\label{derivative of a generalised hypergeometric series}
\DiffOp[z]\,\Hypergeometric[z]{p}{q}{a_1,\cdots,a_p}{b_1,\cdots,b_q}
=\frac{a_1\cdots a_p}{b_1\cdots b_q}\,\Hypergeometric[z]{p}{q}{a_1+1,\cdots,a_p+1}{b_1+1,\cdots,b_q+1}.
\end{equation}

The hypergeometric series satisfies the confluent relation
%\begin{subequations}
\begin{equation}
\label{confluent relations for generalised hypergeometric series}
\lim_{|\alpha|\to\infty}\Hypergeometric[\frac{z}{\alpha}]{p+1}{q}{a_1,\cdots,a_p,\alpha}{b_1,\cdots,b_q}
=\Hypergeometric[z]{p}{q}{a_1,\cdots,a_p}{b_1,\cdots,b_q}.
%=\lim_{|\beta|\to\infty}\Hypergeometric[\beta z]{p}{q+1}{a_1,\cdots,a_p}{b_1,\cdots,b_q,\beta}.
%\quad\text{if } p\leq q,
\end{equation}
%and
%\begin{equation}
%\label{confluent relation for generalised hypergeometric series p,q+1}
%\lim_{|\beta|\to\infty}\Hypergeometric[\beta z]{p}{q+1}{a_1,\cdots,a_p}{b_1,\cdots,b_q,\beta}
%=\Hypergeometric[z]{p}{q}{a_1,\cdots,a_p}{b_1,\cdots,b_q}.
%%\quad\text{if } p\leq q+1,
%\end{equation}

\subsection{Multiple orthogonal polynomials}
\label{Background MOP}
Multiple orthogonal polynomials (see \cite[Ch.~23]{IsmailBook}, \cite[Ch.~4]{NikishinSorokinBook}, and \cite[\S 3]{WalterSurveyAIMS2018}) are a generalization of conventional orthogonal polynomials \cite{ChiharaBook,IsmailBook,SzegoBook} in which the polynomials satisfy orthogonality relations with respect to several measures rather than just one.
We revisit and algebraise the theory of multiple orthogonal polynomials, defining the orthogonality conditions via linear functionals in the ring $R[x]$ of polynomials in one variable with coefficients in a commutative ring $R$.

A linear functional $u$ defined on $R[x]$ is a linear map $u:R[x]\to R$.
The dual space of $R[x]$ is the vector space consisting of all linear functionals on $R[x]$.
The action of a linear functional $u$ on a polynomial $f\in R[x]$ is denoted by $\Functional{u}{f}$.
The moment of order $n\in\N$ of a linear functional $u$ is equal to $\Functional{u}{x^n}$.
By linearity, every linear functional $u$ is uniquely determined by its moments or, alternatively, by the values of $u$ at the elements of any basis of $R[x]$.

There are two types of multiple orthogonal polynomials: type I and type II. 
Both type I and type II polynomials satisfy orthogonality conditions with respect to a vector of $d$ linear functionals for a positive integer $d$ and depend on a multi-index $\n=\left(n_0,\cdots,n_{d-1}\right)\in\N^d$ of norm $|\n|=n_0+\cdots+n_{d-1}$, and both reduce to standard orthogonal polynomials when the number of orthogonality functionals, $d$, is $1$.

The \textit{type I multiple orthogonal polynomials} for $\n=\left(n_0,\cdots,n_{d-1}\right)\in\N^d$ with respect to a vector of $d$ linear functionals $\left(u_0,\cdots,u_{d-1}\right)$ are given by a vector of $r$ polynomials $\left(A_{\n,0},\cdots,A_{\n,d-1}\right)$, with $\deg A_{\n,j}\leq n_j-1$ for each $j\in\{0,\cdots,d-1\}$, satisfying the orthogonality and normalisation conditions
\begin{equation}
\label{type I orthogonality conditions}
\sum_{j=0}^{d-1}\Functional{u_j}{x^kA_{\n,j}}=
\begin{cases}
0 & \text{ if } 0\leq k\leq|\n|-2,\\
1 & \text{ if } k=|\n|-1.
\end{cases}
\end{equation}

The \textit{type II multiple orthogonal polynomial} with respect to a vector of $d$ linear functionals $\left(u_0,\cdots,u_{d-1}\right)$ for $\n=\left(n_0,\cdots,n_{d-1}\right)\in\N^d$ consists of a monic polynomial $P_{\n}$ of degree $|\n|=n_0+\cdots+n_{d-1}$, the norm of $\n$, which satisfies, for each $j\in\{0,\cdots,d-1\}$, the orthogonality conditions
\begin{equation}
\label{type II orthogonality conditions}
\Functional{u_j}{x^kP_{\n}}=0 \quad\text{ if }k\in\{0,\cdots,n_j-1\}.
\end{equation}

Here we focus only on type II multiple orthogonal polynomials.

The orthogonality conditions for multiple (and standard) orthogonal polynomials are usually defined with respect to vectors of measures on the real line or the complex plane instead of linear functionals.
To obtain the type II orthogonality conditions with respect to vectors of measures instead of linear functionals on a commutative ring $R$, consider $R\in\{\R,\C\}$ and linear functionals defined by the integrals with respect to the orthogonality measures, that is, consider linear functionals defined via a measure $\mu$ by
\begin{equation}
\Functional{u}{f}=\int f(x)\mathrm{d}\mu(x)
\quad\text{for any }f\in R[x].
\end{equation}

We are interested in the type II multiple orthogonal polynomials for multi-indices on the so-called \textit{step-line}. 
A multi-index $\left(n_0,\cdots,n_{d-1}\right)\in\N^r$ is on the \textit{step-line} if $n_0\geq n_1\geq\cdots\geq n_{d-1}\geq n_0-1$.
For a fixed $d\in\Z^+$, there is a unique multi-index of norm $n$ on the step-line of $\N^d$, for each $n\in\N$.
Therefore, the type II multiple orthogonal polynomials on the step-line form a sequence with exactly one polynomial of degree $n$, for each $n\in\N$, equal to the length of the corresponding multi-index, and we can replace the multi-index by its length without any ambiguity.
The type II multiple orthogonal polynomials on the step-line of $\N^d$ are often referred to as \textit{$d$-orthogonal polynomials}, as introduced in \cite{MaroniOrthogonalite}, where $d$ is the number of orthogonality functionals.
This means that a polynomial sequence $\seq{P_n(x)}$ is \textit{$d$-orthogonal} with respect to $\left(u_0,\cdots,u_{d-1}\right)$ if $P_n(x)$ is the type II multiple orthogonal polynomial for the multi-index on the step-line  of $\N^d$ with norm $n$.
Applying \eqref{type II orthogonality conditions} to the multi-indices on the step-line, $\seq{P_n(x)}$ satisfies%, for $j\in\{0,\cdots,d-1\}$,
\begin{equation}
\label{d-orthogonality conditions}
\int x^kP_n(x)\mathrm{d}\mu_j(x)
=\begin{cases}
N_n\neq 0 &\text{ if } n=dk+j \\
		0 &\text{ if } n\geq dk+j+1
\end{cases}
\quad\text{ for }j\in\{0,\cdots,d-1\}.
\end{equation}

According to \cite[Th.~2.1]{MaroniOrthogonalite}, a polynomial sequence $\seq{P_n(x)}$ is $d$-orthogonal if and only if it satisfies a $(d+1)$-order recurrence relation of the form
\begin{equation}
\label{recurrence relation for a d-OPS}
P_{n+1}(x)=x\,P_n(x)-\sum_{k=0}^{d}\gamma_{n-k}^{\,[k]}\,P_{n-k}(x).
\end{equation}
The recurrence coefficients in \eqref{recurrence relation for a d-OPS} are collected in the infinite $(d+2)$-banded unit-lower-Hessenberg matrix 
\begin{equation}
\label{Hessenberg matrix d-OP}
\mathrm{H}=\left(\mathrm{h}_{i,j}\right)_{i,j\in\N}=
\begin{bmatrix} 
\gamma^{[0]}_0 & 1  \\
\gamma^{[1]}_0 & \gamma^{[0]}_1 & 1  \\
\vdots & \ddots & \ddots & \ddots \\ 
\gamma^{[d]}_0 & \cdots  & \gamma^{[1]}_{d-1} & \gamma^{[0]}_d & 1  \\
& \gamma^{[d]}_1 & \cdots  & \gamma^{[1]}_d & \gamma^{[0]}_{d+1} & 1  \\
& & \ddots & \ddots & \ddots & \ddots
\end{bmatrix}.
\end{equation}
The $d$-orthogonal polynomials $\seq{P_n(x)}$ are the characteristic polynomials of the truncated finite matrices formed by the first $n$ rows and columns of $\mathrm{H}$ (see \cite[\S 2.2]{WalterVAJ.CoussementGaussianQuadMOP}).
We revisit this property in Proposition \ref{PolySeqAndHessMatrixProp}.

%\newpage
\subsection{Lattice paths and branched continued fractions}
The branched continued fractions appearing in this paper (other types of branched continued fractions exist in the literature) were introduced in \cite{AlanSokalM.PetroelleB.Zhu-LPandBCF1} and were further explored in \cite{AlanSokalM.Petroelle-LPandBCF2LahPolyFunct}. 
We follow their definitions.

For a positive integer $m$, a \textit{$m$-Dyck path} is a path in the lattice $\N\times\N$, starting and ending on the horizontal axis, using steps $(1,1)$, called \textit{rises}, and $(1,-m)$, called \textit{$m$-falls}
(see Fig.\ref{2-Dyck path fig.} for an example of a $2$-Dyck path).
When $m=1$, the $1$-Dyck paths are simply known as Dyck paths.

More generally, we consider \textit{$m$-Dyck paths at level $k$}, which are paths in $\N\times\N_{\geq k}$ using steps $(1,1)$ and $(1,-m)$ and starting and ending at height $k$, and \textit{partial $m$-Dyck paths}, which are paths in $\N\times\N$ using steps $(1,1)$ and $(1,-m)$ and allowed to start and end anywhere in $\N\times\N$.
\begin{figure}[ht]
\centering
\begin{tikzpicture}[scale=0.5]
\draw[->,color=black] (0,0) -- (0,5);
\draw[->,color=black] (0,0) -- (13,0);
\filldraw[black] (0,0) circle (4pt);
\draw[-,color=black] (0,0) -- (1,1);
\filldraw[black] (1,1) circle (4pt);
\draw[-,color=black] (1,1) -- (2,2);
\filldraw[black] (2,2) circle (4pt);
\draw[-,color=black] (2,2) -- (3,3);
\filldraw[black] (3,3) circle (4pt);
\draw[-,color=black] (3,3) -- (4,4);
\filldraw[black] (4,4) circle (4pt);
\draw[-,color=black] (4,4) -- (5,2);
\filldraw[black] (5,2) circle (4pt);
\draw[-,color=black] (5,2) -- (6,3);
\filldraw[black] (6,3) circle (4pt);
\draw[-,color=black] (6,3) -- (7,1);
\filldraw[black] (7,1) circle (4pt);
\draw[-,color=black] (7,1) -- (8,2);
\filldraw[black] (8,2) circle (4pt);
\draw[-,color=black] (8,2) -- (9,0);
\filldraw[black] (9,0) circle (4pt);
\draw[-,color=black] (9,0) -- (10,1);
\filldraw[black] (10,1) circle (4pt);
\draw[-,color=black] (10,1) -- (11,2);
\filldraw[black] (11,2) circle (4pt);
\draw[-,color=black] (11,2) -- (12,0);
\filldraw[black] (12,0) circle (4pt);
\end{tikzpicture}
\caption{A $2$-Dyck path of length $12$.}
\label{2-Dyck path fig.}
\end{figure}
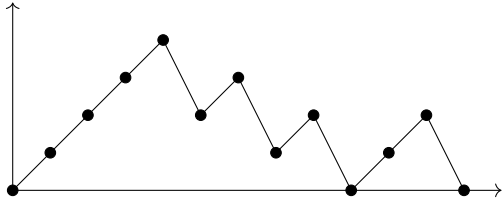

Observe that the length of a $m$-Dyck path, as well as a $m$-Dyck path at level $k$, is always a multiple of $m+1$. %, because there are $m$ rises for each $m$-fall. 
For an infinite sequence of indeterminates $\dis\mathbf{\alpha}=\seq[i\in\N]{\alpha_{i+m}}$, the \textit{$m$-Stieltjes-Rogers polynomial} $\mStieltjesRogersPoly{n}{\mathbf{\alpha}}$, with $n\in\N$, is the generating polynomial for $m$-Dyck paths of length $(m+1)n$, with each rise having weight $1$ and each $m$-fall from height $i$ having weight $\alpha_i$. 
They are an extension of the Stieltjes-Rogers polynomials introduced by Flajolet in \cite{FlajoletContinuedFractions}, which correspond to the case $m=1$.
%Clearly $\dis S_n^{(m)}\left(\mathbf{\alpha}\right)$ is a homogeneous polynomial of degree $n$ with nonnegative integer coefficients.

Let $f_0(t)$ be the generating function for $m$-Dyck paths with the weights specified above, considered as a formal power series in $t$, that is,
\begin{equation}
\label{generating function for m-Dyck paths}
f_0(t)=\sum_{n=0}^{\infty}S_n^{(m)}\left(\mathbf{\alpha}\right)t^n.
\end{equation}
More generally, let $f_k(t)$ be the generating function for $m$-Dyck paths at level $k$ with the same weights.
Observe that $f_k(t)$ is $f_0(t)$ with each $\alpha_i$ replaced by $\alpha_{i+k}$. % that is,
%\begin{equation}
%\label{generating function for m-Dyck paths at level k}
%f_k(t)=\sum_{n=0}^{\infty}S_n^{(m)}\left(\seq[i\in\N]{\alpha_{i+m+k}}\right)t^n.
%\end{equation}
%Observe that we can split a $m$-Dyck path at level $k$ of nonzero length at its last visit to level $k$ and then we further split the remaining part of the path at its last return to height $k+1$, then at its last return to height $k+2$, and so on and rewrite the path in the form $\mathcal{P}_0\,\mathrm{U}\mathcal{P}_1\,\mathrm{U}\cdots\mathcal{P}_m\,\mathrm{D}$ where each $\mathcal{P}_j$ is an arbitrary Dyck path at level $k+j$ for  $j\in\{0,\cdots,m\}$, $\mathrm{U}$ is a rise, and $\mathrm{D}$ is a $m$-fall.
%Hence, we derive the functional equations
The sequence $\seq[k\in\N]{f_k(t)}$ satisfies the functional equations (see \cite[Eqs.~2.26-2.27]{AlanSokalM.PetroelleB.Zhu-LPandBCF1})
\begin{equation}
\label{functional equation for the generating function of m-Dyck paths}
f_k(t)=1+\alpha_{k+m}\,t\prod_{j=0}^{m}f_{k+j}(t)
\quad\text{and}\quad 
f_k(t)=\frac{1}{\dis 1-\alpha_{k+m}\,t\prod_{j=1}^{m}f_{k+j}(t)}.
\end{equation}
Successively iterating the former we find that
\begin{equation}
\begin{aligned}%[t]
\label{m-branched continued fraction}
f_k(t)&
=\cfrac{1}{\dis 1-\alpha_{k+m}\,t\prod_{i_1=1}^{m}\cfrac{1}{\dis 1-\alpha_{k+m+i_1}\,t\prod_{i_2=1}^{m}\cfrac{1}{\dis 1-\alpha_{k+m+i_1+i_2}\,t\prod_{i_3=1}^{m}\cfrac{1}{\dis 1-\cdots}}}}
\\&=
\resizebox{.8\hsize}{!}{$\cfrac{1}{\dis 1-\cfrac{\alpha_{k+m}\,t\vspace*{-0,1 cm}}
{\left(1-\cfrac{\alpha_{k+m+1}\,t} {\left(1-\cfrac{\alpha_{k+m+2}\,t}{\left(\cdots\right)\cdots\left(\cdots\right)}\right)\cdots \left(1-\cfrac{\alpha_{k+2m+1}\,t}{\left(\cdots\right)\cdots\left(\cdots\right)}\right)}\right)
\cdots
\left(1-\cfrac{\alpha_{k+2m}\,t} {\left(1-\cfrac{\alpha_{k+2m+1}\,t}{\left(\cdots\right)\cdots\left(\cdots\right)}\right)\cdots \left(1-\cfrac{\alpha_{k+3m}\,t}{\left(\cdots\right)\cdots\left(\cdots\right)}\right)}\right)}}$}\,.
\end{aligned}
\end{equation}

In particular, a representation for $f_0(t)$ is obtained by taking $k=0$ in \eqref{m-branched continued fraction}.
We call the right-hand side of \eqref{m-branched continued fraction} a \textit{Stieltjes-type $m$-branched continued fraction}, or a \textit{$m$-branched S-fraction} for short.

When $m=1$, \eqref{m-branched continued fraction} reduces to a classical Stieltjes continued fraction or S-fraction
\begin{equation}
\label{S-fraction}
f_k(t)=\cfrac{1}{\dis 1-\cfrac{\alpha_{k+1}\,t}{\dis 1-\cfrac{\alpha_{k+2}\,t}{\dis 1-\cdots}}}.
\end{equation}
The representation for the generating function of the classical Dyck paths with each rise having weight $1$ and each $m$-fall from height $i$ having weight $\alpha_i$ as the continued fraction \eqref{S-fraction} was shown in \cite{FlajoletContinuedFractions}.

We are interested in generalisations of the $m$-Stieltjes-Rogers polynomials.
Observe that every vertex $(x,y)$ of a partial $m$-Dyck path starting at $(0,0)$, and in particular its final vertex, satisfies $x\equiv y \hspace*{-0,1 cm}\mod (m+1)$. 
For an infinite sequence of indeterminates $\dis\mathbf{\alpha}=(\alpha_i)_{i\geq m}$ and $n,k,j\in\N$, the \textit{generalised $m$-Stieltjes-Rogers polynomials of type $j$}, $\generalisedStieltjesRogersPolyTypeJ{n}{k}{j}{\mathbf{\alpha}}$, are defined in \cite{AlanSokalBishalDebM.PetroelleB.ZhuBCF3Laguerre} as the generating polynomials for partial $m$-Dyck paths from $(0,0)$ to $((m+1)n+j,(m+1)k+j)$ in which each rise gets weight $1$ and each $m$-fall from height $i$ gets weight $\alpha_i$.
The fundamental generalised $m$-Stieltjes-Rogers polynomials are the ones of type $0\leq j\leq m$, because we can reduce the generalised $m$-Stieltjes-Rogers polynomials of higher type $j$ to the fundamental types via the trivial relation $\generalisedStieltjesRogersPolyTypeJ{n}{k}{j+\ell(m+1)}{\mathbf{\alpha}}=\generalisedStieltjesRogersPolyTypeJ{n+\ell}{k+\ell}{j}{\mathbf{\alpha}}$.

Here we are interested in the cases where either $j=0$ or $k=0$.
When $j=0$, we have the \textit{generalised $m$-Stieltjes-Rogers polynomials}, $\generalisedStieltjesRogersPoly{n}{k}{\mathbf{\alpha}}$, generating partial $m$-Dyck paths from $(0,0)$ to $((m+1)n,(m+1)k)$; when $k=0$, we have the \textit{modified $m$-Stieltjes-Rogers polynomials of type $j$}, $\modifiedStieltjesRogersPoly{n}{j}{\mathbf{\alpha}}$,  counting partial $m$-Dyck paths from $(0,0)$ to $((m+1)n+k,k)$.
Based on \cite[Prop.~2.3]{AlanSokalM.PetroelleB.Zhu-LPandBCF1}, we know that the generating function of the modified $m$-Stieltjes-Rogers polynomials of type $j$, for any $j\in\N$, is
\begin{equation}
\label{generating function of modified m-Stieltjes-Rogers polynomials of type k}
f_0(t)\cdots f_j(t)=\sum_{n=0}^{\infty}\modifiedStieltjesRogersPoly{n}{j}{\mathbf{\alpha}}\,t^n.
\end{equation}
We consider the matrices $S^{(m;j)}=\left(\generalisedStieltjesRogersPolyTypeJ{n}{k}{j}{\mathbf{\alpha}}\right)_{n,k\in\N}$ of the generalised $m$-Stieltjes-Rogers polynomials of type $j$, for $j\in\N$.
When $j=0$, we denote this matrix by $\mathrm{S}^{(m)}$, that is, $\mathrm{S}^{(m)}=\left(\generalisedStieltjesRogersPoly{n}{k}{\mathbf{\alpha}}\right)_{n,k\in\N}$ is the matrix of the generalised $m$-Stieltjes-Rogers polynomials.
Moreover, we introduce the matrix of modified $m$-Stieltjes-Rogers polynomials $\hat{\mathrm{S}}^{(m)}=\seq[n,j\in\N]{\modifiedStieltjesRogersPoly{n}{j}{\mathbf{\alpha}}}$.
Note that the $j^{\mathrm{th}}$-column of $\hat{\mathrm{S}}^{(m)}$ is the $0^{\mathrm{th}}$-column of $S^{(m;j)}$.
The classical case $m=1$ of this observation corresponds to \cite[Prop.~9.2]{AlanSokalAlgorithmContFrac}.

It is clear that $\generalisedStieltjesRogersPolyTypeJ{n}{0}{0}{\mathbf{\alpha}}=S_n^{(m)}\left(\mathbf{\alpha}\right)$ for all $n\in\N$, so the $0^{\mathrm{th}}$-columns of the matrices $\mathrm{S}^{(m)}$ and $\hat{\mathrm{S}}^{(m)}$ display the ordinary $m$-Stieltjes-Rogers polynomials.
Furthermore, for any $j\in\N$, $\dis\generalisedStieltjesRogersPolyTypeJ{n}{n}{j}{\mathbf{\alpha}}=1$ for all $n\in\N$ and $\generalisedStieltjesRogersPolyTypeJ{n}{k}{j}{\mathbf{\alpha}}=0$ whenever $k>n$, so the generalised $m$-Stieltjes-Rogers polynomials of type $j$ form a unit-lower-triangular matrix.
In particular, $\mathrm{S}^{(m)}$ is a unit-lower-triangular matrix.
The same is not true for $\hat{\mathrm{S}}^{(m)}$.
For instance, $\modifiedStieltjesRogersPoly{0}{k}{\mathbf{\alpha}}=1$ for all $k\in\N$, because the only partial $m$-Dyck path from $(0,0)$ to $(k,k)$ is formed by $k$ consecutive rises, so all entries in the $0^{\mathrm{th}}$-row of $\hat{\mathrm{S}}^{(m)}$ are equal to $1$.

%\newpage
\subsection{Production matrices}
Let $\Pi=\left(\pi_{i,j}\right)_{i,j\in\N}$ be an infinite matrix with entries in a commutative ring $R$.
If all the powers of $\Pi$ are well-defined, we can define an infinite matrix $A=\left(a_{n,k}\right)_{n,k\in\N}$ by $a_{n,k}=\left(\Pi^n\right)_{0,k}$.
In particular, $a_{0,0}=1$ and $a_{0,k}=0$, if $k\geq 1$.
We call $\Pi$ the \textit{production matrix} and $A$ the \textit{output matrix}.
The method of production matrices was introduced in \cite{ProductionMatrices2005,ProductionMatricesAndRiordanArrays2009}.

Production matrices play a fundamental role in the connection between multiple orthogonal polynomials and branched continued fractions, which we explore in detail in Section \ref{Connection of MOP and BCF}. 
This is a consequence of both the production matrix of the generalised $m$-Stieltjes-Rogers polynomials $\left(S_{n,k}^{(m)}\left(\mathbf{\alpha}\right)\right)_{n,k\in\N}$ (see explicit formulas for this production matrix in \eqref{production matrix for m-Stieltjes-Rogers polynomials}-\eqref{formula for the entries of a production matrix}) and the matrix encoding the recurrence relation of a $m$-orthogonal polynomial sequence being $(m+2)$-banded unit-lower-Hessenberg matrices.

Therefore, we are interested in production matrices that are unit-lower-Hessenberg.
A unit-lower-Hessenberg matrix $\Pi$ is always row-finite (i.e., $\Pi$ has only finitely many nonzero entries in each row), so all the powers of $\Pi$ are well defined and we can construct the output matrix of $\Pi$, which is a unit-lower-triangular matrix.
Conversely, the production matrix of a unit-lower-triangular matrix, which always exists and is unique, is a unit-lower-Hessenberg matrix.

\subsection{Total positivity and oscillation matrices}
We say that a matrix with real entries is \textit{totally positive} and \textit{strictly totally positive} if all its minors are, respectively, nonnegative and positive.
This terminology is the same used in \cite{AlanSokalM.PetroelleB.Zhu-LPandBCF1} and \cite{PinkusTotallyPositiveMatrices}. 
However, we warn the reader that other references on total positivity, including \cite{GantmacherKreinOscillationMatrices} and \cite{FallatJohnsonTotallyNonnegativeMatrices}, use the terms totally nonnegative and totally positive matrices for what we define here as totally positive and strictly totally positive matrices, respectively.

Oscillation matrices are a class of matrices intermediary between totally positive and strictly totally positive matrices. \index{oscillation matrix}
A $(n\times n)$-matrix $A$ with real entries is an \textit{oscillation matrix} if $A$ is totally positive and some power of $A$ is strictly totally positive.
We are interested in oscillation matrices because they share the nice spectral properties of strictly totally positive matrices: according to the Gantmacher-Krein theorem (see \cite[Ths.~II-6,~II-14]{GantmacherKreinOscillationMatrices}), the eigenvalues of an oscillation matrix are all simple, real, and positive, and interlace with the eigenvalues of the submatrices obtained by removing either its first or last column and row.

If we consider matrices with entries in a partially ordered commutative ring $R$, we can still define a matrix to be totally positive if all its minors are nonnegative.
The definition of a partially ordered commutative ring $R$ here is the same as in \cite[\S 9]{AlanSokalM.PetroelleB.Zhu-LPandBCF1}: the nonnegative elements form a subset $P\subset R$ such that $0,1\in P$, $a,b\in P$ implies $a+b,ab\in P$, and $P\cup(-P)=\{0\}$; for $a,b\in R$ we write $a\leq b$ if $b-a\in P$.
Moreover, we say that a matrix with entries in a ring of polynomials $R[\mathbf{x}]$, where $R$ is again a partially ordered commutative ring and $\mathbf{x}$ is a (finite or infinite) set of indeterminates, is \textit{coefficient-wise totally positive} if that matrix is totally positive in $R[\mathbf{x}]$ equipped with the coefficient-wise partial order: a polynomial in $R[\mathbf{x}]$ is nonnegative if all its coefficients are nonnegative.

The theory of production matrices is connected to the study of total positivity because, if $\Pi$ is a totally positive matrix that is either row-finite or column-finite, with entries in a partially ordered commutative ring, then its output matrix is also a totally positive matrix (see \cite[Th.~9.4]{AlanSokalM.PetroelleB.Zhu-LPandBCF1}).

\section{Results about generalised and modified $m$-Stieltjes-Rogers polynomials}
\label{Results about generalised and modified m-S.-R. poly}
In this section, we find a relation between the matrices of generalised and modified $m$-Stieltjes-Rogers polynomials (Proposition \ref{relation between generalised and modified m-Stieltjes-Rogers polynomials prop.}) and show that the matrix relating them is coefficient-wise totally positive (Proposition \ref{total positivity of the matrix Lambda}). 
As a result, we conclude that the matrix of modified $m$-Stieltjes-Rogers polynomials is also coefficient-wise totally positive (Theorem \ref{total positivity of the matrix of modified m-S.R. poly}). 
Furthermore, we give explicit expressions for the entries of the production matrix of the generalised $m$-Stieltjes-Rogers polynomials (Proposition \ref{formula for the entries of a production matrix prop.}).

We start by proving the following two propositions:
\begin{proposition}
\label{relation between generalised and modified m-Stieltjes-Rogers polynomials prop.}
For $m\in\Z^+$, let $\mathrm{S}^{(m)}=\left(\generalisedStieltjesRogersPoly{n}{k}{\mathbf{\alpha}}\right)_{n,k\in\N}$ and $\hat{\mathrm{S}}^{(m)}=\seq[n,j\in\N]{\modifiedStieltjesRogersPoly{n}{j}{\mathbf{\alpha}}}$ be the matrices of generalised and modified $m$-Stieltjes-Rogers polynomials, respectively, with weights $\dis\mathbf{\alpha}=\seq[k\in\N]{\alpha_{k+m}}$.
Then, 
\begin{equation}
\label{relation between generalised and modified m-Stieltjes-Rogers polynomials matrix form}
\hat{\mathrm{S}}^{(m)}=\mathrm{S}^{(m)}\,\Lambda^{(m)}, 
\end{equation}
where $\Lambda^{(m)}=\left(\lambda^{(m)}_{i,j}\right)_{i,j\in\N}$ is the upper-triangular matrix such that $\lambda^{(m)}_{i,j}$ is the generating polynomial of the partial $m$-Dyck paths from $(0,(m+1)i)$ to $(j,j)$. 
Therefore, $\lambda^{(m)}_{0,j}=1$ for any $j\in\N$ and
\begin{equation}
\label{coefficients in the relation between generalised and modified m-Stieltjes-Rogers polynomials}
\lambda^{(m)}_{i,j}
=\sum_{j\geq\ell_1\geq\cdots\geq\ell_i\geq i}\,\prod_{k=1}^{i}\alpha_{km+\ell_k}
=\sum_{\ell_1=i}^{j}\alpha_{m+\ell_1}\sum_{\ell_2=i}^{\ell_1}\alpha_{2m+\ell_2}\cdots\sum_{\ell_i=i}^{\ell_{i-1}}\alpha_{im+\ell_i}
\quad\text{when }1\leq i\leq j.
\end{equation}
\end{proposition}

\begin{proposition}
\label{total positivity of the matrix Lambda}
For any $m\in\Z^+$, the upper-triangular matrix $\Lambda^{(m)}$ defined in Proposition \ref{relation between generalised and modified m-Stieltjes-Rogers polynomials prop.} is totally positive in the polynomial ring $\Z[\alpha]$ equipped with the coefficient-wise partial order.
\end{proposition}

Observe that \eqref{relation between generalised and modified m-Stieltjes-Rogers polynomials matrix form} gives the unique $\mathrm{LU}$-factorisation of $\hat{\mathrm{S}}^{(m)}$ 
%(that is, a decomposition of $\hat{\mathrm{S}}^{(m)}$ as a product $\mathrm{LU}$, where $\mathrm{L}$ and $\mathrm{U}$ are, respectively, a lower- and upper-triangular matrix) 
where $\mathrm{L}=\mathrm{S}^{(m)}$ is a unit-lower-triangular matrix.

Moreover, according to \cite[Th.~9.12]{AlanSokalM.PetroelleB.Zhu-LPandBCF1}, the $j^{\mathrm{th}}$-column of $\hat{\mathrm{S}}^{(m)}$, $\seq{\modifiedStieltjesRogersPoly{n}{j}{\mathbf{\alpha}}}$, is a coefficient-wise Hankel-totally positive sequence (or, equivalently, a Stieltjes moment sequence).
Therefore, Proposition \ref{relation between generalised and modified m-Stieltjes-Rogers polynomials prop.} gives an answer to the remark in \cite[\S 4]{AlanSokalMOPd-opProdMatBCF}.

Note that, when $j=i$, \eqref{coefficients in the relation between generalised and modified m-Stieltjes-Rogers polynomials} reduces to $\dis\lambda^{(m)}_{i,i}=\prod_{k=1}^{i}\alpha_{km+i}$ for any $i\in\N$. 
As a result, if the sequence $\seq{\alpha_{k+m}}$ does not have any zeroes or divisors of zero, all the diagonal entries of $\Lambda^{(m)}$ are different from zero.

The upper-triangular matrix defined in Proposition \ref{relation between generalised and modified m-Stieltjes-Rogers polynomials prop.} is
\begin{equation}
	\Lambda^{(m)}=
	\begin{bmatrix} 
		1 & 1 & 1 & 1 & \cdots \\
		& \alpha_{m+1} & \alpha_{m+1}+\alpha_{m+2} & \alpha_{m+1}+\alpha_{m+2}+\alpha_{m+3} & \cdots \\
		&  & \alpha_{m+2}\alpha_{2m+2} & \alpha_{m+2}\alpha_{2m+2}+\alpha_{m+3}\left(\alpha_{2m+2}+\alpha_{2m+3}\right) & \cdots \\
		&  &  & \alpha_{m+3}\alpha_{2m+3}\alpha_{3m+3}  & \cdots \\
		&  &  &  &  \ddots
	\end{bmatrix}.
\end{equation}
In particular, for the classical case $m=1$,
\begin{equation}
\Lambda^{(1)}=
\begin{bmatrix} 
1 & 1 & 1 & 1 & 1 & \cdots \\
   & \alpha_2 & \alpha_2+\alpha_3 & \alpha_2+\alpha_3+\alpha_4 & \alpha_2+\alpha_3+\alpha_4+\alpha_5 & \cdots \\
   &  & \alpha_3\alpha_4 & \alpha_4\left(\alpha_3+\alpha_4+\alpha_5\right) & \alpha_4\left(\alpha_3+\alpha_4+\alpha_5\right)+\alpha_5\left(\alpha_4+\alpha_5+\alpha_6\right) & \cdots \\
   &  &  & \alpha_4\alpha_5\alpha_6  & \alpha_5\alpha_6\left(\alpha_4+\alpha_5+\alpha_6+\alpha_7\right) & \cdots \\
   &  &  &  & \alpha_5\alpha_6\alpha_7\alpha_8 & \cdots \\
   &  &  &  &  & \ddots
\end{bmatrix}.
\end{equation}

The matrix of generalised $m$-Stieltjes-Rogers polynomials $\mathrm{S}^{(m)}=\left(\generalisedStieltjesRogersPoly{n}{k}{\mathbf{\alpha}}\right)_{n,k\in\N}$ is totally positive in the polynomial ring $\Z[\alpha]$ equipped with the coefficient-wise partial order \cite[Th.~9.8]{AlanSokalM.PetroelleB.Zhu-LPandBCF1}.
Moreover, the product of totally positive matrices (when well-defined) is also totally positive, due to the Cauchy-Binet formula (see, for instance, \cite[\S 1.1]{PinkusTotallyPositiveMatrices}).
Therefore, combining Propositions \ref{relation between generalised and modified m-Stieltjes-Rogers polynomials prop.} and \ref{total positivity of the matrix Lambda}, we obtain the following result.
\begin{theorem}
\label{total positivity of the matrix of modified m-S.R. poly}
For any $m\in\Z^+$, the matrix of modified $m$-Stieltjes-Rogers polynomials $\hat{\mathrm{S}}^{(m)}=\seq[n,j\in\N]{\modifiedStieltjesRogersPoly{n}{j}{\mathbf{\alpha}}}$ is totally positive in the polynomial ring $\Z[\alpha]$ equipped with the coefficient-wise partial order.
\end{theorem}

\begin{proof}[Proof of Proposition \ref{relation between generalised and modified m-Stieltjes-Rogers polynomials prop.}]
Firstly, we prove that \eqref{relation between generalised and modified m-Stieltjes-Rogers polynomials matrix form} holds.
Observe that, when $i>j$, there are no partial $m$-Dyck paths from $(0,(m+1)i)$ to $(j,j)$, because the lowest point you can go to from $(0,(m+1)i)$ in $j$ steps is $(j,(m+1)i-mj)$ and $i>j$ implies that $(m+1)i-mj>j$. 
Therefore, $\lambda^{(m)}_{i,j}\left(\mathbf{\alpha}\right)=0$ for $i>j$, $\Lambda^{(m)}$ is indeed an upper-triangular matrix, and \eqref{relation between generalised and modified m-Stieltjes-Rogers polynomials matrix form} is equivalent to
\begin{equation}
\label{relation between generalised and modified m-Stieltjes-Rogers polynomials}
\modifiedStieltjesRogersPoly{n}{j}{\mathbf{\alpha}}
=\sum_{i=0}^{j}\generalisedStieltjesRogersPoly{n}{i}{\mathbf{\alpha}}\,\lambda^{(m)}_{i,j}\left(\mathbf{\alpha}\right)
\quad\text{for all }n,j\in\N.
\end{equation}
Let $n,j\in\N$.
Each partial $m$-Dyck path from $(0,0)$ to $((m+1)n+j,j)$ contains a unique point $((m+1)n,(m+1)i)$, with $0\leq i\leq j$, and it can be split in two at that point.
Hence, we obtain \eqref{relation between generalised and modified m-Stieltjes-Rogers polynomials}, with $\lambda^{(m)}_{i,j}\left(\mathbf{\alpha}\right)$ equal to the generating polynomial of the partial $m$-Dyck paths from $((m+1)n,(m+1)i)$ to $((m+1)n+j,j)$. 
Furthermore, the weights of the $m$-falls in any partial $m$-Dyck path are invariable with horizontal translations, because they only depend on the height from where the $m$-fall occurs. 
Hence, $\lambda^{(m)}_{i,j}\left(\mathbf{\alpha}\right)$ does not depend on $n$ and it is equal to the generating polynomial of the partial $m$-Dyck paths from $(0,(m+1)i)$ to $(j,j)$.

Now we find the explicit expressions for $\lambda^{(m)}_{i,j}\left(\mathbf{\alpha}\right)$, with $i\leq j$.
Observe that it takes $i$ $m$-falls and $j-i$ rises to go from $(0,(m+1)i)$ to $(j,j)$, with the $k^{th}$ $m$-fall counting from the end occuring from height $km+j-\sigma_k$, where $\sigma_k$ is the number of rises happening after this $m$-fall, so $0\leq\sigma_1\leq\cdots\leq\sigma_i\leq j-i$.
Therefore, it is clear that if $i=0$ then $\lambda^{(m)}_{0,j}\left(\mathbf{\alpha}\right)=1$, while for $i\geq 1$, we define $\ell_k=j-\sigma_k$ for $1\leq k\leq i$ to obtain \eqref{coefficients in the relation between generalised and modified m-Stieltjes-Rogers polynomials}.
\end{proof}

To prove the total positivity of $\Lambda^{(m)}$, we use a similar argument to the total-positivity combinatorial proofs using the Lindst\"orm-Gessel-Viennot lemma in \cite[\S 9.4]{AlanSokalM.PetroelleB.Zhu-LPandBCF1}.
For that purpose, we need to first recall some relevant definitions and results, connecting total positivity with walks on graphs. %related to that lemma.

Let $G=(V,\vec{E})$ be a directed graph with vertex set $V$ and edge set $\vec{E}$, and let $\mathbf{w}=\seq[(i,j)\in\vec{E}]{w_{i,j}}$ be a set of commuting indeterminates associated to the edges of $G$, to which we refer to as the \textit{edge weights}.
For $i,j\in V$, a walk from $i$ to $j$ (of length $n\in\N$) is a sequence $\gamma=\left(\gamma_0,\cdots,\gamma_n\right)\in V^{n+1}$ such that $\gamma_0=i$, $\gamma_n=j$, and $\left(\gamma_k,\gamma_{k+1}\right)\in\vec{E}$ for all $0\leq k\leq n-1$.
The \textit{weight of a walk} $\gamma=\left(\gamma_0,\cdots,\gamma_n\right)$, which we denote by $W(\gamma)$, is the product of its edge weights, that is, $W(\gamma)=w_{\gamma_0,\gamma_1}\cdots w_{\gamma_{n-1},\gamma_n}$, with $W(\gamma)=1$ if $\gamma$ is a walk of length $0$.
We define the \textit{walk matrix} $B=\seq[i,j\in V]{b_{i,j}}$ with entries $b_{i,j}\in\Z[[\mathbf{w}]]$ equal to the sum of the weights of all the walks from $i$ to $j$.

%We consider $(r\times r)$-minors of the walk matrix, corresponding to rows $i_1,\cdots,i_r\in V$ and columns $j_1,\cdots,j_r\in V$, which we refer to, respectively, as the source and sink vertices.
For $r\in\Z^+$, let $\mathbf{i}=\left(i_1,\cdots,i_r\right)$, $\mathbf{j}=\left(j_1,\cdots,j_r\right)\in V^r$ be two ordered $r$-tuples of distinct vertices of $G$.
We say that the pair $\mathbf{(i,j)}$ is \textit{nonpermutable} if the set of vertex-distinct walk systems $\mathbf{\gamma}=\left(\gamma_1,\cdots,\gamma_r\right)$ satisfying $\gamma_k: i_k\to j_{\sigma(k)}$ is empty whenever $\sigma$ is not the identity permutation.
Now let $I$ and $J$ be (not necessarily finite) subsets of $V$, equipped with total orders $<_{\, I}$ and $<_{\, J}$, respectively.
We say that the pair $\left(\left(I,<_{\,I}\right),\left(J,<_{\,J}\right)\right)$ is \textit{fully nonpermutable} if each pair $\mathbf{(i,j)}$ of increasing $r$-tuples $\mathbf{i}=\left(i_1,\cdots,i_r\right)$ in $\left(I,<_{\,I}\right)$ and $\mathbf{j}=\left(j_1,\cdots,j_r\right)$ in $\left(J,<_{\,J}\right)$ is nonpermutable for any $r\geq 1$.
Here we use a fully nonpermutable pair to prove the total positivity of $\Lambda^{(m)}$ via the following lemma.
%\newpage
\begin{lemma}(cf. \cite[Cor.~9.17]{AlanSokalM.PetroelleB.Zhu-LPandBCF1})
\label{Fully nonpermutable pairs and total positivity}
For an acyclic graph $G=\left(V,\vec{E}\right)$, let $\left(I,<_{\, I}\right)$ and $\left(J,<_{\, J}\right)$ be totally ordered subsets of $V$ such that the pair $\left(\left(I,<_{\,I}\right),\left(J,<_{\,J}\right)\right)$ is fully nonpermutable.
Then, the submatrix $B_{I,J}$, with rows and columns ordered accordingly with $<_{\, I}$ and $<_{\, J}$, of the walk matrix of $G$ is totally positive with respect to the coefficient-wise order on $\Z[[\mathbf{w}]]$.
\end{lemma}

\begin{proof}[Proof of Proposition \ref{total positivity of the matrix Lambda}]
Recall that $\lambda^{(m)}_{i,j}$ is the generating polynomial of the partial $m$-Dyck paths from $(0,(m+1)i)$ to $(j,j)$. 
Observe that every vertex $(x,y)$ of a partial $m$-Dyck path starting from $(0,(m+1)i)$ and ending at $(j,j)$ satisfies $y\geq x$ and $x\equiv y\hspace*{-0,1 cm} \mod (m+1)$. 
Therefore, we can define $\lambda^{(m)}_{i,j}\left(\mathbf{\alpha}\right)$ as the generating polynomial of the paths from $(0,(m+1)i)$ to $(j,j)$ in the directed graph $G_m=\left(V_m,\vec{E}_m\right)$ with vertex set
\begin{equation}
%\label{m-Dyck paths vertex set}
V_m=\left\{(x,y)\in\N\times\N: y\geq x \text{ and } x\equiv y \mod(m+1)\right\}
\end{equation}
and edge set
\begin{equation}
%\label{m-Dyck paths edge set}
\vec{E}_m=\left\{\big(\left(x_1,y_1\right),\left(x_2,y_2\right)\big)\in V_m\times V_m: x_2-x_1=1 \text{ and } y_2-y_1\in\{1,-m\}\right\}.
\end{equation}

This means that $\Lambda^{(m)}$ is the submatrix of the walk matrix of $G_m$ corresponding to paths with source vertices $\mathrm{I}=\left\{i_n=(0,(m+1)n):n\in\N\right\}$ and sink vertices $\mathrm{J}=\left\{j_n=(n,n):n\in\N\right\}$, totally ordered by $(0,(m+1)n)<(0,(m+1)n')$ and $(n,n)<(n',n')$ if $n<n'$ (see Fig.\ref{Graph Gm}).
Moreover, observe that $G_m$ is a graph embedded in the plane and the vertices of $I\cup J$ lie on the boundary of $G_m$ in the order ``first $I$ in reverse order, then $J$ in order", so the pair $(\mathrm{I},\mathrm{J})$ is clearly fully nonpermutable.
Therefore, because $G_m$ is acyclic, $\Lambda^{(m)}$ is totally positive as a consequence of Lemma \ref{Fully nonpermutable pairs and total positivity}.
\end{proof}

\begin{figure}[ht]
\centering
\begin{tikzpicture}
\filldraw[black] (0,0) circle (3pt) node[anchor=north]{$i_0=(0,0)=j_0$};
\filldraw[black] (0,1) circle (2pt);
\filldraw[black] (0,2) circle (2pt) node[anchor=east]{$i_1=(0,2)$};
\filldraw[black] (0,3) circle (2pt);
\filldraw[black] (0,4) circle (2pt) node[anchor=east]{$i_2=(0,4)$};
\filldraw[black] (0,5) circle (2pt);
\filldraw[black] (0,6) circle (2pt) node[anchor=east]{$i_3=(0,6)$};
\filldraw[black] (1,1) circle (2pt) node[anchor=west]{$j_1=(1,1)$};
\filldraw[black] (2,2) circle (2pt) node[anchor=west]{$j_2=(2,2)$};
\filldraw[black] (3,3) circle (2pt) node[anchor=west]{$j_3=(3,3)$};
\filldraw[black] (1,2) circle (1pt);
\filldraw[black] (1,3) circle (1pt);
\filldraw[black] (1,4) circle (1pt);
\filldraw[black] (1,5) circle (1pt);
\filldraw[black] (2,3) circle (1pt);
\filldraw[black] (2,4) circle (1pt);
\draw[-,color=black] (0,0)  --  (3,3);
\draw[-,color=black, line width=2pt] (0,2) -- (1,1);
\draw[-,color=black, line width=2pt] (0,4) -- (2,2);
\draw[-,color=black, line width=2pt] (0,6) -- (3,3);
\draw[-,color=black] (0,2) -- (2,4);
\draw[-,color=black] (0,4) -- (1,5);
\draw[-,color=black] (0,3) -- (1,2);
\draw[-,color=black] (0,5) -- (2,3);
\draw[-,color=black] (0,1) -- (2.5,3.5);
\draw[-,color=black] (0,3) -- (1.5,4.5);
\draw[-,color=black] (0,5) -- (0.5,5.5);
\draw[dashed,color=black] (3,3)  --  (3.5,3.5) node[anchor=south]{J}; 
\draw[dashed,color=black] (2.5,3.5)  --  (3,4);
\draw[dashed,color=black] (2,4)  --  (2.5,4.5);
\draw[dashed,color=black] (1.5,4.5)  --  (2,5);
\draw[dashed,color=black] (1,5)  --  (1.5,5.5);
\draw[dashed,color=black] (0.5,5.5)  --  (1,6);
\draw[dashed,color=black] (0,6.25)  --  (0,6.75) node[anchor=south]{I};
\end{tikzpicture}
\caption{Graph $G_m$ and sets of source $\mathrm{I}$ and sink $\mathrm{J}$ vertices for $m=1$}
\label{Graph Gm}
\end{figure}
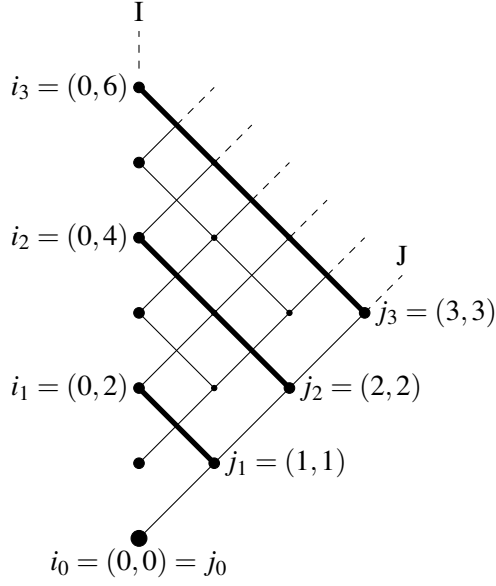

\newpage
Based on \cite[Prop.~8.2]{AlanSokalM.PetroelleB.Zhu-LPandBCF1}, the production matrix of the generalised $m$-Stieltjes-Rogers polynomials $\left(S_{n,k}^{(m)}\left(\mathbf{\alpha}\right)\right)_{n,k\in\N}$ is a ${(m+2)}$-banded unit-lower-Hessenberg matrix admitting a decomposition in $m+1$ bidiagonal matrices:
\begin{equation}
\label{production matrix for m-Stieltjes-Rogers polynomials}
\mathrm{H}^{(m)}=\prod_{i=0}^{m-1}\mathrm{L}\left(\left(\alpha_{k(m+1)+i}\right)_{k\in\Z^+}\right)\cdot\mathrm{U}\left(\left(\alpha_{k(m+1)+m}\right)_{k\in\N}\right),
\end{equation}
where %$\mathrm{L}\left(\left(\alpha_{k(m+1)+i}\right)_{k\in\Z^+}\right)$ and $\mathrm{U}\left(\left(\alpha_{k(m+1)-1}\right)_{k\in\Z^+}\right)$ are, respectively, the lower-bidiagonal and  upper-bidiagonal infinite matrices with entries $L_{k,k}=1$ and $L_{k,k-1}=\alpha_{k(m+1)+i}$, $U_{k,k+1}=1$ and $U_{k,k}=\alpha_{k(m+1)+m}$,for all $k\in\Z^+$.
\begin{itemize}[leftmargin=*]
	\item
	$\mathrm{L}\left(\left(l_k\right)_{k\in\Z^+}\right)$ is the lower-bidiagonal infinite matrix with entries 
	$\mathrm{L}_{k,k}=1$ and $\mathrm{L}_{k+1,k}=l_{k+1}$ for all $k\in\N$, %and
	
	\item
	$\mathrm{U}\left(\left(u_k\right)_{k\in\N}\right)$ is the upper-bidiagonal infinite matrix with entries 
	$\mathrm{U}_{k,k+1}=1$ and $\mathrm{U}_{k,k}=u_k$ for all $k\in\N$.
\end{itemize}
In the following result, we give explicit formulas for the entries of this production matrix.
\begin{proposition}
\label{formula for the entries of a production matrix prop.}
For $m\in\Z^+$ and a sequence of indeterminates $\mathbf{\alpha}=\seq[k\in\N]{\alpha_{k+m}}$, the production matrix of the generalised $m$-Stieltjes-Rogers polynomials $\seq[n_0,n_1\in\N]{\generalisedStieltjesRogersPoly{n_0}{n_1}{\mathbf{\alpha}}}$ is the $(m+2)$-banded unit-lower-Hessenberg matrix $\mathrm{H}^{(m)}=\dis\seq[i,n\in\N]{h^{(m)}_{i,n}\left(\mathbf{\alpha}\right)}$ whose entries are $h^{(m)}_{i,n}\left(\mathbf{\alpha}\right)=0$ if $i\leq n-2$ or $i\geq n+m+1$, $h^{(m)}_{n-1,n}\left(\mathbf{\alpha}\right)=1$ for all $n\geq 1$, and, setting $\alpha_j=0$ for $0\leq j\leq m-1$, 
\begin{equation}
\label{formula for the entries of a production matrix}
h^{(m)}_{n+k,n}\left(\mathbf{\alpha}\right)
=\sum_{m\geq\ell_0>\cdots>\ell_k\geq 0}\,\prod_{j=0}^{k}\alpha_{(m+1)(n+j)+\ell_j}
\quad\text{for all }n\in\N\text{ and }0\leq k\leq m.
\end{equation}
%for all $n\in\N$ and $0\leq k\leq m$, with $\alpha_0=\cdots=\alpha_{m-1}=0$.
\end{proposition}
%Note that the number of summands in \eqref{formula for the entries of a production matrix} is equal to $\binom{m+1}{k+1}$, because each summand corresponds to one way of choosing $k+1$ elements from $\{0,\cdots,m\}$ to be $\ell_0,\cdots,\ell_k$. %(arranged in decreasing order).
%
The cases $m=1$ and $m=2$ of Equations \eqref{production matrix for m-Stieltjes-Rogers polynomials}-\eqref{formula for the entries of a production matrix} are explicitly written in \cite[Eqs.~7.7-7.8]{AlanSokalM.PetroelleB.Zhu-LPandBCF1}.

\begin{proof}
%It is clear from \eqref{production matrix for m-Stieltjes-Rogers polynomials} that $\mathrm{H}^{(m)}\left(\mathbf{\alpha}\right)=\seq[i,n\in\N]{h^{(m)}_{i,n}}$ is a $(m+2)$-banded unit-lower-Hessenberg matrix. 
We prove this result by induction on $m\in\Z^+$. %that the non-trivial entries of $\mathrm{H}^{(m)}\left(\mathbf{\alpha}\right)$ are given by \eqref{formula for the entries of a production matrix}.
	
When $m=1$, \eqref{production matrix for m-Stieltjes-Rogers polynomials} reduces to $\mathrm{H}^{(1)}=\mathrm{L}\left(\left(\alpha_{2k}\right)_{k\in\Z^+}\right)\,\mathrm{U}\left(\left(\alpha_{2k+1}\right)_{k\in\N}\right)$.
As a result,
\begin{equation}
h_{n,n+1}^{(1)}\left(\mathbf{\alpha}\right)=1,
\quad
h_{n,n}^{(1)}\left(\mathbf{\alpha}\right)=\alpha_{2n}+\alpha_{2n+1},
\quad\text{and}\quad
h_{n+1,n}^{(1)}\left(\mathbf{\alpha}\right)=\alpha_{2n+1}\alpha_{2n+2}
\quad\text{for all }n\in\N,
\end{equation}
and all other entries of $\mathrm{H}^{(1)}$ are equal to zero. 
Therefore, Proposition \ref{formula for the entries of a production matrix prop.} holds for $m=1$.
	
Now we suppose that Proposition \ref{formula for the entries of a production matrix prop.} holds for some $m\in\Z^+$ and we show that it also holds for $m+1$.
	
Recalling again \eqref{production matrix for m-Stieltjes-Rogers polynomials},
\begin{equation}
\label{decomposition in bidiagonal matrices m+1}
\mathrm{H}^{(m+1)}
=\prod_{i=0}^{m}\mathrm{L}\left(\left(\alpha_{k(m+2)+i}\right)_{k\in\Z^+}\right)\cdot\mathrm{U}\left(\left(\alpha_{k(m+2)+(m+1)}\right)_{k\in\N}\right)
=\mathrm{L}\left(\left(\alpha_{k(m+2)}\right)_{k\in\Z^+}\right)\,\hat{\mathrm{H}}^{(m)},
\end{equation}
where 
\begin{equation}
\hat{\mathrm{H}}^{(m)}=\prod_{i=0}^{m}\mathrm{L}\left(\left(\alpha_{k(m+2)+i}\right)_{k\in\Z^+}\right)\cdot\mathrm{U}\left(\left(\alpha_{k(m+2)+(m+1)}\right)_{k\in\N}\right).
\end{equation}
Writing $\hat{\mathrm{H}}^{(m)}=\seq[i,n\in\N]{\hat{h}_{i,n}^{(m)}}$, we have
\begin{equation}
\label{entries of H as combination of entries of Hhat}
h_{n+k,n}^{(m+1)}\left(\mathbf{\alpha}\right)
=\hat{h}_{n+k,n}^{(m)}\left(\mathbf{\alpha}\right)+\alpha_{(m+2)(n+k)}\hat{h}_{n+k-1,n}^{(m)}\left(\mathbf{\alpha}\right).
\end{equation}
Using the induction hypothesis, $\hat{\mathrm{H}}^{(m)}$ is a $(m+2)$-banded unit-lower-Hessenberg matrix such that
\begin{equation}
\hat{h}_{n+k,n}^{(m)}\left(\mathbf{\alpha}\right)
=\sum_{m\geq\ell_0>\cdots>\ell_k\geq 0}\,\prod_{j=0}^{k}\alpha_{(m+2)(n+j)+(\ell_j+1)}
\quad\text{for }n\in\N\text{ and }0\leq k\leq m.
\end{equation}
Hence, \eqref{entries of H as combination of entries of Hhat} implies that $h_{i,n}^{(m+1)}\left(\mathbf{\alpha}\right)=0$ if $i\leq n-2$ or $i\geq n+m+2$, $h_{n-1,n}^{(m+1)}\left(\mathbf{\alpha}\right)=1$ for any $n\geq 1$, and
\begin{equation}
h_{n+k,n}^{(m+1)}\left(\mathbf{\alpha}\right)
=\sum_{m\geq\ell_0>\cdots>\ell_k\geq 0}\,\prod_{j=0}^{k}\alpha_{(m+2)(n+j)+(\ell_j+1)}
+\alpha_{(m+2)(n+k)}\sum_{m\geq\ell_0>\cdots>\ell_{k-1}\geq 0}\,\prod_{j=0}^{k-1}\alpha_{(m+2)(n+j)+(\ell_j+1)},
\end{equation}
for any $n\in\N$ and $0\leq k\leq m+1$, with the first sum being an empty sum (thus, equal to $0$) when $k=m+1$.	

Setting $\lambda_j=\ell_j+1$ in both sums and $\lambda_k=0$ in the second sum, we obtain
\begin{equation}
h_{n+k,n}^{(m+1)}\left(\mathbf{\alpha}\right)
=\sum_{m+1\geq\lambda_0>\cdots>\lambda_k\geq 0}\,\prod_{j=0}^{k}\alpha_{(m+2)(n+j)+\lambda_j}.
\end{equation}
Therefore, Proposition \ref{formula for the entries of a production matrix prop.} holds for $m+1$, which concludes our proof.
\end{proof}

%\newpage
\section{Multiple orthogonal polynomials and branched continued fractions}
\label{Connection of MOP and BCF}
The connection between multiple orthogonal polynomials and branched continued fractions was recently introduced and analysed in \cite{AlanSokalMOPd-opProdMatBCF}.
Here we revisit and further explore this connection and its applications in the study of multiple orthogonal polynomials.

Firstly, we define a $m$-orthogonal polynomial sequence such that the moments of its dual sequence are generalised $m$-Stieltjes-Rogers polynomials (Theorem \ref{MOP for gen. m-S.R. poly}), we show that this polynomial sequence is $m$-orthogonal with respect to functionals whose moments are modified $m$-Stieltjes-Rogers polynomials (Proposition \ref{MOP for generalised and modified m-S.R. poly general result}), and we give explicit formulas for the recurrence coefficients of this $m$-orthogonal polynomial sequence (Theorem \ref{recurrence relation for MOP associated with a BCF}).
Then, we assume the positivity of the branched-continued-fraction coefficients and, under that assumption, we show that the orthogonality conditions can be written via positive measures on the positive real line (Corollary \ref{modified m-S.R. poly as moments of positive measures}) and that the zeros of the $m$-orthogonal polynomials are all simple, real, and positive and the zeros of consecutive polynomials interlace (Theorem \ref{location and interlacing of the zeros, general case with positive BCF coeff}).
Finally, we give an upper bound for the zeros of $m$-orthogonal polynomials using the asymptotic behaviour of their recurrence coefficients (Theorem \ref{asymptotic behaviour of the largest zero - general theorem for d-OPS}).

\subsection{Connection via production matrices}
\label{Connection via production matrices}
Here it is useful to introduce matrices representing sequences of linear functionals and polynomial sequences.
Precisely, we call the \textit{moment matrix} of a sequence of linear functionals $\seq[k\in\N]{u_k}$ to the matrix $A=\seq[n,k\in\N]{a_{n,k}}$ such that $a_{n,k}=\Functional{u_k}{x^n}$ and the \textit{coefficient matrix} of a polynomial sequence $\seq{P_n(x)}$ to the matrix $B=\seq[n,k\in\N]{b_{n,k}}$ such that $P_n(x)=\sum\limits_{k=0}^{n}b_{n,k}\,x^k$.
A sequence of linear functionals $\seq[k\in\N]{u_k}$ is the \textit{dual sequence} of a polynomial sequence $\seq{P_n(x)}$ if $\Functional{u_k}{P_n}=\delta_{k,n}$. %=\begin{cases} 1 & \text{if }k=n \\ 0 & \text{if }k\neq n \end{cases}$.
Using their representing matrices, this is equivalent to say that $\seq[k\in\N]{u_k}$ is the dual sequence of $\seq{P_n(x)}$ if and only if the moment matrix of $\seq[k\in\N]{u_k}$ is the inverse of the coefficient matrix of $\seq{P_n(x)}$.

%Observe that the coefficient matrix of a monic polynomial sequence $\seq{P_n(x)}$ is a unit-lower-triangular matrix.
%Therefore, the moment matrix of its dual sequence $\seq[k\in\N]{u_k}$ is also a unit-lower-triangular matrix.
%This is the same to say that $\Functional{u_k}{x^n}=0$ whenever $n<k$ and $\Functional{u_k}{x^k}=1$ for all $k\in\N$.

The theory of production matrices is instrumental to the study of the connection between multiple orthogonal polynomials and branched continued fractions. 
The key result linking polynomial sequences with production matrices is the following.
\begin{proposition}
\label{PolySeqAndHessMatrixProp}
\cite[Prop.~3.2~\&~3.4]{AlanSokalMOPd-opProdMatBCF}
Given a monic polynomial sequence $\seq{P_n(x)}$, there exists an unique unit-lower-Hessenberg matrix $\mathrm{H}=\seq[n,k\in\N]{h_{n,k}}$ such that
\begin{equation}
\label{recurrence relation generic PS}
P_{n+1}(x)=x\,P_n(x)-\sum_{k=0}^{n}h_{n,k}\,P_k(x).
\end{equation}
Conversely, given any unit-lower-Hessenberg matrix $\mathrm{H}=\seq[n,k\in\N]{h_{n,k}}$, the recurrence relation \eqref{recurrence relation generic PS} with the initial condition $P_0(x)=1$ determines an unique polynomial sequence $\seq{P_n(x)}$.
Moreover, for $n\geq 1$, $P_n(x)$ is the characteristic polynomial of the $(n\times n)$-matrix $\mathrm{H}_n$ formed by the first $n$ rows and columns of $\mathrm{H}$, which means that, if we denote the $(n\times n)$-identity matrix by $\mathrm{I}_n$, then 
\begin{equation}
\label{sequence of characteristic polynomials}
P_n(x)=\det\left(x\,\mathrm{I}_n-\mathrm{H}_n\right)
\quad\text{for any }n\geq 1.
\end{equation}
Furthermore, the coefficient matrix of the sequence $\seq{P_n(x)}$ is the inverse of the output matrix of $\mathrm{H}$. 
Therefore, $\mathrm{H}$ is the production matrix of the moment matrix of the dual sequence of $\seq{P_n(x)}$.
\end{proposition}
See the remarks after \cite[Prop.~3.4]{AlanSokalMOPd-opProdMatBCF} to find some references for \eqref{sequence of characteristic polynomials}.
In the case where $\mathrm{H}$ is a $(d+2)$-banded unit-lower-Hessenberg matrix, and consequently $\seq{P_n(x)}$ is a $d$-orthogonal polynomial sequence, \eqref{sequence of characteristic polynomials} was mentioned in the introduction, at the end of Subsection \ref{Background MOP}.

If we consider the previous proposition with the unit-lower-Hessenberg matrix $\mathrm{H}$ being the production matrix of the generalised $m$-Stieltjes-Rogers polynomials $\left(S_{n,k}^{(m)}\left(\boldsymbol{\alpha}\right)\right)_{n,k\in\N}$, we obtain the following result connecting lattice paths and branched continued fractions with multiple orthogonal polynomials.
\begin{theorem}
\label{MOP for gen. m-S.R. poly}
For $m\in\Z^+$ and a sequence $\dis\mathbf{\alpha}=\seq[k\in\N]{\alpha_{k+m}}$ in a commutative ring $R$, let $\mathrm{H}=\seq[n,k\in\N]{h_{n,k}}$ be the $(m+2)$-banded unit-lower-Hessenberg production matrix of the generalised $m$-Stieltjes-Rogers polynomials $\left(S_{n,k}^{(m)}\left(\boldsymbol{\alpha}\right)\right)_{n,k\in\N}$ and $\seq{P_n(x)}$ be the polynomial sequence satisfying the recurrence relation
\begin{equation}
\label{recurrence relation of order m+1 Hessenberg matrix}
P_{n+1}(x)=x\,P_n(x)-\sum_{j=0}^{\min(n,m)}h_{n,n-j}\,P_{n-j}(x),
\end{equation}
with the initial condition $P_0(x)=1$.
Then, the dual sequence of $\seq{P_n(x)}$ is $\seq[k\in\N]{u_k}$ defined by 
\begin{equation}
\label{generalised m-S.R. poly as moments}
\Functional{u_k}{x^n}=S_{n,k}^{(m)}\left(\boldsymbol{\alpha}\right)
\quad\text{for all }k,n\in\N,
\end{equation}
and,  if $h_{n+m,n}\neq 0$ for all $n\in\N$, $\seq{P_n(x)}$ is $m$-orthogonal with respect to $\left(u_0,\cdots,u_{m-1}\right)$.
\end{theorem}
\begin{proof}
Because $\mathrm{H}$ is a $(m+2)$-banded unit-lower-Hessenberg matrix, \eqref{recurrence relation generic PS} reduces to \eqref{recurrence relation of order m+1 Hessenberg matrix}.
As a result, the polynomial sequence $\seq{P_n(x)}$ is $m$-orthogonal with respect to the first $m$ elements of its dual sequence.
Furthermore, due to Proposition \ref{PolySeqAndHessMatrixProp}, we know that, because $\mathrm{H}$ is the production matrix of $\left(S_{n,k}^{(m)}\left(\boldsymbol{\alpha}\right)\right)_{n,k\in\N}$, the dual sequence $\seq[k\in\N]{u_k}$ of $\seq{P_n(x)}$ satisfies \eqref{generalised m-S.R. poly as moments}.
\end{proof}

We know from the latter theorem that the production matrix of the generalised $m$-Stieltjes-Rogers polynomials $\left(S_{n,k}^{(m)}\left(\boldsymbol{\alpha}\right)\right)_{n,k\in\N}$ determines a $m$-orthogonal polynomial sequence with respect to linear functionals whose moments are the same generalised $m$-Stieltjes-Rogers polynomials. 
The following result gives a relation between the sequences of linear functionals whose moment matrices are formed by the generalised and modified $m$-Stieltjes-Rogers polynomials as a corollary of Proposition \ref{relation between generalised and modified m-Stieltjes-Rogers polynomials prop.}. 
%which implies the type II multiple orthogonal polynomials on the step-line with respect to the first elements of those two sequences of linear functionals coincide.
\begin{lemma}
\label{relation between the linear functionals whose moments are generalised and modified m-S.-R. poly}
	
For $m\in\Z^+$ and a sequence $\dis\mathbf{\alpha}=\seq[k\in\N]{\alpha_{k+m}}$ in a commutative ring $R$, let $\Lambda=\left(\lambda_{i,j}\left(\mathbf{\alpha}\right)\right)_{i,j\in\N}$ be the upper-triangular matrix defined in Proposition \ref{relation between generalised and modified m-Stieltjes-Rogers polynomials prop.} and let $\seq[k\in\N]{u_k}$ and $\seq[k\in\N]{v_k}$ be the sequences of linear functionals defined by 
\begin{equation}
\label{functionals with moments equal to m-S.R. poly}
\Functional{u_k}{x^n}=\generalisedStieltjesRogersPoly{n}{k}{\mathbf{\alpha}}
\quad\text{and}\quad
\Functional{v_k}{x^n}=\modifiedStieltjesRogersPoly{n}{k}{\mathbf{\alpha}}
\quad\text{for all } n,k\in\N.
\end{equation}
Then, the sequences $\seq[k\in\N]{u_k}$ and $\seq[k\in\N]{v_k}$ are related by
\begin{equation}
v_k=\sum_{i=0}^{k}\lambda_{i,k}\left(\mathbf{\alpha}\right)\,u_i
\quad\text{for all }k\in\N.
\end{equation}
\end{lemma}
Recall that if the sequence $\dis\mathbf{\alpha}=\seq[k\in\N]{\alpha_{k+m}}$ has no zeroes or divisors of zero, then $\lambda_{i,i}\left(\mathbf{\alpha}\right)\neq 0$ for all $i\in\N$.
Therefore, the type II multiple orthogonal polynomials on the step-line with respect to the first elements of the sequences of linear functionals $\seq[k\in\N]{u_k}$ and $\seq[k\in\N]{v_k}$ coincide, as explained in the following result.
\begin{proposition}
\label{MOP for generalised and modified m-S.R. poly general result}
For $m\in\Z^+$ and a sequence $\dis\mathbf{\alpha}=\seq[k\in\N]{\alpha_{k+m}}$ without any zeroes or divisors of zero in a commutative ring $R$, let $\seq[k\in\N]{u_k}$ and $\seq[k\in\N]{v_k}$ be the sequences of linear functionals defined by \eqref{functionals with moments equal to m-S.R. poly}.
Then, for any $d\in\Z^+$, a polynomial sequence is $d$-orthogonal with respect to $\left(v_0,\cdots,v_{d-1}\right)$ if and only if it is $d$-orthogonal with respect to $\left(u_0,\cdots,u_{d-1}\right)$.
In particular, the $m$-orthogonal polynomial sequence $\seq{P_n(x)}$ with respect to $\left(v_0,\cdots,v_{m-1}\right)$ exists and it satisfies the recurrence relation \eqref{recurrence relation of order m+1 Hessenberg matrix}, where $\mathrm{H}=\seq[n,k\in\N]{h_{n,k}}$ is the production matrix of the generalised $m$-Stieltjes-Rogers polynomials $\left(S_{n,k}^{(m)}\left(\boldsymbol{\alpha}\right)\right)_{n,k\in\N}$.
\end{proposition}

We know from Theorem \ref{MOP for gen. m-S.R. poly} that the coefficients of the recurrence relation satisfied by the $m$-orthogonal polynomial sequence $\seq{P_n(x)}$ are the nontrivial entries of the production matrix of the corresponding generalised $m$-Stieltjes-Rogers polynomials.
In the following result we give explicit expressions for the recurrence coefficients as a consequence of Proposition \ref{formula for the entries of a production matrix prop.}.

%\newpage
\begin{theorem}
\label{recurrence relation for MOP associated with a BCF}
For $m\in\Z^+$ and a sequence $\dis\mathbf{\alpha}=\seq[k\in\N]{\alpha_{k+m}}$ without any zeroes or divisors of zero in a commutative ring $R$, let $\seq{P_n(x)}$ be the monic $m$-orthogonal polynomial sequence with respect to the vectors of linear functionals $\left(u_0,\cdots,u_{m-1}\right)$ and $\left(v_0,\cdots,v_{m-1}\right)$ such that
\begin{equation}
\label{moments as generalised and modified m-S.-R. poly}
\Functional{u_k}{x^n}=\generalisedStieltjesRogersPoly{n}{k}{\mathbf{\alpha}}
\quad\text{and}\quad
\Functional{v_k}{x^n}=\modifiedStieltjesRogersPoly{n}{k}{\mathbf{\alpha}}
\quad\text{for all } n\in\N \text{  and  }0\leq k\leq m-1.
\end{equation}
Then, $\seq{P_n(x)}$ satisfies the recurrence relation
\begin{equation}
\label{recurrence relation m-OP}
P_{n+1}(x)=x\,P_n(x)-\sum_{k=0}^{\min(m,n)}\gamma_{n-k}^{\,[k]}\,P_{n-k}(x),
\end{equation}
with initial condition $P_0(x)=1$ and coefficients
\begin{equation}
\label{recurrence coefficients as a combination of BCF coefficients}
\gamma_n^{\,[k]}=\sum_{m\geq\ell_0>\cdots>\ell_k\geq 0}\,\prod_{j=0}^{k}\alpha_{(m+1)(n+j)+\ell_j}
\quad\text{for any }n\in\N\text{ and }0\leq k\leq m,
\end{equation}
where $\alpha_j=0$ for $0\leq j\leq m-1$.
\end{theorem}
Note that the number of summands in \eqref{recurrence coefficients as a combination of BCF coefficients} is equal to $\binom{m+1}{k+1}$, because each summand corresponds to one way of choosing $k+1$ elements from $\{0,\cdots,m\}$ to be $\ell_0,\cdots,\ell_k$. %(arranged in decreasing order).

Because the sequence $\seq{\alpha_{k+m}}$ does not have any zeroes or divisors of zero,
\begin{equation}
\gamma_n^{\,[m]}=\prod_{j=0}^{m}\alpha_{(m+1)n+m(j+1)}\neq 0
\quad\text{for all }n\in\N.
\end{equation}
For the classical case $m=1$, Theorem \ref{recurrence relation for MOP associated with a BCF} states that the orthogonal polynomials $\seq{P_n(x)}$ with respect to the linear functional $u$ defined by $\Functional{u}{x^n}=\mStieltjesRogersPoly{n}{\mathbf{\alpha}}$ for all $n\in\N$ satisfy the second-order recurrence relation
\begin{equation}
P_{n+1}(x)=\left(x-\alpha_{2n}-\alpha_{2n+1}\right)P_n(x)-\alpha_{2n+1}\alpha_{2n+2}\,P_{n-1}(x).
\end{equation}
%\begin{equation}
%P_{n+1}(x)=\left(x-\gamma_n^{[0]}\right)P_n(x)-\gamma_n^{[1]}\,P_{n-1}(x)
%\quad\text{with   }
%\gamma_n^{[0]}=\alpha_{2n}+\alpha_{2n+1}
%\text{   and   }
%\gamma_n^{[1]}=\alpha_{2n+1}\alpha_{2n+2}.
%\end{equation}
%with
%\begin{equation}
%\gamma_n^{[0]}=\alpha_{2n}+\alpha_{2n+1}
%\quad\text{and}\quad
%\gamma_n^{[1]}=\alpha_{2n+1}\alpha_{2n+2}.
%\end{equation}
Analogously, for any $m\geq 2$, Theorem \ref{recurrence relation for MOP associated with a BCF} gives an explicit $(m+1)$-order recurrence relation satisfied by the $m$-orthogonal polynomials $\seq{P_n(x)}$ with respect to the functionals whose moments are given in \eqref{moments as generalised and modified m-S.-R. poly}.

\subsection{Positive branched-continued-fraction coefficients, orthogonality measures, and zeros}
\label{Positive BCF coefficients and orthogonality measures}
In the study of multiple orthogonal polynomials, we are usually interested in considering orthogonality measures instead of linear functionals. 
Here we are particularly interested in positive orthogonality measures over $\R^+$.
However, we know that $\seq{\generalisedStieltjesRogersPoly{n}{k}{\alpha}}$ cannot be a moment sequence of a positive measure over $\R^+$ for any $k\geq 1$, because $\generalisedStieltjesRogersPoly{n}{k}{\alpha}=0$ for $n<k$, and, in particular, the moment of order zero is $0$.
Combining Proposition \ref{MOP for generalised and modified m-S.R. poly general result} with \cite[Th.~9.12]{AlanSokalM.PetroelleB.Zhu-LPandBCF1}, we find that the modified $m$-Stieltjes-Rogers polynomials solve this problem when the coefficients $\alpha_i$ are all positive.
\begin{corollary}
\label{modified m-S.R. poly as moments of positive measures}
For $m\in\Z^+$ and a sequence $\dis\mathbf{\alpha}=\seq[k\in\N]{\alpha_{k+m}}$ of positive real numbers, $\seq{\modifiedStieltjesRogersPoly{n}{k}{\mathbf{\alpha}}}$ is the moment sequence of a positive measure $\mu_k$ on $\R^+$ for any $0\leq k\leq m$ (cf. \cite[Th.~9.12]{AlanSokalM.PetroelleB.Zhu-LPandBCF1}).
Furthermore, as a consequence of Proposition \ref{MOP for generalised and modified m-S.R. poly general result}, the $m$-orthogonal polynomials with respect to the vector of measures $\left(\mu_0,\cdots,\mu_{m-1}\right)$ exist and satisfy the recurrence relation \eqref{recurrence relation of order m+1 Hessenberg matrix}.
%such that $\dis\Functional{u_k}{x^n}=\generalisedStieltjesRogersPoly{n}{k}{\mathbf{\alpha}}$ for all $k,n\in\N$.
%Therefore, the vector of measures $\left(\mu_0,\cdots,\mu_{m-1}\right)$ is a solution to the problem in \cite[Remark~in~\S 4]{AlanSokalMOPd-opProdMatBCF}.
\end{corollary}

When the branched-continued-fraction coefficients $\alpha_{k+m}$ are all positive, it is clear from \eqref{recurrence coefficients as a combination of BCF coefficients} that the recurrence coefficients of the $m$-orthogonal polynomial sequence $\seq{P_n(x)}$ are also all positive.
Moreover, the positivity of the branched-continued-fraction coefficients also leads to nice properties about the location of the zeros of $P_n(x)$, as detailed in Theorem \ref{location and interlacing of the zeros, general case with positive BCF coeff}, which is a consequence of the following lemma.
\begin{lemma}
\label{location and interlacing of the zeros, general case with Hessenberg oscillation matrix}
For $m\in\Z^+$, let $\seq{P_n(x)}$ be a $m$-orthogonal polynomial sequence and $\mathrm{H}=\seq[n,k\in\N]{h_{n,k}}$ be the $(m+2)$-banded unit-lower-Hessenberg such that $\seq{P_n(x)}$ satisfies the recurrence relation \eqref{recurrence relation of order m+1 Hessenberg matrix}.
If the matrices $\mathrm{H}_n$ formed by the first $n$ rows and columns of $\mathrm{H}$ are oscillation matrices for all $n\in\Z^+$, then the zeros of $P_n(x)$ are all simple, real, and positive, and the zeros of consecutive polynomials interlace.
\end{lemma}
\begin{proof}
Recalling \eqref{sequence of characteristic polynomials}, $P_n(x)$ is the characteristic polynomial of $\mathrm{H}_n$ for any $n\geq 1$, so the zeros of $P_n(x)$ are the eigenvalues of $\mathrm{H}_n$.
Moreover, due to the Gantmacher-Krein theorem (see \cite[Ths.~II-6~\&~II-14]{GantmacherKreinOscillationMatrices}), the eigenvalues of an oscillation matrix are all simple, real, and positive, and interlace with the eigenvalues of the submatrices obtained by removing either its first or last column and row.
In addition, observe that if we remove the last column and row of $\mathrm{H}_n$, with $n\geq 2$, we obtain $\mathrm{H}_{n-1}$.
Therefore, the lemma holds.
\end{proof}

A bidiagonal matrix is totally positive if and only if all its entries are nonnegative (see \cite[Lemma~9.1]{AlanSokalM.PetroelleB.Zhu-LPandBCF1}). 
Moreover, the product of matrices preserves total positivity.
Therefore, if $\alpha_{k+m}\geq 0$ for all $k\in\N$, the production matrix of the generalised $m$-Stieltjes-Rogers polynomials $\left(S_{n,k}^{(m)}\left(\boldsymbol{\alpha}\right)\right)_{n,k\in\N}$ is totally positive.
Furthermore, based on \cite[Th.~5.2]{PinkusTotallyPositiveMatrices}, a $(n\times n)$-matrix of real numbers is an oscillation matrix if and only if it is totally positive, nonsingular, and all the entries lying in its subdiagonal and its supradiagonal are positive.
Therefore, if $\alpha_{k+m}>0$ for any $k\in\N$, the $(n\times n)$-matrices $\mathrm{H}_n$ formed by its first $n$ rows and columns are oscillation matrices for all $n\geq 1$.
Hence, combining Proposition \ref{MOP for generalised and modified m-S.R. poly general result} and Lemma \ref{location and interlacing of the zeros, general case with Hessenberg oscillation matrix}, we obtain the following result.
%\newpage
\begin{theorem}
\label{location and interlacing of the zeros, general case with positive BCF coeff}
For $m\in\Z^+$ and a sequence of positive real constants $\mathbf{\alpha}=\seq[k\in\N]{\alpha_{k+m}}$, 
%let $\mathrm{H}=\dis\seq[n,k\in\N]{h_{n,k}}$ be the production matrix of the generalised $m$-Stieltjes-Rogers polynomials $\seq[n,k\in\N]{\generalisedStieltjesRogersPoly{n}{k}{\mathbf{\alpha}}}$ and 
let $\seq{P_n(x)}$ be the $m$-orthogonal polynomial sequence with respect to the vector of linear functionals $\left(u_0,\cdots,u_{m-1}\right)$ and to the vector of measures $\left(\mu_0,\cdots,\mu_{m-1}\right)$ supported on a subset of $\R^+$ such that
%satisfying the recurrence relation \eqref{recurrence relation of order m+1 Hessenberg matrix} with the initial condition $P_0=1$.
\begin{equation}
\Functional{u_k}{x^n}=\generalisedStieltjesRogersPoly{n}{k}{\mathbf{\alpha}}
\quad\text{and}\quad
\int x^n\mathrm{d}\mu_k(x)=\modifiedStieltjesRogersPoly{n}{k}{\alpha}
\quad\text{for all } n\in\N \text{  and  }0\leq k\leq m-1.
\end{equation}
Then, the zeros of $P_n(x)$ are all simple, real, and positive, and the zeros of consecutive polynomials interlace.
Furthermore, the coefficients \eqref{recurrence coefficients as a combination of BCF coefficients} of the recurrence relation \eqref{recurrence relation m-OP} are all positive.
\end{theorem}

When investigating the location of the zeros of multiple orthogonal polynomials, we are usually interested in finding an upper bound for them.
We finish this section by deriving an upper bound for the zeros of a $m$-orthogonal polynomial sequence from the asymptotic behaviour of its recurrence coefficients.
\begin{theorem}
\label{asymptotic behaviour of the largest zero - general theorem for d-OPS}
For $m\in\Z^+$, let $\seq{P_n(x)}$ be a $m$-orthogonal polynomial sequence satisfying a recurrence relation of the form \eqref{recurrence relation m-OP} such that $\gamma_n^{{[k]}}\in\R$ for any $n\in\N$ and $0\leq k\leq m$, with $\gamma_n^{[m]}>0$, and suppose there exist real constants $\gamma^{[m]}>0$ and $\gamma^{{[k]}}\geq 0$ for $0\leq k\leq m-1$ and a non-decreasing unbounded positive sequence $\seq{f_n}$ such that%, for any $k\in\{0,\cdots,r\}$ and $n\in\Z^+$,
\begin{equation}
\label{asymptotic behaviour recurrence coefficients r-ops}
\left|\gamma_n^{{[k]}}\right|\leq\gamma^{{[k]}}f_n^{\,k+1}+o\left(f_n^{\,k+1}\right)
\quad\text{as }  n\to+\infty.
\end{equation}
For any $n\in\Z^+$, we denote by $x_n^{(n)}$ the largest zero in absolute value of $P_n(x)$. Then,
\begin{equation}
\label{asymptotic behaviour of the largest zero of a r-OP}
\left|x_n^{(n)}\right|\leq\min_{t\in\R^+}\left(t+\sum_{k=0}^{r}\frac{\gamma^{{[k]}}}{t^k}\right)f_n+o\left(f_n\right)
\quad\text{as } n\to+\infty.
\end{equation}
\end{theorem}
The particular case $r=2$ of the latter theorem corresponds to \cite[Th.~3.5]{PaperTricomiWeights} and the following proof is a generalisation of the proof therein.
\begin{proof}%[Proof of Theorem \ref{asymptotic behaviour of the largest zero - general result for 2-OPS}.]
Let $n\in\Z^+$ and $\mathrm{H}_n$ be the $(m+2)$-banded lower Hessenberg matrix such that $P_n(x)=\det\left(x\mathrm{I}_n-\mathrm{H}_n\right)$. 
Then, each zero of $P_n$ is an eigenvalue of the matrix $\mathrm{H}_n$ and $\dis\left|x_n^{(n)}\right|$ is equal to the spectral radius of $\mathrm{H}_n$, the maximum of the absolute values of the eigenvalues of $\mathrm{H}_n$.
	
Therefore, based on \cite[Cor.~6.1.8]{MatrixAnalysis}, we have
\begin{equation}
\left|x_n^{(n)}\right|\leq\min_{t_0,\cdots,t_{n-1}\in\R^+}\,\max_{i\in\{0,\cdots,n-1\}}\left\{\sum_{j=0}^{n-1}\frac{t_j}{t_i}\left|(\mathrm{H}_n)_{i,j}\right|\right\}.
\end{equation}
Recalling the values of the entries of $\mathrm{H}_n$ from \eqref{Hessenberg matrix d-OP} with $d=m$, the latter implies that
\begin{equation}
\left|x_n^{(n)}\right|\leq\min_{t_0,\cdots,t_{n-1}\in\R^+}\,\max_{i\in\{0,\cdots,n-1\}}\left\{\frac{t_{i+1}}{t_i}+\sum_{k=0}^{m}\left|\gamma_{i-k}^{[k]}\right|\frac{t_{i-k}}{t_i}\right\},
\end{equation}
with $t_j=0$ if $j=n$ or $j<0$.

In particular, we can set $t_j=t^j\prod\limits_{l=1}^{j}f_l>0$ for $0\leq j\leq n-1$ and $t\in\R^+$, to find that
\begin{equation}
\left|x_n^{(n)}\right|\leq\min_{t\in\R^+}\,\max_{i\in\{0,\cdots,n-1\}}\left\{t\,f_{i+1}+\sum_{k=0}^{m}\left(\left|\gamma_{i-k}^{[k]}\right|t^{-k}\prod_{l=0}^{k-1}f_{i-l}^{-1}\right)\right\}.
\end{equation}
Furthermore, recalling \eqref{asymptotic behaviour recurrence coefficients r-ops}, we derive that%, as $n\to+\infty$,
\begin{equation}
\left|x_n^{(n)}\right|\leq\min_{t\in\R^+}\,\max_{i\in\{0,\cdots,n-1\}} \left\{t\,f_{i+1}+\sum_{k=0}^{m}\frac{\gamma^{[k]}f_{i-k}^{k+1}}{t^k\prod_{l=0}^{k-1}f_{i-l}}+o\left(f_{i+1}\right)\right\}.
\end{equation}
Therefore, due to the sequence $\seq{f_n}$ being non-decreasing,
\begin{equation}
\left|x_n^{(n)}\right|\leq\min_{t\in\R^+}\,\max_{i\in\{0,\cdots,n-1\}}\left\{\left(t+\sum_{k=0}^{m}\frac{\gamma^{{[k]}}}{t^k}\right)f_{i+1}+o\left(f_{i+1}\right)\right\},
\end{equation}
and, because $\seq{f_n}$ is unbounded, we can conclude that \eqref{asymptotic behaviour of the largest zero of a r-OP} holds.
\end{proof}

%\newpage
\section{Branched continued fractions for ratios of hypergeometric series}
\label{BCF for ratios of hypergeometric series}

Branched-continued-fraction representations for three types of ratios of contiguous hypergeometric series,
\begin{equation}
\label{ratios of hypergeometric series in PSZ}
\frac{\Hypergeometric[t]{r+1}{s}{a_1,\cdots,a_{r+1}}{b_1,\cdots,b_s}}{\Hypergeometric[t]{r+1}{s}{a_1,\cdots,a_r,a_{r+1}-1}{b_1,\cdots,b_{s-1},b_s-1}},
\;\;
\frac{\Hypergeometric[t]{r+1}{s}{a_1,\cdots,a_{r+1}}{b_1,\cdots,b_s}}{\Hypergeometric[t]{r+1}{s}{a_1,\cdots,a_r,a_{r+1}-1}{b_1,\cdots,b_s}},
\;\;\text{and}\;\;
\frac{\Hypergeometric[t]{r}{s}{a_1,\cdots,a_r}{b_1,\cdots,b_s}}{\Hypergeometric[t]{r}{s}{a_1,\cdots,a_r}{b_1,\cdots,b_{s-1},b_s-1}},
\end{equation} 
were introduced in \cite[\S 14]{AlanSokalM.PetroelleB.Zhu-LPandBCF1}, where they are referred to as, respectively, the first, second, and third ratios of contiguous hypergeometric series.

The main result of this section is Theorem \ref{BCF for ratios of r+1Fs - theorem}, where we present new branched continued fractions that include the first and second ratios of contiguous hypergeometric series in \cite{AlanSokalM.PetroelleB.Zhu-LPandBCF1} as particular cases.
In Proposition \ref{conditions for non-negativity of the coefficients of the BCF for a ratio of r+1Fs, r>=s}, we find necessary and sufficient conditions for non-negativity and positivity of the coefficients of the branched continued fractions in Theorem \ref{BCF for ratios of r+1Fs - theorem} when $r\geq s$ and all the indeterminates are real and positive.
In Corollary \ref{BCF for ratios of r+1Fs - corollary a_(r+1)=1}, we give explicit expressions for the modified $m$-Stieltjes-Rogers polynomials corresponding to the branched continued fractions in Theorem \ref{BCF for ratios of r+1Fs - theorem} when $a_{r+1}=1$; they are ratios of products of Pochhammer symbols up to multiplication by a binomial coefficient.
%In particular, we show that the modified $m$-Stieltjes-Rogers polynomials of type $j$, with $0\leq j\leq m$, are ratios of products of Pochhammer symbols when $a_{r+1}=1$.
At the end of the section, we explain how an extension of the third ratio in \eqref{ratios of hypergeometric series in PSZ} can be obtained as a limiting case of the branched continued fractions introduced in Theorem \ref{BCF for ratios of r+1Fs - theorem}.

Throughout this section, we work on the commutative ring $R=\Q\left(b_1,\cdots,b_s\right)\left[a_1,\cdots,a_{r+1}\right]$ of polynomials in the indeterminates $a_1,\cdots,a_{r+1}$ whose coefficients are rational functions in the indeterminates $b_1,\cdots,b_s$.
To simplify the notation, we denote by $[k]_n$, with $k,n\in\Z$ and $n\geq 1$, the unique element of $\{1,\cdots,n\}$ congruent with $k$ modulo $n$, that is, $[k]_n=[(k-1)\hspace*{-0,2 cm}\mod n]+1$.

\subsection{Construction of the branched continued fractions}
To construct branched continued fractions, we use the Euler-Gauss recurrence method for $m$-S-fractions introduced in \cite{AlanSokalM.PetroelleB.Zhu-LPandBCF1}, which is a generalisation of the Euler-Gauss method for classical S-fractions (see \cite{AlanSokalAlgorithmContFrac}).
\begin{lemma}
(cf. \cite[Prop.~2.3]{AlanSokalM.PetroelleB.Zhu-LPandBCF1})
\label{Euler-Gauss method for m-S-fractions}
For $m\in\Z^+$, let $\seq[i\in\N]{\alpha_{i+m}}$ be a sequence in a commutative ring $R$ and let $\seq[k\in\N]{f_k(t)}$ and $\seq[k \geq -1]{g_k(t)}$ be two sequences of formal power series related by $f_k(t)={g_k(t)}\slash{g_{k-1}(t)}$ for all $k\in\N$.
Then, $\seq[k\in\N]{f_k}$ satisfies the functional relations in \eqref{functional equation for the generating function of m-Dyck paths} if and only if $\seq[k \geq -1]{g_k}$ satisfies the recurrence relation
\begin{equation}
\label{recurrence relation for the g_k general case}
g_k(t)-g_{k-1}(t)=\alpha_{k+m}\,t\,g_{k+m}(t).
\end{equation}
Therefore, if we find $\seq[k\geq -1]{g_k(t)}$ satisfying \eqref{recurrence relation for the g_k general case}, then $f_k(t)={g_k(t)}\slash{g_{k-1}(t)}$ is the generating function for $m$-Dyck paths at level $k$ with weights $\seq[i\in\N]{\alpha_{i+m}}$ for any $k\in\N$, and $f_k(t)$ admits the $m$-branched-continued-fraction representation \eqref{m-branched continued fraction}.
Furthermore, the generating function of the modified $m$-Stieltjes-Rogers polynomials of type $k$, $\modifiedStieltjesRogersPoly{n}{k}{\mathbf{\alpha}}$, is $g_k(t)\slash g_{-1}(t)=f_0(t)\cdots f_k(t)$.
\end{lemma}

To find sequences $\seq[k\geq -1]{g_k(t)}$ of hypergeometric series satisfying relations of the form \eqref{recurrence relation for the g_k general case}, we use the following relations involving contiguous hypergeometric series.
%\newpage
\begin{lemma}
\cite[Lemma~14.1]{AlanSokalM.PetroelleB.Zhu-LPandBCF1}
\label{contiguous relations for hypergeometric functions}
The hypergeometric series ${}_pF_q$ satisfies the following three-term contiguous relations:
\begin{equation}
\label{contiguous relation 1}
\begin{aligned}
&\Hypergeometric[t]{p}{q}{a_1,\cdots,a_p}{b_1,\cdots,b_q}-\Hypergeometric[t]{p}{q}{a_1,\cdots,a_{i-1},a_i-1,a_{i+1},\cdots,a_p}{b_1,\cdots,b_q} 
\\
=&\,\dfrac{\prod\limits_{k=1,\,k\neq i}^{p}a_k}{\prod\limits_{l=1}^{q}b_l} \, t \, \Hypergeometric[t]{p}{q}{a_1+1,\cdots,a_{i-1}+1,a_i,a_{i+1}+1,\cdots,a_p+1}{b_1+1,\cdots,b_q+1},
\end{aligned}
\end{equation}
	
\begin{equation}
\label{contiguous relation 2}
\begin{aligned}
&\Hypergeometric[t]{p}{q}{a_1,\cdots,a_p}{b_1,\cdots,b_q}-\Hypergeometric[t]{p}{q}{a_1,\cdots,a_{i-1},a_i-1,a_{i+1},\cdots,a_p}{b_1,\cdots,b_{j-1},b_j-1,b_{j+1},\cdots,b_q} 
\\
=&\,\dfrac{\left(b_j-a_i\right)\prod\limits_{k=1,\,k\neq i}^{p}a_k}{\left(b_j-1\right)\prod\limits_{l=1}^{q}b_l}\, t \, \Hypergeometric[t]{p}{q}{a_1+1,\cdots,a_{i-1}+1,a_i,a_{i+1}+1,\cdots,a_p+1}{b_1+1,\cdots,b_q+1},
\end{aligned}
\end{equation}

\begin{equation}
\label{contiguous relation 3}
\begin{aligned}
&\Hypergeometric[t]{p}{q}{a_1,\cdots,a_p}{b_1,\cdots,b_q}-\Hypergeometric[t]{p}{q}{a_1,\cdots,a_p}{b_1,\cdots,b_{j-1},b_j-1,b_{j+1},\cdots,b_q} 
\\
=&\,-\dfrac{\prod\limits_{k=1}^{p}a_k}{\left(b_j-1\right)\prod\limits_{l=1}^{q}b_l} \, t \, \Hypergeometric[t]{p}{q}{a_1+1,\cdots,a_p+1}{b_1+1,\cdots,b_q+1}.
\end{aligned}
\end{equation}
	
\end{lemma}

As in \cite{AlanSokalM.PetroelleB.Zhu-LPandBCF1}, we start from the case $r=s=m$. 
In that case, our branched continued fraction coincides with the one introduced in \cite[Th.~14.2]{AlanSokalM.PetroelleB.Zhu-LPandBCF1} for the first ratio in \eqref{ratios of hypergeometric series in PSZ}, which we revisit in the following result.
\begin{theorem}
\label{BCF for ratios of m+1Fm PSZ}
(cf. \cite[Th.~14.2]{AlanSokalM.PetroelleB.Zhu-LPandBCF1})
For $m\geq 1$, let
\begin{equation}
\label{ratios of m+1Fm PSZ}
g_k(t)=\Hypergeometric[t]{m+1}{m}{a_1^{(k)},\cdots,a_{m+1}^{(k)}\vspace*{0,1 cm}}{b_1^{(k)},\cdots,b_m^{(k)}}
\quad\text{for any   }k\geq -1,
\end{equation}
with
\begin{equation}
\label{BCF parameters m+1Fm}
a_i^{(k)}=a_i+\ceil{\frac{k+1-i}{m+1}}
\text{   and   }
b_j^{(k)}=b_j+\ceil{\frac{k+1-j}{m}}.
\end{equation}
Then, the ratios of contiguous hypergeometric series $\seq[k\in\N]{f_k(t)=\dfrac{g_k(t)}{g_{k-1}(t)}}$ admit the $m$-branched-continued-fraction representation \eqref{m-branched continued fraction} with coefficients
\begin{equation}
\label{BCF coeff m+1Fm PSZ}
\alpha_{k+m}=\frac{\left(b'_k-a'_k\right)\prod\limits_{i=1,\,i\neq[k]_{m+1}}^{m+1}a_i^{(k)}}{\left(b'_k-1\right)\prod\limits_{i=1}^{m}b_i^{(k)}}
\quad\text{for any }k\in\N,
\end{equation}
where%, for any $1\leq i\leq m+1$, $1\leq j\leq m$, and $k\in\N$,
\begin{equation}
\label{BCF parameters m+1Fm a'k, b'k}
a'_k=a_{[k]_{m+1}}^{(k)}=a_{[k]_{m+1}}+\ceil{\frac{k}{m+1}}
\quad\text{and}\quad
b'_k=b_{[k]_m}^{(k)}=b_{[k]_m}+\ceil{\frac{k}{m}}.
\end{equation}
\end{theorem}

Note that, %\eqref{BCF parameters m+1Fm} implies
%\begin{equation}
%\label{a'_k and b'_k m+1Fm}
%a'_k%=a_{[k]_{m+1}}+\ceil{\frac{k+1-[k]_{m+1}}{m+1}}
%=a_{[k]_{m+1}}+\ceil{\frac{k}{m+1}}
%\quad\text{and}\quad
%b'_k%=b_{[k]_m}+\ceil{\frac{k+1-[k]_r}{r}}
%=b_{[k]_m}+\ceil{\frac{k}{m}}.
%\quad\text{for any }k\in\N.
%\end{equation}
%Then, 
for any $k\in\N$, $1\leq i\leq m+1$, and $1\leq j\leq m$, the coefficients defined in \eqref{BCF parameters m+1Fm a'k, b'k} satisfy
\begin{equation}
a'_{k+1-i}=a_{[k+1-i]_{m+1}}^{(k-i)}+\ceil{\frac{k+1-i}{m+1}}=a_{[k+1-i]_{m+1}}^{(k)}
\quad\text{ and }\quad
b'_{k+1-j}=b_{[k+1-j]_m}^{(k+1-j)}+\ceil{\frac{k+1-j}{m}}=b_{[k+1-j]_m}^{(k)}.
%\quad\text{for all }1\leq i\leq r.
\end{equation}
Therefore,
\begin{equation}
\label{a,b}
\left\{a_1^{(k)},\cdots,a_{m+1}^{(k)}\right\}=\left\{a'_k,\cdots,a'_{k-m}\right\}
\quad\text{ and }\quad
\left\{b_1^{(k)},\cdots,b_m^{(k)}\right\}=\left\{b'_k,\cdots,b'_{k-m+1}\right\}.
\end{equation}
Moreover, $b'_k-1=b'_{k-m}$ for all $k\in\N$.
%\begin{equation}
%b'_k-1=b_{[k]_m}+\ceil{\frac{k}{m}}-1=b_{[k-m]_m}+\ceil{\frac{k-m}{m}}=b'_{k-m}.
%\end{equation}
%
As a result, we can simplify \eqref{ratios of m+1Fm PSZ} and \eqref{BCF coeff m+1Fm PSZ} to obtain the following alternative version of \cite[Th.~14.2]{AlanSokalM.PetroelleB.Zhu-LPandBCF1}.%, with simpler notation.
%\eqref{BCF coeff m+1Fm} and \eqref{ratios of m+1Fm} can be simplified to
\begin{corollary}
\label{BCF for ratios of m+1Fm modified coeff}
For $m\geq 1$, let %$a'_k$ and $b'_k$ be given by \eqref{a'_k and b'_k m+1Fm} and let
\begin{equation}
\label{ratios of m+1Fm}
g_k(t)=\Hypergeometric[t]{m+1}{m}{a'_k,\cdots,a'_{k-m}\vspace*{0,1 cm}}{b'_k,\cdots,b'_{k-m+1}}
\quad\text{for any   }k\geq -1,
\end{equation}
with $a'_k$ and $b'_k$ defined by \eqref{BCF parameters m+1Fm a'k, b'k}.
Then, the ratios of contiguous hypergeometric series $\seq[k\in\N]{f_k(t)=\dfrac{g_k(t)}{g_{k-1}(t)}}$ admit the $m$-branched-continued-fraction representation \eqref{m-branched continued fraction} with coefficients
\begin{equation}
\label{BCF coeff m+1Fm}
\alpha_{k+m}=\frac{\left(b'_k-a'_k\right)\prod\limits_{i=1}^{m}a'_{k-i}}{\prod\limits_{i=0}^{m}b'_{k-i}}
\quad\text{for any }k\in\N.
\end{equation}
\end{corollary}
Alternatively, we could set $\alpha_m=\dfrac{a_1\cdots a_m}{b_1\cdots b_m}$ and let $\alpha_{k+m}$ be defined by \eqref{BCF coeff m+1Fm PSZ} or \eqref{BCF coeff m+1Fm} for any $k\geq 1$.
This choice of coefficients $\seq[k\in\N]{\alpha_{k+m}}$ gives the branched-continued-fraction for the second ratio of contiguous ${}_{m+1}F_m$-hypergeometric series obtained in \cite[Th.~14.5]{AlanSokalM.PetroelleB.Zhu-LPandBCF1}. 
This is a particular instance of a more generic observation: 
if $\seq[k\geq -1]{g_k(t)}$ is a sequence of functions satisfying the recurrence relation \eqref{recurrence relation for the g_k general case}, then changing the value of $\alpha_m$ changes $g_{-1}(t)$, without changing $g_k(t)$ for $k\in\N$. 
More generally, changing the values of $\alpha_m,\cdots,\alpha_{m+n}$ with $n\in\N$ changes $g_{-1}(t),\cdots,g_{n-1}(t)$, without changing $g_k(t)$ for $k\geq n$.

We focus now on the cases with $r\neq s$.
When $r>s$, the branched continued fractions in \cite{AlanSokalM.PetroelleB.Zhu-LPandBCF1} are obtained from the case $r=s$ replacing $t$ by $b_1\cdots b_{r-s}\,t$, taking $b_1,\cdots,b_{r-s}\to\infty$, and relabelling $b_i\to b_{i-(r-s)}$; when $s>r$ they are obtained from the case $r=s$ replacing $t$ by $\left(a_1\cdots a_{s-r}\right)^{-1}\,t$, taking $a_1,\cdots,a_{s-r}\to\infty$, and relabelling $a_i\to a_{i-(s-r)}$.
Here we generalise this process, also starting from the case $r=s$, but considering that the indeterminates which we take to infinity do not need to be $b_1,\cdots,b_{r-s}$ (if $r>s$) or $a_1,\cdots,a_{s-r}$ (if $s>r$), but can instead be any $r-s$ indeterminates among $b_1,\cdots,b_s$ (if $r>s$) or any $s-r$ indeterminates among $a_1,\cdots,a_r$ (if $s>r$).
For this purpose, when $r>s$, we choose $1\leq\lambda_1<\cdots<\lambda_s\leq r$, 
define $B=\prod\limits_{j\in\mathrm{J}}b_j$ with $\mathrm{J}:=\{1,\cdots,r\}\backslash\{\lambda_1,\cdots,\lambda_s\}\neq\emptyset$,
and construct new branched continued fractions by replacing $t$ by $B\,t$, 
taking $b_j\to\infty$ for all $j\in J$, 
and relabelling $b_{\lambda_j}\to b_j$ for $1\leq j\leq s$.
Similarly, when $r<s$, 
we choose $1\leq\sigma_1<\cdots<\sigma_r\leq s$,
define $A=\prod\limits_{i\in\mathrm{I}}a_i$ with $\mathrm{I}:=\{1,\cdots,s\}\backslash\{\sigma_1,\cdots,\sigma_r\}\neq\emptyset$,
and construct new branched continued fractions by replacing $t$ by $A^{-1}\,t$, 
taking $a_i\to\infty$ for all $i\in I$, 
and relabelling $a_{\sigma_i}\to\hat{a}_i$ for $1\leq i\leq r$.

Taking these limits in the branched continued fractions from Theorem \ref{BCF for ratios of m+1Fm PSZ} and Corollary \ref{BCF for ratios of m+1Fm modified coeff}, we obtain the following result.
%\newpage
\begin{theorem}
\label{BCF for ratios of r+1Fs - theorem}
	
For $r,s\in\N$ such that $m=\max(r,s)\geq 1$, let $1\leq\lambda_1<\cdots<\lambda_s\leq r$ and $\Lambda=\{\lambda_1,\cdots,\lambda_s\}$ if $r\geq s$ or let $1\leq\sigma_1<\cdots<\sigma_r\leq s$, $\sigma_{r+1}=s+1$, and $\Sigma=\{\sigma_1,\cdots,\sigma_r,\,s+1\}$ if $r\leq s$, and define
\begin{equation}
\label{g_k as a r+1Fs hypergeometric function}
g_k(t)=\Hypergeometric[t]{r+1}{s}{a_1^{(k)},\cdots,a_{r+1}^{(k)}\vspace*{0,1 cm}}{b_1^{(k)},\cdots,b_s^{(k)}}
\quad\text{for any}\quad k\geq -1,
\end{equation}
where
\begin{equation}
\label{BCF parameters r+1Fs, r>=s}
a_i^{(k)}=a_i+\ceil{\frac{k+1-i}{r+1}}
\quad\text{and}\quad
b_j^{(k)}=b_j+\ceil{\frac{k+1-\lambda_j}{r}}
\quad\text{if   }r\geq s,
\end{equation}
or
\begin{equation}
\label{BCF parameters r+1Fs, r<=s}
a_i^{(k)}=a_i+\ceil{\frac{k+1-\sigma_i}{s+1}}
\quad\text{and}\quad
b_j^{(k)}=b_j+\ceil{\frac{k+1-j}{s}}
\quad\text{if   }r\leq s.
\end{equation}
Then, the ratios of contiguous hypergeometric series $\seq[k\in\N]{f_k(t)=\dfrac{g_k(t)}{g_{k-1}(t)}}$ \vspace*{0,1 cm}
admit a $m$-branched-continued-fraction representation of the form \eqref{m-branched continued fraction} with coefficients $\seq[k\in\N]{\alpha_{k+m}}$ defined as follows:
\begin{itemize}
\item
if $r\geq s$,
\begin{equation}
\label{BCF coeff r+1Fs, r>=s}
		\alpha_{k+r}=%\alpha_{k+r}^{[(r,s),\left(\lambda_1,\cdots,\lambda_s\right)]}=
		\begin{cases}
			\dfrac{\prod\limits_{i=1,\,i\neq[k]_{r+1}}^{r+1}a_i^{(k)}}{\prod\limits_{j=1}^{s}b_j^{(k)}}
			=\dfrac{\prod\limits_{i=1}^{r}a'_{k-i}}{\prod\limits_{\substack{i=1\\ [k-i]_r\in\Lambda}}^{r-1}b'_{k-i}} 
			& \text{if } [k]_r\not\in\Lambda, \vspace*{0,2 cm} \\
			\dfrac{\left(b_\ell^{(k)}-a_{[k]_{r+1}}^{(k)}\right)\prod\limits_{i=1,\,i\neq[k]_{r+1}}^{r+1}a_i^{(k)}}{\left(b_\ell^{(k)}-1\right)\prod\limits_{j=1}^{s}b_j^{(k)}}
			=\dfrac{\left(b'_k-a'_k\right)\prod\limits_{i=1}^{r}a'_{k-i}}{\prod\limits_{\substack{i=0\\ [k-i]_r\in\Lambda}}^{r}b'_{k-i}}
			& \text{if }\; [k]_r=\lambda_\ell\in\Lambda, %\;\text{ for }\; \ell\in\{1,\cdots,s\},
		\end{cases}
		%\quad\text{for }\nu\in\N,
	\end{equation}
with	%where, for $k\in\Z$, $i\in\{1,\cdots,r+1\}$, and $j\in\{1,\cdots,s\}$,
	\begin{equation}
		\label{BCF parameters r+1Fs, r>=s*}
		a'_k=a_{[k]_{r+1}}^{(k)}=a_{[k]_{r+1}}+\ceil{\frac{k}{r+1}}
		\quad\text{and}\quad
		b'_k=b_l^{(k)}=b_l+\ceil{\frac{k}{r}}
		\text{   if   }
		[k]_r=\lambda_l;
	\end{equation}
	
\item
if $r\leq s$,
	\begin{equation}
		\label{BCF coeff r+1Fs, r<=s}
		\alpha_{k+s}=%\alpha_{k+s}^{[(r,s),\left(\sigma_1,\cdots,\sigma_r\right)]}=
		\begin{cases}
			-\dfrac{\prod\limits_{i=1}^{r+1}a_i^{(k)}}{\left(b_{[k]_s}^{(k)}-1\right)\prod\limits_{j=1}^{s}b_j^{(k)}}
			=-\dfrac{\prod\limits_{\substack{i=1 \\ [k-i]_{s+1}\in\Sigma}}^{s}a'_{k-i}}{\prod\limits_{i=0}^{s}b'_{k-i}}  
			& \text{if } [k]_{s+1}\not\in\Sigma, \vspace*{0,2 cm} \\
			\dfrac{\left(b_{[k]_s}^{(k)}-a_\ell^{(k)}\right)\prod\limits_{i=1,\,i\neq\ell}^{r+1}a_i^{(k)}}{\left(b_{[k]_s}^{(k)}-1\right)\prod\limits_{j=1}^{s}b_j^{(k)}}
			=\dfrac{\left(b'_k-a'_k\right)\prod\limits_{\substack{i=1 \\ [k-i]_{s+1}\in\Sigma}}^{s}a'_{k-i}}{\prod\limits_{i=0}^{s}b'_{k-i}}
			& \text{if }\; [k]_{s+1}=\sigma_\ell\in\Sigma, %\;\text{ for }\; \ell\in\{1,\cdots,r+1\},
		\end{cases}
	\end{equation}
with %where, for $k\in\Z$, $i\in\{1,\cdots,r+1\}$, and $j\in\{1,\cdots,s\}$,
	\begin{equation}
		\label{BCF parameters r+1Fs, r<=s*}
		a'_k=a_l^{(k)}=a_l+\ceil{\frac{k}{s+1}}
		\text{   if   }
		[k]_{s+1}=\sigma_l
		\quad\text{and}\quad
		b'_k=b_{[k]_s}^{(k)}=b_{[k]_s}+\ceil{\frac{k}{s}}.		
	\end{equation}
\end{itemize}
\end{theorem}

When $m=1$, that is, when $(r,s)\in\{(1,1),(1,0),(0,1)\}$, Theorem \ref{BCF for ratios of r+1Fs - theorem} gives well-known classical-continued-fraction representations of the form \eqref{S-fraction} for the ratios of contiguous hypergeometric series
\begin{equation}
	\frac{\Hypergeometric[t]{2}{1}{a_1+\ceil{\frac{k}{2}},a_2+\ceil{\frac{k-1}{2}}}{b+k}}{\Hypergeometric[t]{2}{1}{a_1+\ceil{\frac{k-1}{2}},a_2+\ceil{\frac{k-2}{2}}}{b+k-1}},
	\quad
	\frac{\Hypergeometric[t]{2}{0}{a_1+\ceil{\frac{k}{2}},a_2+\ceil{\frac{k-1}{2}}}{-}}{\Hypergeometric[t]{2}{0}{a_1+\ceil{\frac{k-1}{2}},a_2+\ceil{\frac{k-2}{2}}}{-}},
	\quad\text{and}\quad
	\frac{\Hypergeometric[t]{1}{1}{a+\ceil{\frac{k-1}{2}}}{b+k}}{\Hypergeometric[t]{1}{1}{a+\ceil{\frac{k-2}{2}}}{b+k-1}},
\end{equation}
where the continued-fraction coefficients are, respectively,
\begin{equation}
	\alpha_{2k+1}=\frac{\left(a_1+k\right)\left(b-a_2+k\right)}{\left(b+2k-1\right)\left(b+2k\right)}
	\quad\text{and}\quad
	\alpha_{2k+2}=\frac{\left(a_2+k\right)\left(b-a_1+k\right)}{\left(b+2k\right)\left(b+2k+1\right)}
	\quad\text{for  }k\in\N,
\end{equation}
\begin{equation}
	\alpha_{2k+j}=a_j+k
	\quad\text{for  }k\in\N\text{  and  }j\in\{1,2\},
\end{equation}
and
\begin{equation}
	\alpha_{2k+1}=\frac{(b-a+k)}{(b+2k-1)(b+2k)}
	\quad\text{and}\quad
	\alpha_{2k+2}=-\frac{(a+k)}{(b+2k)(b+2k+1)}
	\quad\text{for  }k\in\N.
\end{equation}
%respectively.

For any $r,s\in\N$ not both equal to $0$, Theorem \ref{BCF for ratios of r+1Fs - theorem} gives $\binom{\max(r,s)}{\min(r,s)}$ distinct branched continued fractions, including as particular cases the representations obtained in \cite[\S 14]{AlanSokalM.PetroelleB.Zhu-LPandBCF1} for the first and second ratios of contiguous ${}_{r+1}F_s$-hypergeometric series in \eqref{ratios of hypergeometric series in PSZ}.
In Section \ref{Type II MOP w.r.t. linear functionals}, we investigate multiple orthogonal polynomials corresponding to each of these branched continued fractions (with $a_{r+1}=1$).

When $r\geq s$, setting $\lambda_j=r-s+j$ for all $1\leq j\leq s$ in Theorem \ref{BCF for ratios of r+1Fs - theorem} gives the branched-continued-fraction for the first ratio of contiguous ${}_{r+1}F_s$-hypergeometric series obtained in \cite[Th.~14.3]{AlanSokalM.PetroelleB.Zhu-LPandBCF1} if $\lambda_s=r$ or the branched-continued-fraction for the first ratio of contiguous ${}_{r+1}F_s$-hypergeometric series obtained in \cite[Th.~14.6]{AlanSokalM.PetroelleB.Zhu-LPandBCF1} if $\lambda_s\leq r$. 
When $r\leq s$, setting $\sigma_i=s-r+i$ for all $1\leq i\leq r$ in Theorem \ref{BCF for ratios of r+1Fs - theorem} gives the branched-continued-fraction for the first ratio of contiguous ${}_{r+1}F_s$-hypergeometric series obtained in \cite[Th.~14.3]{AlanSokalM.PetroelleB.Zhu-LPandBCF1}.

%Furthermore, the functions $\seq[k\geq -1]{g_k(t)}$ involved in the construction of the branched continued fractions for the first and second ratios of contiguous ${}_{r+1}F_s$-hypergeometric series in \eqref{ratios of hypergeometric series in PSZ} are the same for any $k\geq 0$.
%Therefore
Furthermore, we can always choose if we get a branched-continued-fraction for the first or second ratio of contiguous ${}_{r+1}F_s$-hypergeometric series by keeping the values of $\alpha_{k+m}$ for $k\geq 1$ defined in Theorem \ref{BCF for ratios of r+1Fs - theorem} and changing $\alpha_m$ and, consequently, changing $g_{-1}(t)$ without changing $g_k(t)$ for any $k\in\N$.
Precisely, we get the first and second ratios in \eqref{ratios of hypergeometric series in PSZ} if we set, respectively, 
\begin{equation}
	\alpha_m=\dfrac{a_1\cdots a_r\left(b_s-a_{r+1}\right)}{b_1\cdots b_s\left(b_s-1\right)} 
	\quad\text{or}\quad 
	\alpha_m=\dfrac{a_1\cdots a_r}{b_1\cdots b_s}.
\end{equation}
The expressions on the left- and right-hand sides of \eqref{BCF coeff r+1Fs, r>=s} and \eqref{BCF coeff r+1Fs, r<=s} are obtained by taking the limits described before the statement of Theorem \ref{BCF for ratios of r+1Fs - theorem} in \eqref{BCF coeff m+1Fm PSZ} or in \eqref{BCF coeff m+1Fm}, respectively.
Alternatively, the expressions on the right-hand side can be obtained from the expressions on left-hand side analogously to how we derived \eqref{BCF coeff m+1Fm} from \eqref{BCF coeff m+1Fm PSZ}.
We now give an alternative proof of Theorem \ref{BCF for ratios of r+1Fs - theorem}, using the relations in Lemma \ref{contiguous relations for hypergeometric functions}.
%\newpage
\begin{proof}[Second proof of Theorem \ref{BCF for ratios of r+1Fs - theorem}]
It is sufficient to show that the sequence $\seq[k\geq -1]{g_k(t)}$ defined in \eqref{g_k as a r+1Fs hypergeometric function} satisfies the recurrence relation
\begin{equation}
\label{rec rel g_k}
g_k(t)-g_{k-1}(t)=\alpha_{k+m}\,t\,g_{k+m}(t)
\quad\text{for all }k\in\N,
\end{equation}
involving the coefficients $\seq[k\in\N]{\alpha_{k+m}}$ defined by \eqref{BCF coeff r+1Fs, r>=s} if $r\geq s$ or by \eqref{BCF coeff r+1Fs, r<=s} if $r\leq s$.
Then, using Lemma \ref{Euler-Gauss method for m-S-fractions}, $f_k(t)={g_k(t)}\slash{g_{k-1}(t)}$ is the generating function of the $m$-Dyck paths at height $k$ for any $k\in\N$ and admits the $m$-branched-continued-fraction representation \eqref{m-branched continued fraction}.
	
%\newpage
To prove that \eqref{rec rel g_k} holds, we use the relations in Lemma \ref{contiguous relations for hypergeometric functions}.
We start by considering the case $r\geq s$.
Then, for any $k\in\N$,
\begin{itemize}
\item
$a_i^{(k)}=a_i^{(k-1)}+1$ if $i=[k]_{r+1}$ and $a_i^{(k)}=a_i^{(k-1)}$ if $i\neq[k]_{r+1}$,
\vspace*{0,1 cm}
\item
$b_j^{(k)}=b_j^{(k-1)}+1$ if $\lambda_j=[k]_r$ and $b_j^{(k)}=b_j^{(k-1)}$ if $\lambda_j\neq[k]_r$.
\end{itemize}
%Set $i=[k]_{r+1}$.
If $[k]_r\not\in\Lambda=\{\lambda_1,\cdots,\lambda_s\}$, we use \eqref{contiguous relation 1} to find that  \vspace*{-0,2 cm}
\begin{equation}
\begin{aligned}
g_k(t)-g_{k-1}(t) 
=&\,
\Hypergeometric[t]{r+1}{s}{a_1^{(k)},\cdots,a_{r+1}^{(k)}\vspace*{0,1 cm}} {b_1^{(k)},\cdots,b_s^{(k)}}
-\Hypergeometric[t]{r+1}{s}{a_1^{(k)},\cdots,a_{[k]_{r+1}-1}^{(k)},a_{[k]_{r+1}}^{(k)}-1,a_{[k]_{r+1}+1}^{(k)},\cdots,a_{r+1}^{(k)}\vspace*{0,1 cm}} {b_1^{(k)},\cdots,b_s^{(k)}}
\\=&\,
\dfrac{\prod\limits_{i=1,\,i\neq[k]_{r+1}}^{r}a_i^{(k)}}{\prod\limits_{j=1}^{s}b_j^{(k)}} \, \Hypergeometric[t]{r+1}{s}{a_1^{(k)}+1,\cdots,a_{[k]_{r+1}-1}^{(k)}+1,a_{[k]_{r+1}}^{(k)},a_{[k]_{r+1}+1}^{(k)}+1,\cdots,a_{r+1}^{(k)}+1\vspace*{0,1 cm}} {b_1^{(k)}+1,\cdots,b_s^{(k)}+1}
\\=&\,
\alpha_{k+r}\,t\,g_{k+r}(t).
\end{aligned}
\end{equation}
Otherwise, $[k]_r=\lambda_\ell\in\Lambda$ with $\ell\in\{1,\cdots,s\}$, and, using \eqref{contiguous relation 2}, we get \vspace*{-0,2 cm}
\begin{equation}
\begin{aligned}
&\,g_k(t)-g_{k-1}(t) 
=\Hypergeometric[t]{r+1}{s}{a_1^{(k)},\cdots,a_{r+1}^{(k)}\vspace*{0,1 cm}}{b_1^{(k)},\cdots,b_s^{(k)}}
-\Hypergeometric[t]{r+1}{s}{a_1^{(k)},\cdots,a_{[k]_{r+1}-1}^{(k)},a_{[k]_{r+1}}^{(k)}-1,a_{[k]_{r+1}+1}^{(k)},\cdots,a_{r+1}^{(k)}\vspace*{0,1 cm}} {b_1^{(k)},\cdots,b_{\ell-1}^{(k)},b_\ell^{(k)}-1,b_{\ell+1}^{(k)},\cdots,b_s^{(k)}}
\\=&\,
\dfrac{\left(b_\ell^{(k)}-a_{[k]_{r+1}}^{(k)}\right)\prod\limits_{i=1,\,i\neq[k]_{r+1}}^{r}a_i^{(k)}}{\left(b_\ell^{(k)}-1\right)\prod\limits_{j=1}^{s}b_j^{(k)}} \, \Hypergeometric[t]{r+1}{s}{a_1^{(k)}+1,\cdots,a_{[k]_{r+1}-1}^{(k)}+1,a_{[k]_{r+1}}^{(k)},a_{[k]_{r+1}+1}^{(k)}+1,\cdots,a_r^{(k)}+1 \vspace*{0,1 cm}} {b_1^{(k)}+1,\cdots,b_s^{(k)}+1}
\\=&\,
\alpha_{k+r}\,t\,g_{k+r}(t).
\end{aligned}
\end{equation}
	
Next, we consider the case $r\leq s$.
Then, for any $k\in\N$,
\begin{itemize}
\item
$a_i^{(k)}=a_i^{(k-1)}+1$ if $\sigma_i=[k]_{s+1}$ and $a_i^{(k)}=a_i^{(k-1)}$ if $\sigma_i\neq[k]_{s+1}$,
\vspace*{0,1 cm}
\item
$b_j^{(k)}=b_j^{(k-1)}+1$ if $j=[k]_s$ and $b_j^{(k)}=b_j^{(k-1)}$ if $j\neq[k]_s$.
\end{itemize}
%Set $j=[k]_s$.
If $[k]_{s+1}\not\in\Sigma=\{\sigma_1,\cdots,\sigma_r,\,s+1\}$, we use \eqref{contiguous relation 3} to find that %\vspace*{-0,2 cm}
\begin{equation}
\begin{aligned}
g_k(t)-g_{k-1}(t) 
=&\,
\Hypergeometric[t]{r+1}{s}{a_1^{(k)},\cdots,a_{r+1}^{(k)}\vspace*{0,1 cm}} {b_1^{(k)},\cdots,b_s^{(k)}}
-\Hypergeometric[t]{r+1}{s}{a_1^{(k)},\cdots,a_{r+1}^{(k)}\vspace*{0,1 cm}} {b_1^{(k)},\cdots,b_{[k]_s-1}^{(k)},b_{[k]_s}^{(k)}-1,b_{[k]_s+1}^{(k)},\cdots,b_s^{(k)}}
\\=&\,
-\dfrac{\prod\limits_{i=1}^{r+1}a_i^{(k)}}{\left(b_{[k]_s}-1\right)\prod\limits_{j=1}^{s}b_j^{(k)}} \,
\Hypergeometric[t]{r+1}{s}{a_1^{(k)}+1,\cdots,a_r^{(k)}+1\vspace*{0,1 cm}}{b_1^{(k)}+1,\cdots,b_s^{(k)}+1}
\\=&\,
\alpha_{k+s}\,t\,g_{k+s}(t).
\end{aligned}
\end{equation}
Otherwise, $[k]_{s+1}=\sigma_\ell\in\Sigma$ with $\ell\in\{1,\cdots,r+1\}$ ($\sigma_{r+1}=s+1$), and, using \eqref{contiguous relation 2}, we get %\vspace*{-0,2 cm}
\begin{equation}
\begin{aligned}
g_k(t)-g_{k-1}(t) 
=&\,
\Hypergeometric[t]{r+1}{s}{a_1^{(k)},\cdots,a_{r+1}^{(k)}\vspace*{0,1 cm}}{b_1^{(k)},\cdots,b_s^{(k)}}
-\Hypergeometric[t]{r+1}{s}{a_1^{(k)},\cdots,a_{\ell-1}^{(k)},a_{\ell}^{(k)}-1,a_{\ell+1}^{(k)},\cdots,a_{r+1}^{(k)}\vspace*{0,1 cm}} {b_1^{(k)},\cdots,b_{[k]_s-1}^{(k)},b_{[k]_s}^{(k)}-1,b_{[k]_s+1}^{(k)},\cdots,b_s^{(k)}}
\\=&\,
\dfrac{\left(b_{[k]_s}^{(k)}-a_\ell^{(k)}\right)\prod\limits_{i=1,\,i\neq\ell}^{r}a_i^{(k)}}{\left(b_{[k]_s}^{(k)}-1\right)\prod\limits_{j=1}^{s}b_j^{(k)}} \, \Hypergeometric[t]{r+1}{s}{a_1^{(k)}+1,\cdots,,a_{\ell-1}^{(k)}+1,a_{\ell}^{(k)},a_{\ell+1}^{(k)}+1,\cdots,a_r^{(k)}+1 \vspace*{0,1 cm}} {b_1^{(k)}+1,\cdots,b_s^{(k)}+1}
\\=&\,
\alpha_{k+s}\,t\,g_{k+s}(t).
\end{aligned}
\end{equation}
\end{proof}

\subsection{Conditions for positivity of the coefficients}
The following result gives necessary and sufficient conditions for the non-negativity and for positivity of all coefficients of the branched continued fractions introduced in Theorem \ref{BCF for ratios of r+1Fs - theorem}. %when $r\geq s$.
In Section \ref{Type II MOP w.r.t. Meijer G-functions}, we revisit in more detail these conditions when $a_{r+1}=1$.
\begin{proposition}
	\label{conditions for non-negativity of the coefficients of the BCF for a ratio of r+1Fs, r>=s}
	For $r,s\in\N$ such that $s\leq r\neq 0$, let $1\leq\lambda_1<\cdots<\lambda_s\leq r$ and $a_1,\cdots,a_{r+1},b_1,\cdots,b_s\in\R^+$.
	Then, the coefficients in the sequence $\seq[k\in\N]{\alpha_{k+r}}$ defined by \eqref{BCF coeff r+1Fs, r>=s} are all nonnegative if and only if 
	%\begin{subequations}
	\begin{equation}
		\label{BCF conditions for non-negativity of the coefficients 1}
		b_j\geq a_i-\ceil{\frac{i-\lambda_j}{r}}=
		\begin{cases}
			a_i &\text{if }i\leq\lambda_j \\
			a_i-1 &\text{if }i\geq\lambda_j+1 
		\end{cases}
		\quad\text{for all }1\leq i\leq r+1\text{ and }1\leq j\leq s,
	\end{equation}
	and
	\begin{equation}
		\label{BCF conditions for non-negativity of the coefficients 2}
		\frac{b_s-a_{r+1}}{b_s-1}\geq 0
		\quad\text{when}\quad
		\lambda_s=r.
	\end{equation}
	%\end{subequations}
	Furthermore, the coefficients in $\seq[k\in\N]{\alpha_{k+r}}$ are all positive if and only if all the inequalities above are strict.
\end{proposition}
When $r<s$, we cannot have non-negativity of all coefficients in $\seq[k\in\N]{\alpha_{k+s}}$ due to the minus sign in \eqref{BCF coeff r+1Fs, r<=s}.

\begin{proof}
	%Because $a_1,\cdots,a_r,b_1,\cdots,b_s\in\R^+$, $\alpha_r\in\R^+$.
	For any $k\in\N$, $a_i^{(k)}\geq a_i$ for $i\in\{1,\cdots,r+1\}$ and $b_j^{(k)}\geq b_i$ for $j\in\{1,\cdots,r\}$, so $\alpha_{k+r}>0$ whenever $[k]_r\not\in\{\lambda_1,\cdots,\lambda_s\}$.
	Otherwise, we consider separately the cases $k=0$ and $k\geq 1$.
	If $k=0$, $[k]_r\in\{\lambda_1,\cdots,\lambda_s\}$ implies that $[k]_r=r=\lambda_s$, so $\alpha_k\geq 0$ if and only if \eqref{BCF conditions for non-negativity of the coefficients 2} holds.
	If $k\geq 1$ and $[k]_r=\lambda_j$ for some $j\in\{1,\cdots,s\}$, then $\alpha_{k+r}\geq 0$ if and only if $b_j^{(k)}\geq a_{[k]_{r+1}}^{(k)}$, because $b_j^{(k)}\geq b_j+1$ implies $b_j^{(k)}-1\in\R^+$ for any $k\geq 1$. 
	We will show now that $b_j^{(k)}\geq a_{[k]_{r+1}}^{(k)}$ for all $k\geq 1$ if and only if \eqref{BCF conditions for non-negativity of the coefficients 1} holds.
	
	%To shorten the notation, we define $a_{k}^{\,\prime}=a_{[k]_{r+1}}^{(k)}$ and $b_{k}^{\,\prime}=b_j^{(k)}$ if $[k]_r=\lambda_j$.
	If $k=r(r+1)+\ell$, with $\ell\geq 1$, and $[k]_r=[\ell]_r=\lambda_j$, we have
	\begin{equation}
		b_j^{(k)}-a_{[k]_{r+1}}^{(k)}
		=\left(b_j^{(\ell)}+(r+1)\right)-\left(a_{[k]_{r+1}}^{(\ell)}+r\right)
		=b_j^{(\ell)}-a_{[\ell]_{r+1}}^{(\ell)}+1.
	\end{equation}
	
	Hence, $b_j^{(k)}-a_{[k]_{r+1}}^{(k)}\geq 0$ holds for all $k\geq 1$ such that $[k]_r=\lambda_j$ for some $j\in\{1,\cdots,s\}$ if and only if that holds for all $1\leq k\leq r(r+1)$. 
	%satisfying those conditions. %such that $[k]_r\in\{\lambda_1,\cdots,\lambda_s\}$.
	Note that each $k\in\{1,\cdots,r(r+1)\}$ can be uniquely written in the form
	%\begin{equation}
	%\{1,\cdots,m(m+1)\}=\left\{k=qm+p=q(m+1)+(p-q)\;|\;q\in\{0,\cdots,m\}\;\text{and}\;p=[k]_m\in\{1,\cdots,m\}\right\}.
	%\end{equation}
	\begin{equation}
		\label{nu from 1 to r(r+1)}
		k=qr+p=q(r+1)+(p-q)
		\quad\text{with}\quad
		q\in\{0,\cdots,r\} 
		\quad\text{and}\quad 
		p\in\{1,\cdots,r\}.
	\end{equation}
	
	Therefore, $\alpha_{k+r}\geq 0$ for all $k\geq 1$ if and only if $b_j^{(k)}-a_{[k]_{r+1}}^{(k)}\geq 0$ for all $k\in\{1,\cdots,r(r+1)\}$ of the form \eqref{nu from 1 to r(r+1)} with $p=[k]_r=\lambda_j$ with $j\in\{1,\cdots,s\}$.
	For $k$ of the form in \eqref{nu from 1 to r(r+1)}, we get
	\begin{equation}
		b_j^{(k)}=b_j+q+1
		\quad\text{and}\quad
		a_{[k]_{r+1}}^{(k)}=
		\begin{cases}
			a_{\lambda_j-q}+(q+1) & \text{if } 0\leq q\leq\lambda_j-1, \\
			a_{\lambda_j-q+r+1}+q & \text{if } \lambda_j\leq q\leq r.
		\end{cases}
	\end{equation}
	If $0\leq q\leq\lambda_j-1$, we set $i=\lambda_j-q\in\{1,\cdots,\lambda_j\}$ and we have $b_j^{(k)}-a_{[k]_{r+1}}^{(k)}\geq 0$ if and only if $b_j\geq a_i$. 
	Otherwise, $\lambda_j\leq q\leq r$, we set $i=\lambda_j-q+r+1\in\{\lambda_j+1,\cdots,r+1\}$, and we get $b_j^{(k)}-a_{[k]_{r+1}}^{(k)}\geq 0$ if and only if $b_j\geq a_i-1$.
	Therefore, we conclude that $b_j^{(k)}\geq a_{[k]_{r+1}}^{(k)}$ for all $k\geq 1$ if and only if \eqref{BCF conditions for non-negativity of the coefficients 1} holds.
\end{proof}

\subsection{Modified $m$-Stieltjes-Rogers-polynomials when $a_{r+1}=1$}
The following result focus on the particular case $a_{r+1}=1$ of Theorem \ref{BCF for ratios of r+1Fs - theorem}.
In that case, we can find explicit expressions for the modified $m$-Stieltjes-Rogers-polynomials $\seq[n,k\in\N]{\modifiedStieltjesRogersPoly{n}{k}{\alpha}}$ and they reduce to ratios of products of Pochhammer symbols when $k\leq m$, which correspond to the moments of the linear functionals of orthogonality for the multiple orthogonal polynomials studied in Section \ref{Type II MOP w.r.t. linear functionals}.
\begin{corollary}
\label{BCF for ratios of r+1Fs - corollary a_(r+1)=1}
For $r,s\in\N$ with $m=\max(r,s)\geq 1$, let $\seq[k\in\N]{\alpha_{k+m}}$ be defined, as in Theorem \ref{BCF for ratios of r+1Fs - theorem}, by \eqref{BCF coeff r+1Fs, r>=s} if $r\geq s$ or by \eqref{BCF coeff r+1Fs, r<=s} if $r\leq s$, and suppose that $a_{r+1}=1$. 
Then, the generating function of the modified $m$-Stieltjes-Rogers polynomials of type $k$, $\seq{\modifiedStieltjesRogersPoly{n}{k}{\alpha}}$, with $k\in\N$, is $g_k(t)$ defined in \eqref{g_k as a r+1Fs hypergeometric function}, and
\begin{equation}
\label{modified m-S.-R. poly a_(r+1)=1}
\modifiedStieltjesRogersPoly{n}{k}{\alpha}
=\binom{n+\ceil{\frac{k-m}{m+1}}}{n}\,\frac{\pochhammer{a_1^{(k)}}\cdots\pochhammer{a_r^{(k)}}}{\pochhammer{b_1^{(k)}}\cdots\pochhammer{b_s^{(k)}}}
\quad\text{for any   }n,k\in\N.
\end{equation}
In particular,
\begin{equation}
\label{modified m-S.-R. poly a_(r+1)=1 k<=r}
\modifiedStieltjesRogersPoly{n}{k}{\alpha}=\frac{\pochhammer{a_1^{(k)}}\cdots\pochhammer{a_r^{(k)}}}{\pochhammer{b_1^{(k)}}\cdots\pochhammer{b_s^{(k)}}}
\quad\text{for   }n\in\N\text{   and   }0\leq k\leq m.
\end{equation}

\end{corollary}

\begin{proof}
If we take $a_{r+1}=1$ in \eqref{g_k as a r+1Fs hypergeometric function}, then $a_{r+1}^{(-1)}=a_{r+1}-1=0$ and, as a result, $g_{-1}(t)=1$.
Therefore, applying Lemma \ref{Euler-Gauss method for m-S-fractions} to the recurrence relation \eqref{rec rel g_k}, we find that the generating function of the modified $m$-Stieltjes-Rogers polynomials of type $k$ is $g_k(t)=f_0(t)\cdots f_k(t)$.
Furthermore,
\begin{equation}
a_{r+1}^{(k)}=1+\ceil{\frac{k-m}{m+1}}
\implies
\frac{\pochhammer{a_{r+1}^{(k)}}}{n!}=\frac{\pochhammer{1+\ceil{\frac{k-m}{m+1}}}}{n!}=\binom{n+\ceil{\frac{k-m}{m+1}}}{n}.
\end{equation}
Therefore,
\begin{equation}
\sum_{n=0}^{\infty}\modifiedStieltjesRogersPoly{n}{k}{\alpha}\,t^n
=\Hypergeometric[t]{r+1}{s}{a_1^{(k)},\cdots,a_r^{(k)},1+\ceil{\frac{k-m}{m+1}} \vspace*{0,1 cm}}{b_1^{(k)},\cdots,b_s^{(k)}}
=\sum_{n=0}^{\infty}
	\binom{n+\ceil{\frac{k-m}{m+1}}}{n}\,\frac{\pochhammer{a_1^{(k)}}\cdots\pochhammer{a_r^{(k)}}}{\pochhammer{b_1^{(k)}}\cdots\pochhammer{b_s^{(k)}}}\,t^n.
\end{equation}
As a result, we obtain \eqref{modified m-S.-R. poly a_(r+1)=1}.
When $0\leq k\leq m$, $\ceil{\dfrac{k-m}{m+1}}=0$ and \eqref{modified m-S.-R. poly a_(r+1)=1} reduces to \eqref{modified m-S.-R. poly a_(r+1)=1 k<=r}.
\end{proof}	

\subsection{A limiting type of ratios}
The limiting case $a_{r+1}\to\infty$ of the branched continued fractions introduced in Theorem \ref{BCF for ratios of r+1Fs - theorem} gives new branched-continued-fraction representations for the third ratio of hypergeometric series ${}_rF_s$ in \eqref{ratios of hypergeometric series in PSZ}.
Let $\seq[k\geq -1]{g_k(t)}$ and $\seq[k\in\N]{\alpha_{k+m}}$ be defined by \eqref{g_k as a r+1Fs hypergeometric function} and \eqref{BCF coeff r+1Fs, r>=s}-\eqref{BCF coeff r+1Fs, r<=s}, respectively, as in Theorem \ref{BCF for ratios of r+1Fs - theorem}, and define
\begin{equation}
\hat{g}_k(t)=\lim_{a_{r+1}\to\infty}g_k\left(\frac{t}{a_{r+1}}\right)
=\Hypergeometric[t]{r}{s}{a_1^{(k)},\cdots,a_r^{(k)}\vspace*{0,1 cm}}{b_1^{(k)},\cdots,b_s^{(k)}}.
\end{equation}
Then, the ratios of contiguous hypergeometric series $\seq[k\in\N]{\hat{f}_k(t)=\dfrac{\hat{g}_k(t)}{\hat{g}_{k-1}(t)}}$ \vspace*{0,1 cm} 
admit a $m$-branched-continued-fraction representation with coefficients $\seq[k\in\N]{\hat{\alpha}_{k+m}}$ defined by
\begin{equation}
\label{BCF coefficients third ratio as limits}
\hat{\alpha}_{k+m}=\lim_{a_{r+1}\to\infty}\frac{\alpha_{k+m}}{a_{r+1}}.
\end{equation}

This construction generalises the branched continued fractions obtained in \cite[\S 14]{AlanSokalM.PetroelleB.Zhu-LPandBCF1} for the third ratio of contiguous hypergeometric series in \eqref{ratios of hypergeometric series in PSZ}, analogously to how Theorem \ref{BCF for ratios of r+1Fs - theorem} generalises the branched continued fractions obtained in \cite[\S 14]{AlanSokalM.PetroelleB.Zhu-LPandBCF1} for the first and second ratios of contiguous hypergeometric series ${}_{r+1}F_s$.
In fact, this construction reduces to parts (a) and (b) of \cite[Th.~14.12]{AlanSokalM.PetroelleB.Zhu-LPandBCF1} when $\lambda_j=r-s+j$ for all $1\leq j\leq s$ with $r\geq s$ and when $\sigma_i=s-r+i$ for all $1\leq i\leq r$ with $r\leq s$, respectively.

However, we do not explicitly write the coefficients of these branched continued fractions here, because they are very similar to the coefficients in Theorem \ref{BCF for ratios of r+1Fs - theorem}.
Moreover, Theorem \ref{BCF for ratios of r+1Fs - theorem} already gives branched-continued-fraction representations for ratios of contiguous hypergeometric series ${}_rF_s$ when $r\geq 1$.
To find branched-continued-fraction representations for ratios of contiguous ${}_0F_s$ with $s\geq 1$, see \cite[Th.~14.8]{AlanSokalM.PetroelleB.Zhu-LPandBCF1}.

The functions $g_0$ and $g_{-1}$ in Theorem \ref{BCF for ratios of r+1Fs - theorem}, and consequently their ratio $f_0$, are invariant under permutations of $\left(a_1,\cdots,a_r\right)$ (but not $a_{r+1}$) and $\left(b_1,\cdots,b_{s-1}\right)$ (but not necessarily $b_s$).
The branched-continued-fraction coefficients $\seq[k\in\N]{\alpha_{k+m}}$ defined by \eqref{BCF coeff r+1Fs, r>=s}-\eqref{BCF coeff r+1Fs, r<=s} are not invariant under these permutations.
However, when $a_{r+1}=1$, their production matrix is invariant under permutations of $\left(a_1,\cdots,a_r\right)$, but not of any of the $b_j$, as well as under different choices of $\left(\sigma_1,\cdots,\sigma_r\right)$ when $r<s$, but not of $\left(\lambda_1,\cdots,\lambda_s\right)$ when $r>s$.
We will prove this statement at the end of the next section, as a consequence of the corresponding multiple orthogonal polynomials having the same symmetries.

%\newpage
\section{A general class of hypergeometric multiple orthogonal polynomials}
\label{Type II MOP w.r.t. linear functionals}

In this section, we investigate the type II multiple orthogonal polynomials on the step-line with respect to the linear functionals $\left(u_0,\cdots,u_{m-1}\right)$, with $m\in\Z^+$, whose moments are the modified $m$-Stieltjes-Rogers polynomials in \eqref{modified m-S.-R. poly a_(r+1)=1 k<=r}, equal to ratios of products of Pochhammer symbols.
These multiple orthogonal polynomials have coefficients in the ring $R=\Q\left(b_1,\cdots,b_s\right)\left[a_1,\cdots,a_r\right]$ of polynomials in the indeterminates $a_1,\cdots,a_r$ whose coefficients are rational functions in the indeterminates $b_1,\cdots,b_s$.

Firstly, we find explicit expressions as terminating hypergeometric series for the multiple orthogonal polynomials under study (Theorem \ref{explicit formula as hypergeometric polynomials for the r-OP - theorem}).
Then, we derive differential properties satisfied by the polynomials (Propositions \ref{differential equation for type II MOP prop.} and \ref{derivative of the type II MOP as a shift on the parameters prop.}). 
Finally, we use the connection of these multiple orthogonal polynomials with the branched continued fractions introduced in Theorem \ref{BCF for ratios of r+1Fs - theorem} to obtain expressions for the recurrence coefficients of the polynomials as combinations of the branched-continued-fraction coefficients (Theorem \ref{recurrence coefficients for the MOP and link to the BCF for ratios of hypergeometric series}).

\subsection{Explicit expressions as terminating hypergeometric series}
The main goal of this section is to prove the following result.
\begin{theorem}
\label{explicit formula as hypergeometric polynomials for the r-OP - theorem}

For $r,s\in\N$ with $m=\max(r,s)\geq 1$, let: \vspace*{-0,3 cm}
\begin{itemize}[leftmargin=*]
\item
$1\leq\lambda_1<\cdots<\lambda_s\leq r$ when $r\geq s$ and $1\leq\sigma_1<\cdots<\sigma_r\leq s$ when $r\leq s$;
\item
$R=\Q\left(b_1,\cdots,b_s\right)\left[a_1,\cdots,a_r\right]$ for $a_1,\cdots,a_r,b_1,\cdots,b_s$ such that, for all $1\leq i\leq r$, $1\leq j\leq s$ and $n\in\N$,
\begin{itemize}
\item
$a_i+n,b_j+n\neq 0$,
\item
$b_j-a_i+n\neq 0$ if $i\leq\lambda_j$ and $b_j-a_i+1+n\neq 0$ if $i\geq\lambda_j+1$ when $r\geq s$,
\item
$b_j-a_i+n\neq 0$ if $\sigma_i\leq j$ and $b_j-a_i+1+n\neq 0$ if $\sigma_i>j+1$ when $r\leq s$;
\end{itemize}
%\begin{equation}
%	a_i+n,b_j+n\neq 0,\;
%	b_j-a_i+n\neq 0 \text{ if }i\leq\lambda_j,\;
%	b_j-a_i+1+n\neq 0 \text{ if }i\geq\lambda_j+1
%	\quad\text{where }\lambda_j=j\text{ if }r\leq s,
%	%\quad\text{for all}\;1\leq i\leq r,\;1\leq j\leq s,\;\text{and}\;n\in\N,
%\end{equation}
\item $\left(v_0,\cdots,v_{m-1}\right)$ the vector of linear functionals acting on $R[x]$ with moments
\begin{equation}
\label{definition of the orthogonality functionals}
\Functional{v_k}{x^n}=\frac{\pochhammer{a_1^{(k)}}\cdots\pochhammer{a_r^{(k)}}}{\pochhammer{b_1^{(k)}}\cdots\pochhammer{b_s^{(k)}}}
\quad\text{for } n\in\N \text{  and  }0\leq k\leq m-1,
\end{equation}
where the $a_i^{(k)}$ and $b_j^{(k)}$ are defined by \eqref{BCF parameters r+1Fs, r>=s}-\eqref{BCF parameters r+1Fs, r<=s}, which means that, for any $0\leq k\leq m-1$,
\begin{equation}
\label{parameters of the orthogonality functionals r>=s}
a_i^{(k)}=%a_i+\ceil{\dfrac{k+1-i}{r}}=
\begin{cases}
a_i+1 & \text{if }1\leq i\leq k \\
	a_i & \text{if }k+1\leq i\leq r
\end{cases}
\quad\text{and}\quad
b_j^{(k)}=%b_j+\ceil{\dfrac{k+1-j}{s}}=
\begin{cases}
b_j+1 & \text{if }1\leq\lambda_j\leq k \\
	b_j & \text{if }k+1\leq\lambda_j\leq r
\end{cases}
\quad\text{when   }r\geq s,
\end{equation}
and
\begin{equation}
\label{parameters of the orthogonality functionals r<=s}
a_i^{(k)}=%a_i+\ceil{\dfrac{k+1-i}{r}}=
\begin{cases}
a_i+1 & \text{if }1\leq\sigma_i\leq k \\
	a_i & \text{if }k+1\leq\sigma_i\leq s
\end{cases}
\quad\text{and}\quad
b_j^{(k)}=%b_j+\ceil{\dfrac{k+1-j}{s}}=
\begin{cases}
b_j+1 & \text{if }1\leq j\leq k \\
	b_j & \text{if }k+1\leq j\leq s
\end{cases}
\quad\text{when   }r\leq s.
\end{equation}
\end{itemize}
Then, the $m$-orthogonal polynomial sequence $\seq{P_n(x)}$ with respect to $\left(u_0,\cdots,u_{m-1}\right)$ is given by
\begin{equation}
%\begin{aligned}%[t]
\label{hypergeometric type II MOP - explicit formula as a s+1Fr}
P_n(x) %&
=\frac{(-1)^n\pochhammer{a_1}\cdots\pochhammer{a_r}}{\pochhammer{b_1^{(n-1)}}\cdots\pochhammer{b_s^{(n-1)}}} \,\Hypergeometric{s+1}{r}{-n,b_1^{(n-1)},\cdots,b_s^{(n-1)}}{a_1,\cdots,a_r},
%\\&%=\sum_{k=0}^{n}(-1)^k\binom{n}{k}\frac{\pochhammer[k]{a_1+n-k}\cdots\pochhammer[k]{a_r+n-k}} {\pochhammer[k]{b_1^{(n-1)}+n-k}\cdots\pochhammer[k]{b_s^{(n-1)}+n-k}}\,x^{n-k},
%\end{aligned}
\end{equation}
with 
\begin{equation}
\label{b_j parameters in the type II MOP}
b_j^{(n-1)}=b_j+\ceil{\dfrac{n-\lambda_j}{r}}
\text{  when  }r\geq s
\quad\text{or}\quad
b_j^{(n-1)}=b_j+\ceil{\dfrac{n-j}{s}}
\text{  when  }\;r\leq s.
\end{equation}

\end{theorem}

%Note that the definitions of $b_j^{(n-1)}$ in \eqref{b_j parameters in the type II MOP} and $b_j^{(k)}$ in \eqref{parameters of the orthogonality functionals r>=s}-\eqref{parameters of the orthogonality functionals r<=s} coincide when $0\leq k=n-1\leq r-1$.
%
%Observe that $a_i^{(0)}=a_i$ and $b_j^{(0)}=b_j$ for all $i\in\{1,\cdots,r\}$ and $j\in\{1,\cdots,s\}$.
%Therefore, the moments of $u_0$ are the ratios of products of Pochhammer symbols in \eqref{moment sequence ratio of Pochhammers} and
%\begin{equation}
%P_1(x)=-\frac{a_1\cdots a_r}{b_1\cdots b_s}\,\Hypergeometric{s+1}{r}{-1,b_1,\cdots,b_s}{a_1,\cdots,a_r}=x-\frac{a_1\cdots a_r}{b_1\cdots b_s}.
%\end{equation}

When $r=s=m$, we have
\begin{equation}
\label{explicit formula for the r-OP with s=r}
P_n(x)=\frac{(-1)^n\pochhammer{a_1}\cdots\pochhammer{a_m}}{\pochhammer{b_1^{(n-1)}}\cdots\pochhammer{b_m^{(n-1)}}}
\,\Hypergeometric{m+1}{m}{-n,b_1^{(n-1)},\cdots,b_m^{(n-1)}}{a_1,\cdots,a_m}
\quad\text{with } b_j^{(n-1)}=b_j+\ceil{\dfrac{n-j}{m}}.
\end{equation}
The polynomials in \eqref{hypergeometric type II MOP - explicit formula as a s+1Fr} with $r\neq s$ can be obtained as limiting cases of \eqref{explicit formula for the r-OP with s=r}. 

When we want to highlight the parameters in the polynomial sequence $\seq{P_n(x)}$ defined by \eqref{hypergeometric type II MOP - explicit formula as a s+1Fr}, we denote $P_n(x)$ by $P_n^{\left[(r,s),\left(\lambda_1,\cdots,\lambda_s\right)\right]}\left(x\left|\begin{matrix}a_1,\cdots,a_r \\ b_1,\cdots,b_s\end{matrix}\right.\right)$ if $r\geq s$ or by $P_n^{[(r,s)]}\left(x\left|\begin{matrix}a_1,\cdots,a_r \\ b_1,\cdots,b_s\end{matrix}\right.\right)$ if $r\leq s$.
Note that $\seq{P_n(x)}$ defined by \eqref{hypergeometric type II MOP - explicit formula as a s+1Fr} with $r\leq s$ does not depend on the choice of $\left(\sigma_1,\cdots,\sigma_r\right)$, although the same is not true for the corresponding linear functionals.
%We use an analogous notation to highlight the parameters in the linear functionals defined by \eqref{definition of the orthogonality functionals}-\eqref{parameters of the orthogonality functionals}.

If $s<r$, let $\mathrm{J}:=\{1,\cdots,r\}\backslash\{\lambda_1,\cdots,\lambda_s\}\neq\emptyset$, $B=\dis\prod\limits_{j\in\mathrm{J}}b_j$, and $\hat{b}_j=b_{\lambda_j}$ for $1\leq j\leq s$.
Then,
\begin{equation}
\label{limiting relation MOP (r,r) to (r,s), s<r}
P_n^{\left[(r,s),\left(\lambda_1,\cdots,\lambda_s\right)\right]}
\left(x \left|\begin{matrix} a_1,\cdots,a_r \\ \hat{b}_1,\cdots,\hat{b}_s \end{matrix}\right.\right)
=\lim\limits_{\substack{b_j\to\infty\\ j\in\mathrm{J}}}\,B^n\,P_n^{\left[(r,r)\right]}\left(\frac{x}{B} \left|\begin{matrix} a_1,\cdots,a_r \\ b_1,\cdots,b_r \end{matrix}\right.\right).
\end{equation}
Analogously, if $r<s$, let $\mathrm{I}:=\{1,\cdots,s\}\backslash\{\sigma_1,\cdots,\sigma_r\}\neq\emptyset$, $A=\dis\prod\limits_{i\in\mathrm{I}}a_i$, and $\hat{a}_i=a_{\sigma_i}$ for $1\leq i\leq r$.
Then,
\begin{equation}
\label{limiting relation MOP (s,s) to (r,s), s>r}
P_n^{\left[(r,s)\right]}\left(x \left|\begin{matrix} \hat{a}_1,\cdots,\hat{a}_r \\ b_1,\cdots,b_s \end{matrix}\right.\right)
=\lim\limits_{\substack{a_i\to\infty\\ i\in\mathrm{I}}}\,A^{-n}\,P_n^{\left[(s,s)\right]}\left(A\,x\left|\begin{matrix} a_1,\cdots,a_s \\ b_1,\cdots,b_s \end{matrix}\right.\right).
\end{equation}
Similar limiting relations hold for the corresponding orthogonality functionals.

%\textbf{Give some motivation for the reciprocal polynomials.}
By definition of the hypergeometric series, the $m$-orthogonal polynomials $\seq{P_n(x)}$ can be written as
\begin{equation}
\label{hypergeometric type II MOP - explicit expression as a combination of powers of x}
P_n(x) =\sum_{k=0}^{n}(-1)^k\binom{n}{k}\frac{\pochhammer[k]{a_1+n-k}\cdots\pochhammer[k]{a_r+n-k}} {\pochhammer[k]{b_1^{(n-1)}+n-k}\cdots\pochhammer[k]{b_s^{(n-1)}+n-k}}\,x^{n-k}.
\end{equation}
Therefore, the reciprocal polynomials $\seq{x^n\,P_n\left(\frac{1}{x}\right)}$, which are the polynomials with the same coefficients as $\seq{P_n(x)}$ in reverse order, are
\begin{equation}
x^n\,P_n\left(\frac{1}{x}\right)
=\sum_{k=0}^{n}(-1)^k\binom{n}{k}\frac{\pochhammer[k]{a_1+n-k}\cdots\pochhammer[k]{a_r+n-k}} {\pochhammer[k]{b_1^{(n-1)}+n-k}\cdots\pochhammer[k]{b_s^{(n-1)}+n-k}}\,x^k.
\end{equation}
Using again the definition of the hypergeometric series, the latter is equivalent to
\begin{equation}
x^n\,P_n\left(\frac{1}{x}\right)=\Hypergeometric[(-1)^{r+s}\,x]{r+1}{s}{-n,-n-a_1+1,\cdots,-n-a_r+1}{-n-b_1^{(n-1)}+1,\cdots,-n-b_s^{(n-1)}+1}.
\end{equation}
The latter formula can also be obtained as a particular case of the more general formula for the reciprocal of a hypergeometric polynomial:
\begin{equation}
x^n\,\Hypergeometric[\frac{1}{x}]{p+1}{q}{-n,c_1,\cdots,c_p}{d_1,\cdots,d_q}
=\frac{(-1)^n\pochhammer{c_1}\cdots\pochhammer{c_p}}{\pochhammer{d_1}\cdots\pochhammer{d_q}}\, \Hypergeometric[(-1)^{p+q}\,x]{q+1}{p}{-n,-n-d_1+1,\cdots,-n-d_q+1}{-n-c_1+1,\cdots,-n-c_p+1}.
\end{equation}
The latter result can be found in \cite{FasenmyerHypergeomPoly}.
Therein it is assumed that $c_1=n+1$ and $d_1=\frac{1}{2}$, but the proof is the same in the general case.

Recalling the orthogonality conditions in \eqref{d-orthogonality conditions}, Theorem \ref{explicit formula as hypergeometric polynomials for the r-OP - theorem} holds if and only if $\seq{P_n(x)}$ defined by \eqref{hypergeometric type II MOP - explicit formula as a s+1Fr} satisfies the orthogonality conditions
\begin{equation}
\label{orthogonality conditions w.r.t. linear functionals}
v_{\ell}\left[x^kP_n\right]
=\begin{cases}
N_n\neq 0 &\text{ if } n=mk+\ell, \\
			   0 &\text{ if } n\geq mk+\ell+1,
\end{cases}
\quad\text{for all }k,n\in\N \text{ and } \ell\in\{0,\cdots,m-1\}.
\end{equation}

To show that \eqref{orthogonality conditions w.r.t. linear functionals} holds, we use the following Lemmas \ref{Integrals of hypergeometric polynomials} and \ref{explicit formulas for hypergeometric series}.
\begin{lemma}
\label{Integrals of hypergeometric polynomials}	
For $p,q,r,s\in\N$, let:
\begin{itemize}
\item $R:=\Q\left(b_1,\cdots,b_s,c_1,\cdots,c_p\right)\left[a_1,\cdots,a_r,d_1,\cdots,d_q\right]$,

\item $v:R[x]\to R$ the linear functional defined by 
\begin{equation}
\label{linear functional moments equal to ratio of Pochhammers}
\Functional{v}{x^n}
=\frac{\pochhammer{a_1}\cdots\pochhammer{a_r}}{\pochhammer{b_1}\cdots\pochhammer{b_s}}
\quad\text{for any }n\in\N,
\end{equation}

\item $\seq{P_n(x)}$ the polynomial sequence in $R[x]$ whose elements are
\begin{equation}
P_n(x)=\frac{(-1)^n\,\pochhammer{d_1}\cdots\pochhammer{d_q}}{\pochhammer{c_1}\cdots\pochhammer{c_p}}
\,\Hypergeometric{p+1}{q}{-n,c_1,\cdots,c_p}{d_1,\cdots,d_q}.
\end{equation}
\end{itemize}

Then, for any $k,n\in\N$,
\begin{equation}
\label{hypergeometric integral}
\Functional{v}{x^kP_n}
=\frac{(-1)^n\,\pochhammer{d_1}\cdots\pochhammer{d_q}\,\pochhammer[k]{a_1}\cdots\pochhammer[k]{a_r}} {\pochhammer{c_1}\cdots\pochhammer{c_p}\,\pochhammer[k]{b_1}\cdots\pochhammer[k]{b_s}}
\Hypergeometric[1]{p+r+1}{q+s}{-n,c_1,\cdots,c_p,a_1+k,\cdots,a_r+k\vspace*{0,1 cm}}{d_1,\cdots,d_q,b_1+k,\cdots,b_s+k}. 
\end{equation}
\end{lemma}	

\begin{proof}
By definition of $P_n$ and linearity of $v$, we have
\begin{equation}
\Functional{v}{x^kP_n}
=\frac{(-1)^n\,\pochhammer{d_1}\cdots\pochhammer{d_q}}{\pochhammer{c_1}\cdots\pochhammer{c_p}}
\,\mathlarger{\sum}_{j=0}^{n}\,
\frac{\pochhammer[j]{c_1}\cdots\pochhammer[j]{c_p}}{j!\pochhammer[j]{d_1}\cdots\pochhammer[j]{d_q}}\,v\left[x^{k+j}\right].
\end{equation}

Applying \eqref{linear functional moments equal to ratio of Pochhammers} to the latter, we get
\begin{equation}
\Functional{v}{x^kP_n}
=\frac{(-1)^n\,\pochhammer{d_1}\cdots\pochhammer{d_q}}{\pochhammer{c_1}\cdots\pochhammer{c_p}}
\,\mathlarger{\sum}_{j=0}^{n}\,
\frac{\pochhammer[j]{c_1}\cdots\pochhammer[j]{c_p}\pochhammer[k+j]{a_1}\cdots\pochhammer[k+j]{a_r}} {j!\,\pochhammer[j]{d_1}\cdots\pochhammer[j]{d_q}\pochhammer[k+j]{b_1}\cdots\pochhammer[k+j]{b_s}},
\end{equation}
which is equivalent to
\begin{equation}
\Functional{v}{x^kP_n}
=\frac{(-1)^n\,\pochhammer{d_1}\cdots\pochhammer{d_q}\,\pochhammer[k]{a_1}\cdots\pochhammer[k]{a_r}} {\pochhammer{c_1}\cdots\pochhammer{c_p}\,\pochhammer[k]{b_1}\cdots\pochhammer[k]{b_s}}
\,\mathlarger{\sum}_{j=0}^{n}\,
\frac{\pochhammer[j]{c_1}\cdots\pochhammer[j]{c_p}\,\pochhammer[j]{a_1+k}\cdots\pochhammer[j]{a_r+k}} {j!\,\pochhammer[j]{d_1}\cdots\pochhammer[j]{d_q}\,\pochhammer[j]{b_1+k}\cdots\pochhammer[j]{b_s+k}}.
\end{equation}
Therefore, by definition of the generalised hypergeometric series, \eqref{hypergeometric integral} holds.
\end{proof}

\begin{lemma}
%\cite[Lemma 3.2]{PaperTricomiWeights}
\label{explicit formulas for hypergeometric series}
Let $p,n,m_1,\cdots,m_p\in\N$ such that $\dis m:=\sum_{i=1}^{p}m_i\leq n$ and let $R=\Q\left(c_1,\cdots,c_p,d\right)$.
Then,
\begin{equation}
\label{hypergeometric formula 1}
\Hypergeometric[1]{p+1}{p}{-n,c_1+m_1,\cdots,c_p+m_p}{c_1,\cdots,c_p}
=\begin{cases}
0 & \text{ if } m\leq n-1,\\
\dis\frac{(-1)^n\,n!}{\dis\prod_{i=1}^{p}\pochhammer[m_i]{c_i}} & \text{ if } m=n;
\end{cases}
\end{equation}
and
\begin{equation}
\label{Minton's formula}
\Hypergeometric[1]{p+2}{p+1}{-n,d,c_1+m_1,\cdots,c_p+m_p}{d+1,c_1,\cdots,c_p}
=\frac{n!\dis\prod_{i=1}^{p}\pochhammer[m_i]{c_i-d}}{\pochhammer{d+1}\dis\prod_{i=1}^{p}\pochhammer[m_i]{c_i}}.
\end{equation}
	
\end{lemma}
Formula \eqref{Minton's formula} was deduced in \cite{Minton} and \eqref{hypergeometric formula 1} can be obtained by taking the limit $d\to +\infty$ in \eqref{Minton's formula}.

%We can now prove Theorem \ref{explicit formula as hypergeometric polynomials for the r-OP - theorem}. 

\begin{proof}[Proof of Theorem \ref{explicit formula as hypergeometric polynomials for the r-OP - theorem}.]
Our aim is to prove that \eqref{orthogonality conditions w.r.t. linear functionals} holds.
Using Lemma \ref{Integrals of hypergeometric polynomials}, we have
\begin{equation}
\label{orthogonality conditions r=s=m - functional to compute}
\Functional{v_{\ell}}{x^kP_n}
=\frac{\dis(-1)^n\prod\limits_{i=1}^{r}\pochhammer{a_i}\,\prod\limits_{i=1}^{r}\pochhammer[k]{a_i^{(\ell)}}} {\dis\prod\limits_{j=1}^{s}\pochhammer{b_j^{(n-1)}}\,\prod\limits_{j=1}^{s}\pochhammer[k]{b_j^{(\ell)}}}
\,\Hypergeometric[1]{r+s+1}{r+s}{-n,a_1^{(\ell)}+k,\cdots,a_r^{(\ell)}+k,b_1^{(n-1)},\cdots,b_s^{(n-1)}\vspace*{0,1 cm}} {a_1,\cdots,a_r,b_1^{(\ell)}+k,\cdots,b_s^{(\ell)}+k}.
\end{equation}
for any $k,n\in\N$ and $\ell\in\{0,\cdots,m-1\}$.
To shorten the notation throughout this proof, we define
\begin{equation}
%\begin{aligned}
\label{orthogonality conditions r=s=m - hypergeometric series involved}
F_{\ell}(n,k)=\Hypergeometric[1]{r+s+1}{r+s}{-n,a_1^{(\ell)}+k,\cdots,a_r^{(\ell)}+k,b_1^{(n-1)},\cdots,b_s^{(n-1)}\vspace*{0,1 cm}} {a_1,\cdots,a_r,b_1^{(\ell)}+k,\cdots,b_s^{(\ell)}+k}.
\end{equation}
Suppose that $n\geq mk+\ell+1$. 
Then,
\begin{equation}
b_j^{(n-1)}=b_j+\ceil{\frac{n-\lambda_j}{m}}\geq b_j+k+\ceil{\frac{\ell+1-\lambda_j}{m}}=b_j^{(\ell)}+k
\quad\text{for each }1\leq j\leq m.
\end{equation}
Therefore, $a_i^{(\ell)}+k-a_i$ and $b_j^{(n-1)}-b_j^{(\ell)}-k$ are nonnegative integers for any $1\leq i\leq r$ and $1\leq j\leq s$.
Furthermore,
\begin{equation}
\sum_{i=1}^{r}\left(a_i^{(\ell)}+k-a_i\right)\leq rk+\ell\leq mk+\ell
\end{equation}
and 
\begin{equation}
\begin{aligned}
\sum_{j=1}^{s}\left(b_j^{(n-1)}-b_j^{(\ell)}-k\right)
&
=\sum_{j=1}^{s}\left(\ceil{\frac{n-\lambda_j}{m}}-\ceil{\dfrac{\ell+1-\lambda_j}{m}}-k\right)
\\&
\leq\sum_{i=1}^{m}\left(\ceil{\frac{n-i}{m}}-\ceil{\dfrac{\ell+1-i}{m}}-k\right)
=n-1-mk-\ell.
\end{aligned}
\end{equation}
As a result,
\begin{equation}
\sum_{i=1}^{r}\left(a_i^{(\ell)}+k-a_i\right)+\sum_{j=1}^{s}\left(b_j^{(n-1)}-b_j^{(\ell)}-k\right)\leq n-1.
\end{equation}
Hence, using \eqref{hypergeometric formula 1}, we find that $F_{\ell}(n,k)=0$ and, consequently, $\Functional{v_{\ell}}{x^k\,P_n}=0$ whenever $n\geq mk+\ell+1$.

Suppose now that $n=mk+\ell$. 
Then,
\begin{equation}
b_j^{(n-1)}=b_j+\ceil{\dfrac{n-\lambda_j}{m}}=b_j+k+\ceil{\dfrac{\ell-\lambda_j}{m}}=
\begin{cases}
b_j+k+1 & \text{if } 1\leq\lambda_j\leq\ell-1 \\
    b_j+k & \text{if } \ell\leq\lambda_j\leq m
\end{cases}
\quad\text{for }1\leq\ell\leq m-1,
\end{equation}
and
\begin{equation}
b_j^{(mk-1)}=b_j+\ceil{\dfrac{mk-j}{m}}=%b_j+k-1+\ceil{\dfrac{m-j}{m}}=
\begin{cases}
  b_j+k & \text{if } 1\leq \lambda_j\leq m-1 \\
b_j+k-1 & \text{if } \lambda_j=m
\end{cases}
\quad\text{for }\ell=0.
\end{equation}

Therefore, there are two cases to consider separately: either $[\ell]_m\not\in\{\lambda_1,\cdots,\lambda_s\}$ or $[\ell]_m\in\{\lambda_1,\cdots,\lambda_s\}$.

We define $\eta(\ell)=\max\{0\leq j\leq s\left|\lambda_j\leq\ell\right.\}$ and $\zeta(\ell)=\max\{0\leq i\leq r\left|\sigma_i\leq\ell\right.\}$ with $\lambda_0=\sigma_0=0$.
Note that $\eta(\ell)\leq\ell$ with $\eta(\ell)=\ell$ when $r\leq s$, and $\zeta(\ell)\leq\ell$ with $\zeta(\ell)=\ell$ when $r\geq s$.

If $[\ell]_m\not\in\{\lambda_1,\cdots,\lambda_s\}$, then, recalling \eqref{hypergeometric formula 1},
\begin{equation}
F_{\ell}(n,k)
=\Hypergeometric[1]{r+1}{r}{-n,a_1^{(\ell)}+k,\cdots,a_r^{(\ell)}+k \vspace*{0,1 cm}} {a_1,\cdots,a_r}
=\frac{(-1)^n\,n!}{\dis\prod\limits_{i=1}^{\zeta(\ell)}\pochhammer[k+1]{a_i}\prod\limits_{i=\zeta(\ell)+1}^{r}\pochhammer[k]{a_i}}.
%\quad\text{when }n=mk+\ell,
\end{equation}
Therefore,
\begin{equation}
\label{Th. 6.1 - nonzero orthogonality condition 1}
\Functional{v_{\ell}}{x^k\,P_n}
=\frac{\dis n!\,\prod\limits_{i=1}^{\zeta(\ell)}\pochhammer[n-1]{a_i+1}\prod\limits_{i=\zeta(\ell)+1}^{r}\pochhammer{a_i}} {\dis\prod\limits_{j=1}^{s}\pochhammer{b_j^{(n-1)}}\,\prod\limits_{j=1}^{s}\pochhammer[k]{b_j^{(\ell)}}}\neq 0.
%\quad\text{when }n=mk+\ell\text{ and }[\ell]_m\not\in\{\lambda_1,\cdots,\lambda_s\}.
\end{equation}

If $[\ell]_m\in\{\lambda_1,\cdots,\lambda_s\}$, then $[\ell]_m=\lambda_{\eta(\ell)}$. 
As a result, using \eqref{Minton's formula} and defining $b_0=b_m-1$, we have
\begin{equation}
F_{\ell}(n,k)
=\Hypergeometric[1]{r+2}{r+1}{-n,a_1^{(\ell)},\cdots,a_r^{(\ell)},b_{\eta(\ell)}+k\vspace*{0,1 cm}} {a_1,\cdots,a_r,b_{\eta(\ell)}+k+1}
=\frac{\dis n!\,\prod\limits_{i=1}^{\zeta(\ell)}\pochhammer[k+1]{a_i-b_{\eta(\ell)}-k}\,\prod\limits_{i=\zeta(\ell)+1}^{r}\pochhammer[k]{a_i-b_{\eta(\ell)}-k}} {\dis\pochhammer{b_{\eta(\ell)}+k+1}\;\prod\limits_{i=1}^{\zeta(\ell)}\pochhammer[k+1]{a_i}\,\prod\limits_{i=\zeta(\ell)+1}^{r}\pochhammer[k]{a_i}},
%\quad\text{when }n=mk+\ell,
\end{equation}
which means that %can be simplified to obtain
\begin{equation}
F_{\ell}(n,k)
=\frac{\dis(-1)^{rk+\zeta(\ell)}\,n!\,\prod\limits_{i=1}^{\zeta(\ell)}\pochhammer[k+1]{b_{\eta(\ell)}-a_i}\, \prod\limits_{i=\zeta(\ell)+1}^{r}\pochhammer[k]{b_{\eta(\ell)}-a_i+1}} {\dis\pochhammer{b_{\eta(\ell)}+k+1}\;\prod\limits_{i=1}^{\zeta(\ell)}\pochhammer[k+1]{a_i}\,\prod\limits_{i=\zeta(\ell)+1}^{r}\pochhammer[k]{a_i}}.
%\quad\text{when }n=mk+\ell.
\end{equation}
As a result,
\begin{equation}
\label{Th. 6.1 - nonzero orthogonality condition 2}
\Functional{v_{\ell}}{x^k\,P_n}=
\frac{\dis(-1)^{(m-r)k+\ell-\zeta(\ell)}\,n!\,\prod\limits_{i=1}^{r}\pochhammer{a_i}\, \prod\limits_{i=1}^{\zeta(\ell)}\pochhammer[k+1]{b_{\eta(\ell)}-a_i}\,\prod\limits_{i=\zeta(\ell)+1}^{r}\pochhammer[k]{b_{\eta(\ell)}-a_i+1}} {\dis\pochhammer{b_{\eta(\ell)}+k}\;\prod\limits_{j=1}^{\zeta(\ell)}\pochhammer[n+k]{b_j+1}\,\prod\limits_{j=\zeta(\ell)+1}^{s}\pochhammer[n+k]{b_j}}\neq 0.
\end{equation}
Hence, we have proved that \eqref{orthogonality conditions w.r.t. linear functionals} always holds.
%Note that $(m-r)k+\ell-\zeta(\ell)=0$ when $r\geq s$.
\end{proof}

\subsection{Differential properties}
In the following result, we write the $m$-orthogonal polynomials defined by \eqref{hypergeometric type II MOP - explicit formula as a s+1Fr} as solutions to an ordinary differential equation, as a consequence of their explicit representation as hypergeometric series.
\begin{proposition}
\label{differential equation for type II MOP prop.}
For $r,s\in\N$ and $m=\max(r,s)\geq 1$, let $\seq{P_n(x)}$ be the $m$-orthogonal polynomial sequence defined by \eqref{hypergeometric type II MOP - explicit formula as a s+1Fr}. 
Then, $P_n(x)$ satisfies the $(m+1)$-order differential equation
\begin{equation}
\label{differential equation for generalised hypergeometric polynomials}
\left[\prod_{i=1}^{r}\left(x\,\DiffOp+a_i\right)\right]\DiffOp\,P_n(x)
=\left[\prod_{j=1}^{s}\left(x\,\DiffOp+b_j^{(n-1)}\right)\right]\left(x\,\DiffOp-n\right)P_n(x),
\end{equation}
which can be written in the form
\begin{equation}
\label{differential equation for generalised hypergeometric polynomials expanded}
n\prod\limits_{j=1}^{s}b_j^{(n-1)}\,P_n(x)+\sum_{k=0}^{m}x^k\,\psi_n^{[k]}(x)\,\DiffOpHigherOrder{k+1}\left(P_n(x)\right)=0,
\end{equation}
for some polynomials $\psi_n^{[k]}$, $0\leq k\leq m$, of degree not greater than $1$.
In particular, the coefficient of the highest-order derivative, $x^m\,\psi_n^{[m]}$, is equal to $x^m$ when $s\neq r$ and to $x^m(1-x)$ when $s=r$.
\end{proposition}
\begin{proof}
Applying the differential equation \eqref{generalised hypergeometric differential equation} to \eqref{hypergeometric type II MOP - explicit formula as a s+1Fr}, we derive \eqref{differential equation for generalised hypergeometric polynomials}.
Expanding both sides of \eqref{differential equation for generalised hypergeometric polynomials},
\begin{equation}
\left[\prod_{i=1}^{r}\left(x\,\DiffOp+a_i\right)\right]\DiffOp\,P_n(x)
=\sum_{k=0}^{r}\eta^{[k]}x^{k}\DiffOpHigherOrder{k+1}\,P_n(x)
\end{equation}
and
\begin{equation}
\left[\prod_{j=1}^{s}\left(x\,\DiffOp+b_j^{(n-1)}\right)\right]\left(x\,\DiffOp-n\right)P_n(x)
=-n\prod\limits_{j=1}^{s}b_j^{(n-1)}\,P_n(x)+\sum_{k=0}^{s}\xi_n^{[k]}x^{k+1}\DiffOpHigherOrder{k+1}\,P_n(x),
\end{equation}
with $\eta^{[r]}=1=\xi_n^{[s]}$.
Therefore, defining $\psi_n^{[k]}=\eta^{[k]}-\xi_n^{[k]}\,x$ for $0\leq k\leq m$, we obtain \eqref{differential equation for generalised hypergeometric polynomials expanded}.
\end{proof}

Furthermore, applying the differentiation formula \eqref{derivative of a generalised hypergeometric series} to the $m$-orthogonal polynomials defined by \eqref{hypergeometric type II MOP - explicit formula as a s+1Fr}, we find that the differentiation operator acts on them as a shift on the parameters.
Therefore, these polynomials satisfy Hahn's property, that is, their sequence of derivatives is also $m$-orthogonal.
\begin{proposition}
\label{derivative of the type II MOP as a shift on the parameters prop.}
For $r,s\in\N$ and $m=\max(r,s)\geq 1$, let $\seq{P_n(x)}$ be the $m$-orthogonal polynomial sequence defined by \eqref{hypergeometric type II MOP - explicit formula as a s+1Fr}. Then, when $r\geq s$, we have
\begin{equation}
\label{differentiation as a shift on the parameters r>=s}
\DiffOp\,P_n^{\left[(r,s),\left(\lambda_1,\cdots,\lambda_s\right)\right]}\left(x\left|\begin{matrix}a_1,\cdots,a_r \\ b_1,\cdots,b_s\end{matrix}\right.\right)
=\begin{cases}
n\,P_{n-1}^{\left[(r,s),\left(\lambda_2-1,\cdots,\lambda_s-1,r\right)\right]} \left(x\left|\begin{matrix}a_1+1,\cdots,a_r+1 \\ b_2+1,\cdots,b_s+1,b_1+2 \end{matrix}\right.\right)
& \text{if }\lambda_1=1, 
\\[0,5 cm]
n\,P_{n-1}^{\left[(r,s),\left(\lambda_1-1,\cdots,\lambda_s-1\right)\right]} \left(x\left|\begin{matrix}a_1+1,\cdots,a_r+1 \\ b_1+1,\cdots,b_s+1\end{matrix}\right.\right)
& \text{if }\lambda_1\geq 2,
\end{cases}
\end{equation}
and, when $r\leq s$,
\begin{equation}
\label{differentiation as a shift on the parameters r<=s}
\DiffOp\,P_n^{[(r,s)]}\left(x\left|\begin{matrix}a_1,\cdots,a_r \\ b_1,\cdots,b_s \end{matrix}\right.\right)
=n\,P_{n-1}^{[(r,s)]}\left(x\left|\begin{matrix} a_1+1,\cdots,a_r+1 \vspace*{0,1 cm}\\ b_2+1,\cdots,b_s+1,b_1+2 \end{matrix}\right.\right).
\end{equation}
\end{proposition}

\subsection{Recurrence relation}
We end this section showing explicit expressions for the coefficients of the recurrence relation satisfied by the $m$-orthogonal polynomials defined by \eqref{hypergeometric type II MOP - explicit formula as a s+1Fr}.
Recall that the parameters $a_i^{(k)}$ and $b_j^{(k)}$ in \eqref{definition of the orthogonality functionals} are defined by \eqref{BCF parameters r+1Fs, r>=s}-\eqref{BCF parameters r+1Fs, r<=s} with $a_{r+1}=1$.
Therefore, based on Theorem \ref{recurrence relation for MOP associated with a BCF}, we obtain the following result.
\begin{theorem}
\label{recurrence coefficients for the MOP and link to the BCF for ratios of hypergeometric series}
For $r,s\in\N$ and $m=\max(r,s)\geq 1$, let $\seq{P_n(x)}$ be the $m$-orthogonal polynomial sequence in $R=\Q\left(b_1,\cdots,b_s\right)\left[a_1,\cdots,a_r\right]$ defined by \eqref{hypergeometric type II MOP - explicit formula as a s+1Fr} and let $\mathrm{H}^{(m)}=\seq[n,k\in\N]{h_{n,k}}$ be the production matrix of the generalised $m$-Stieltjes-Rogers polynomials $\generalisedStieltjesRogersPoly{n}{k}{\mathbf{\alpha}}$, where $\alpha=\seq[k\in\N]{\alpha_{k+m}}$ is defined by \eqref{BCF coeff r+1Fs, r>=s}-\eqref{BCF coeff r+1Fs, r<=s} with $a_{r+1}=1$.
Then, $\seq{P_n(x)}$ satisfies the $(m+1)$-order recurrence relation
\begin{equation}
\label{recurrence relation for our m-OPS}
P_{n+1}(x)=x\,P_n(x)-\sum_{k=0}^{m}\gamma_{n-k}^{[k]}\,P_{n-k}(x),
\end{equation}
where, setting $\alpha_j=0$ for $0\leq j\leq m-1$, %for all $n\in\N$ and $0\leq k\leq m$,
\begin{equation}
\label{recurrence coefficients for our m-OPS}
\gamma_n^{[k]}=h_{n+k,n}
=\sum_{m\geq\ell_0>\cdots>\ell_k\geq 0}\,\prod_{j=0}^{k}\alpha_{(m+1)(n+j)+\ell_j}
%=\sum_{\ell_0=k}^{m}\left(\sum_{\ell_1=k-1}^{\ell_0-1}\cdots\sum_{\ell_k=0}^{\ell_{k-1}-1}\left(\prod_{j=0}^{k}\alpha_{(m+1)(n+j)+\ell_j}\right)\right)
\quad\text{for all } n\in\N \text{ and } 0\leq k\leq m.
\end{equation}
\end{theorem}

A consequence of this theorem is that the production matrix $\mathrm{H}^{(m)}$ of the generalised $m$-Stieltjes-Rogers polynomials $\generalisedStieltjesRogersPoly{n}{k}{\mathbf{\alpha}}$, where $\alpha=\seq[k\in\N]{\alpha_{k+m}}$ is defined by \eqref{BCF coeff r+1Fs, r>=s}-\eqref{BCF coeff r+1Fs, r<=s} with $a_{r+1}=1$, is determined by the polynomial sequence $\seq{P_n(x)}$ defined in Theorem \ref{explicit formula as hypergeometric polynomials for the r-OP - theorem}.
Recalling \eqref{hypergeometric type II MOP - explicit formula as a s+1Fr}, $\seq{P_n(x)}$ is invariant under permutations of $\left(a_1,\cdots,a_r\right)$ as well as under different choices of $\left(\sigma_1,\cdots,\sigma_r\right)$ when $r<s$.
However,  $\seq{P_n(x)}$ is not invariant under permutations of $\left(b_1,\cdots,b_s\right)$ or under different choices of $\left(\lambda_1,\cdots,\lambda_s\right)$ when $r>s$.

Therefore, as mentioned at the end of Section \ref{BCF for ratios of hypergeometric series}, the production matrix $\mathrm{H}^{(m)}$ is invariant under permutations of $\left(a_1,\cdots,a_r\right)$ (but not of $\left(b_1,\cdots,b_s\right)$) and under different choices of $\left(\sigma_1,\cdots,\sigma_r\right)$ (but not of $\left(\lambda_1,\cdots,\lambda_s\right)$), in spite of these permutations and choices corresponding to different branched-continued-fraction coefficients.
As a result, permutations of $\left(a_1,\cdots,a_r\right)$ and different choices of $\left(\sigma_1,\cdots,\sigma_r\right)$ lead to different decompositions in bidiagonal matrices of the same production matrix $\mathrm{H}^{(m)}$. 
However, taking $k=(m+1)n$, with $n\in\N$, in \eqref{BCF coeff r+1Fs, r>=s}-\eqref{BCF coeff r+1Fs, r<=s}, we check that $\alpha_{(m+1)n+m}$ is invariant under these permutations and choices, so the $\mathrm{L}\mathrm{U}$-factorisation of $\mathrm{H}^{(m)}$ given by its different decompositions in bidiagonal matrices is the same (which had to be the case because a matrix $\mathrm{L}\mathrm{U}$-factorisation with $\mathrm{L}$ unit-lower-triangular is unique).

%\newpage
\section{Generalised $m$-Stieltjes-Rogers polynomials}
\label{Generalised m-S.-R. poly}
In this section, we present an explicit formula for the generalised $m$-Stieltjes-Rogers polynomials corresponding to the branched continued fractions introduced in Theorem \ref{BCF for ratios of r+1Fs - theorem} with $a_{r+1}=1$, which are the moments of the dual sequence of the $m$-orthogonal polynomials defined in Theorem \ref{explicit formula as hypergeometric polynomials for the r-OP - theorem}.
\begin{theorem}
\label{generalised m-Stieltjes-Rogers poly th.}
For $r,s\in\N$ and $m=\max(r,s)\geq 1$, let $\seq[k\in\N]{\alpha_{k+m}}$ in $R=\Q\left(b_1,\cdots,b_s\right)\left[a_1,\cdots,a_r\right]$ defined by \eqref{BCF coeff r+1Fs, r>=s}-\eqref{BCF coeff r+1Fs, r<=s} with $a_{r+1}=1$. 
The generalised $m$-Stieltjes-Rogers polynomials with weights $\seq[k\in\N]{\alpha_{k+m}}$ are
\begin{equation}	
\label{generalised m-Stieltjes-Rogers poly formula general (r,s)}
\generalisedStieltjesRogersPoly{n}{k}{\mathbf{\alpha}}
=\binom{n}{k}\frac{\prod\limits_{i=1}^{r}\pochhammer[n-k]{a_i+k}}{\prod\limits_{j=1}^{s}\pochhammer[n-k]{b_j^{(k)}+k}}
%=\binom{n}{k}\frac{\prod\limits_{i=1}^{r}\pochhammer[n-k]{a_i+k}}{\prod\limits_{j=1}^{s}\pochhammer[n-k]{b_j+\ceil{\frac{k+1-\lambda_j}{m}}+k}}.
\quad\text{with}\quad b_j^{(k)}=b_j+\ceil{\frac{k+1-\lambda_j}{m}}.
\end{equation}
\end{theorem}
The case $s=0$ and $m=r=2$ of this theorem corresponds to \cite[Prop.~7.1]{AlanSokalMOPd-opProdMatBCF}.

\begin{proof}
Let $\mathrm{S}=\seq[n,k\in\N]{s_{n,k}}$ and $\mathrm{T}=\seq[n,k\in\N]{t_{n,k}}$ be the unit-lower-triangular matrices with entries
\begin{equation}
s_{n,k}=\binom{n}{k}\frac{\prod\limits_{i=1}^{r}\pochhammer[n-k]{a_i+k}}{\prod\limits_{j=1}^{s}\pochhammer[n-k]{b_j^{(k)}+k}}
\quad\text{and}\quad
t_{n,k}=(-1)^{n-k}\binom{n}{k}\frac{\prod\limits_{i=1}^{r}\pochhammer[n-k]{a_i+k}}{\prod\limits_{j=1}^{s}\pochhammer[n-k]{b_j^{(n-1)}+k}}
\quad\text{when }n\geq k.
\end{equation}
Reverting the order of summation in \eqref{hypergeometric type II MOP - explicit expression as a combination of powers of x}, we rewrite the polynomials defined in Theorem \ref{explicit formula as hypergeometric polynomials for the r-OP - theorem} as
\begin{equation}
P_n(x)=\sum_{k=0}^{n}t_{n,k}\,x^k
\quad\text{for all }n\in\N.
\end{equation}
Hence, $\mathrm{T}$ is the coefficient matrix of $\seq{P_n(x)}$. %the polynomial sequence defined by \eqref{hypergeometric type II MOP - explicit formula as a s+1Fr}. 
Moreover, based on Proposition \ref{PolySeqAndHessMatrixProp}, $\mathrm{S}^{(m)}=\seq[n,k\in\N]{\generalisedStieltjesRogersPoly{n}{k}{\mathbf{\alpha}}}$ is the moment matrix of the dual sequence of $\seq{P_n(x)}$.
Therefore, $\mathrm{S}^{(m)}$ is the inverse matrix of $\mathrm{T}$, so $\mathrm{S}^{(m)}=\mathrm{S}$ if and only if
\begin{equation}
\label{entry of the product TS formula intended}
\left(\mathrm{TS}\right)_{n,k}
=\sum_{l=k}^{n}t_{n,l}\,s_{l,k}
=\begin{cases}
1 & \text{if } n=k \\
0 & \text{if } n\neq k.
\end{cases}
\end{equation}
It is clear that $\left(\mathrm{TS}\right)_{n,k}=0$ whenever $n<k$ and $\left(\mathrm{TS}\right)_{n,n}=t_{n,n}\,s_{n,n}=1$ for any $n\in\N$.
For $n>k$,
\begin{equation}
\label{entry of the product TS formula 1}
\sum_{l=k}^{n}t_{n,l}\,s_{l,k}
=\sum_{l=k}^{n}(-1)^{n-l}\binom{n}{l}\binom{l}{k}\frac{\prod\limits_{i=1}^{r}\pochhammer[l-k]{a_i+k}\pochhammer[n-l]{a_i+l}} {\prod\limits_{j=1}^{s}\pochhammer[l-k]{b_j^{(k)}+k}\pochhammer[n-l]{b_j^{(n-1)}+l}}.
\end{equation}

Furthermore, for any $k<l<n$, we have
\begin{equation}
\binom{n}{l}\binom{l}{k}=\binom{n}{k}\binom{n-k}{l-k},
\;
\pochhammer[l-k]{a_i+k}\pochhammer[n-l]{a_i+l}=\pochhammer[n-k]{a_i+k},
\;\text{and}\;
\pochhammer[n-l]{b_j^{(n-1)}+l}=\frac{\pochhammer[n-k]{b_j^{(n-1)}+k}}{\pochhammer[l-k]{b_j^{(n-1)}+k}}.
\end{equation}

Applying the formulas above to \eqref{entry of the product TS formula 1} and making the change of variable $\ell=l-k$, we find that
\begin{equation}
\sum_{l=k}^{n}t_{n,l}\,s_{l,k}
=(-1)^{n-k}\binom{n}{k}\,\frac{\prod\limits_{i=1}^{r}\pochhammer[n-k]{a_i+k}}{\prod\limits_{j=1}^{s}\pochhammer[n-k]{b_j^{(n-1)}+k}}
\sum\limits_{\ell=0}^{n-k}(-1)^{\ell}\binom{n-k}{\ell}\prod\limits_{j=1}^{s}\frac{\pochhammer[\ell]{b_j^{(n-1)}+k}}{\pochhammer[\ell]{b_j^{(k)}+k}},
\end{equation}
which is equivalent to
\begin{equation}
\sum_{l=k}^{n}t_{n,l}\,s_{l,k}
=(-1)^{n-k}\binom{n}{k}\,\frac{\prod\limits_{i=1}^{r}\pochhammer[n-k]{a_i+k}}{\prod\limits_{j=1}^{s}\pochhammer[n-k]{b_j^{(n-1)}+k}}\,
\Hypergeometric[1]{s+1}{s}{-(n-k),b_1^{(n-1)}+k,\cdots,b_s^{(n-1)}+k \vspace*{0,1 cm}}{b_1^{(k)}+k,\cdots,b_s^{(k)}+k}.
\end{equation}
When $n>k$, $b_j^{(n-1)}-b_j^{(k)}$ is a nonnegative integer for any $1\leq j\leq s$ and
\begin{equation}
\sum_{j=1}^{s}\left(b_j^{(n-1)}-b_j^{(k)}\right)
=\sum_{j=1}^{s}\left(\ceil{\frac{n-\lambda_j}{m}}-\ceil{\dfrac{k+1-\lambda_j}{m}}\right)
\leq\sum_{i=1}^{m}\left(\ceil{\frac{n-i}{m}}-\ceil{\dfrac{k+1-i}{m}}\right)
=n-k-1.
\end{equation}
Therefore, recalling \eqref{hypergeometric formula 1}, we have
\begin{equation}
\Hypergeometric[1]{s+1}{s}{-(n-k),b_1^{(n-1)}+k,\cdots,b_s^{(n-1)}+k \vspace*{0,1 cm}}{b_1^{(k)}+k,\cdots,b_s^{(k)}+k}=0
\quad\text{for  }n>k.
\end{equation}
As a result, we conclude that \eqref{entry of the product TS formula intended} and, consequently, \eqref{generalised m-Stieltjes-Rogers poly formula general (r,s)} hold.
\end{proof}

The row-generating polynomials of the matrix $\mathrm{S}^{(m)}=\seq[n,k\in\N]{\generalisedStieltjesRogersPoly{n}{k}{\mathbf{\alpha}}}$ determined by \eqref{generalised m-Stieltjes-Rogers poly formula general (r,s)} are %$\dis\sum_{k=0}^{n}\left(\generalisedStieltjesRogersPoly{n}{k}{\mathbf{\alpha}}\,x^k\right)$. 
\begin{equation}
\label{row-generating poly generalised m-S.R. poly}
\sum_{k=0}^{n}\generalisedStieltjesRogersPoly{n}{k}{\mathbf{\alpha}}\,x^k
=\sum_{k=0}^{n}\binom{n}{k}\frac{\prod\limits_{i=1}^{r}\pochhammer[n-k]{a_i+k}}{\prod\limits_{j=1}^{s}\pochhammer[n-k]{b_j^{(k)}+k}}\,x^k.
\end{equation} 
When $s=0$, $\mathrm{S}^{(m)}$ is its own unsigned inverse, which means that the inverse of $\mathrm{S}^{(m)}$ is $\seq[n,k\in\N]{(-1)^{n-k}\,\generalisedStieltjesRogersPoly{n}{k}{\mathbf{\alpha}}}$, and the row-generating polynomials of $\mathrm{S}^{(m)}$ are
\begin{equation}
\label{hypergeometric type II MOP as row-generating poly}
\sum_{k=0}^{n}\binom{n}{k}\prod_{i=1}^{m}\pochhammer[n-k]{a_i+k}\,x^k=(-1)^n\,P_n(-x).
\end{equation}
That is not the case when $s\geq 1$, because then the entries of $\mathrm{S}^{(m)}$ have terms $b_j^{(k)}$, while the entries of its inverse $\mathrm{T}=\seq[n,k\in\N]{t_{n,k}}$ have terms $b_j^{(n-1)}$ instead.
In that case, it is not clear what are the row-generating polynomials \eqref{row-generating poly generalised m-S.R. poly}, but we can determine the column-generating series of $\mathrm{S}^{(m)}$:
\begin{equation}
\sum_{n=k}^{\infty}\generalisedStieltjesRogersPoly{n}{k}{\mathbf{\alpha}}\,x^{n-k}
=\sum_{n=k}^{\infty}\binom{n}{k}\frac{\prod\limits_{i=1}^{r}\pochhammer[n-k]{a_i+k}}{\prod\limits_{j=1}^{s}\pochhammer[n-k]{b_j^{(k)}+k}}\,x^{n-k}
=\Hypergeometric{r+1}{s}{a_1+k,\cdots,a_r+k,1+k \vspace*{0,1 cm}}{b_1^{(k)}+k,\cdots,b_s^{(k)}+k}.
\end{equation}

%\newpage
\section{Multiple orthogonal polynomials with respect to Meijer G-functions}
\label{Type II MOP w.r.t. Meijer G-functions}
In this section, we focus the analysis of the multiple orthogonal polynomials introduced in Section \ref{Type II MOP w.r.t. linear functionals} to the cases where the branched-continued-fraction coefficients defined by \eqref{BCF coeff r+1Fs, r>=s} with $a_{r+1}=1$ are all positive.
The positivity of the branched-continued-fraction coefficients implies that the zeros of the corresponding multiple orthogonal polynomials are all simple, real, and positive, the zeros of consecutive polynomials interlace, and the recurrence coefficients are all positive.
Furthermore, it is a sufficient (but far from necessary) condition for the linear functionals of orthogonality in Theorem \ref{explicit formula as hypergeometric polynomials for the r-OP - theorem} to be induced by measures on the positive real line whose densities are Meijer G-functions (see Theorem \ref{MOP wrt Meijer G-functions th.}).

We find the asymptotic behaviour of the recurrence coefficients (Proposition \ref{asymptotic behaviour of the recurrence coefficients r-OPS}) and we present a Mehler-Heine-type asymptotic formula near the origin (Proposition \ref{Mehler-Heine formula r-OPS prop.}).
Then, we use these two results to find an upper bound for the largest zero as well as the asymptotic behaviour of the zeros near the origin (Theorem \ref{results about the zeros of the r-OP}).
Then, we focus our analysis on two special cases with $s=r$: a $r$-orthogonal polynomial sequence with constant recurrence coefficients and particular instances of the Jacobi-Pi\~neiro polynomials (see Subsections \ref{r-OP w/ constant rec coeff} and \ref{Jacobi-Pineiro poly}, respectively). 
Finally, we use the connection with the Jacobi-Pi\~neiro polynomials to find the asymptotic zero distribution of the polynomials under analysis here when $s=r$ (Theorem \ref{asymptotic zero distribution, s=r, th.}).

\subsection{Positivity of branched-continued-fraction coefficients and Meijer G-functions}
When $a_{r+1}=1$, the condition \eqref{BCF conditions for non-negativity of the coefficients 2} is trivial and the same is true for the condition \eqref{BCF conditions for non-negativity of the coefficients 1} with $i=r+1$, because it reduces to $b_j\geq 0$ for all $1\leq j\leq s$.
Therefore, the coefficients in $\seq[k\in\N]{\alpha_{k+r}}$ defined by \eqref{BCF coeff r+1Fs, r>=s}, with $a_{r+1}=1$ and $a_1,\cdots,a_r,b_1,\cdots,b_s\in\R^+$, are all positive if and only if 
%\begin{subequations}
\begin{equation}
\label{BCF conditions for positivity of the coefficients, a_(r+1)=1}
b_j>a_i-\ceil{\frac{i-\lambda_j}{r}}=
\begin{cases}
a_i &\text{if }i\leq\lambda_j \\
a_i-1 &\text{if }i\geq\lambda_j+1 
\end{cases}
\quad\text{for all }1\leq i\leq r\text{ and }1\leq j\leq s.
\end{equation}

Recalling Remark \ref{modified m-S.R. poly as moments of positive measures}, these conditions imply that the sequence of modified $r$-Stieltjes-Rogers polynomials $\seq{\modifiedStieltjesRogersPoly[r]{n}{k}{\mathbf{\alpha}}}$, which accordingly to Corollary \ref{BCF for ratios of r+1Fs - corollary a_(r+1)=1} is the moment sequence in \eqref{definition of the orthogonality functionals}, is a Stieltjes moment sequence for any $0\leq k\leq r$.
Furthermore, these conditions imply that $b_j>a_{\lambda_j}$ and, consequently, $b_j^{(k)}>a_{\lambda_j}^{(k)}$ for any $1\leq j\leq s$ and $0\leq k\leq r-1$.
Therefore, the moment sequence in \eqref{definition of the orthogonality functionals} with $r\geq s$ and $a_1,\cdots,a_r,b_1,\cdots,b_s\in\R^+$ satisfying \eqref{BCF conditions for positivity of the coefficients, a_(r+1)=1} is the entrywise product of $r$ Stieltjes moment sequences 
\begin{equation}
m_n^{(k)}=\frac{\pochhammer{a_1^{(k)}}\cdots\pochhammer{a_r^{(k)}}}{\pochhammer{b_1^{(k)}}\cdots\pochhammer{b_s^{(k)}}}
=\prod_{j=1}^{s}\frac{\pochhammer{a_{\lambda_j}^{(k)}}}{\pochhammer{b_j^{(k)}}}\prod_{i\in\mathrm{I}}\pochhammer{a_i^{(k)}}
\quad\text{with   }\mathrm{I}=\{1,\cdots,r\}\backslash\{\lambda_1,\cdots,\lambda_s\}.
\end{equation}
This decomposition gives an alternative proof that the conditions in \eqref{BCF conditions for positivity of the coefficients, a_(r+1)=1} are sufficient for $\seq{m_n^{(k)}}$ to be a Stieltjes moment sequence for any $0\leq k\leq r-1$, because an entrywise product of Stieltjes moment sequences is also a Stieltjes moment sequence.
However, these conditions are not necessary; see \cite[\S 2]{KarpPrilepkina2016} for more information about sharper sufficient conditions for $m_n^{(k)}$ to be a Stieltjes moment sequence.

The Meijer G-function $G^{\,r,0}_{s,r}$ (see \cite{LukeSpecialFunctionsVolI, DLMF} for more details) is defined by the Mellin-Barnes type integral
\begin{equation}
\label{Meijer G-function definition}
\MeijerG{r,0}{s,r}{b_1,\cdots,b_s}{a_1,\cdots,a_r}
=\frac{1}{2\pi i}\int\limits_{c-i\infty}^{c+i\infty} \frac{\Gamma\left(a_1+u\right)\cdots\Gamma\left(a_r+u\right)}{\Gamma\left(b_1+u\right)\cdots\Gamma\left(b_s+u\right)}\,x^{-u}\mathrm{d}u,
\quad c>-\min\limits_{1\leq i\leq r}\{\Real(a_i)\}.
\end{equation}
%with $c>-\min\limits_{1\leq i\leq r}\{\Real(a_i)\}$.

%According to \cite[Eq.~16.21.1]{DLMF}, the Meijer G-function in \eqref{Meijer G-function definition} is a solution of the ordinary differential equation or order $\max(r,s)$
%\begin{equation}
%\label{Meijer G-function ODE}
%\left[\mathlarger{\prod}_{i=1}^{r}\left(x\,\DiffOp-a_i\right)-(-1)^{r-s}\,x\,\mathlarger{\prod}_{j=1}^{s}\left(x\,\DiffOp-b_j+1\right)\right]G(x)=0.
%\end{equation}
%
When it exists, the Mellin transform of the Meijer G-function $G(x)$ is equal to the ratio of gamma functions in the integrand on the right-hand side of \eqref{Meijer G-function definition}.
For $r,s\in\N$ with $r\geq 1$ and $r\geq s$, let $a_1,\cdots,a_r,b_1,\cdots,b_s\in\C$.
Then, based on \cite[Eq.~2.24.2.1]{PrudnikovEtAlVol3},
\begin{equation}
\label{Mellin transform of a Meijer G-function}
\int\limits_{0}^{\infty}\MeijerG{r,0}{s,r}{b_1,\cdots,b_s}{a_1,\cdots,a_r}x^{z-1}\,\mathrm{d}x =\frac{\Gamma\left(a_1+z\right)\cdots\Gamma\left(a_r+z\right)}{\Gamma\left(b_1+z\right)\cdots\Gamma\left(b_s+z\right)},
\end{equation}
for any $z\in\C$ such that $\Real(z)>-\min\limits_{1\leq i\leq r}\{\Real\left(a_i\right)\}$.
In particular, if $\Real(a_i)>0$ for all $1\leq i\leq r$, we have
\begin{equation}
\label{moment sequence ratio of Pochhammers Meijer G-function}
\frac{\Gamma(b_1)\cdots\Gamma(b_s)}{\Gamma(a_1)\cdots\Gamma(a_r)}\,\int\limits_{0}^{\infty}\MeijerG{r,0}{s,r}{b_1,\cdots,b_s}{a_1,\cdots,a_r}x^{n-1}\,\mathrm{d}x
=\frac{\pochhammer{a_1}\cdots\pochhammer{a_r}}{\pochhammer{b_1}\cdots\pochhammer{b_s}}
\quad \text{for all } n\in\N.
\end{equation}

We assume now that $s=r$ and $a_1,\cdots,a_r,b_1,\cdots,b_r\in\R^+$ satisfy \eqref{BCF conditions for positivity of the coefficients, a_(r+1)=1}.
In particular, $b_i>a_i$ for all $1\leq i\leq r$, which implies that $\sum\limits_{i=1}^{r}\left(b_i-a_i\right)>0$. 
Therefore, based on \cite[Lemma~1]{KarpPrilepkinaHypergeometricFunctions},
%$\dis\MeijerG{r,0}{r,r}{b_1,\cdots,b_r}{a_1,\cdots,a_r}=0$ for all $x>1$.
\begin{equation}
\label{MeijerGfunctionEqualTo0Forx>1}
\MeijerG{r,0}{r,r}{b_1,\cdots,b_r}{a_1,\cdots,a_r}=0
\quad\text{for all  }x>1.
\end{equation}
As a result, the integration in \eqref{moment sequence ratio of Pochhammers Meijer G-function} is over the interval $(0,1)$ instead of the whole positive real line. 

Considering \eqref{moment sequence ratio of Pochhammers Meijer G-function}-\eqref{MeijerGfunctionEqualTo0Forx>1}, the following result is the special case of Theorem \ref{explicit formula as hypergeometric polynomials for the r-OP - theorem} obtained when the parameters satisfy the conditions for positivity of the corresponding branched-continued-fraction coefficients.
\begin{theorem}
\label{MOP wrt Meijer G-functions th.}
For $r,s\in\N$ such that $s\leq r\neq 0$, let $1\leq\lambda_1<\cdots<\lambda_s\leq r$, $a_1,\cdots,a_r,b_1,\cdots,b_s\in\R^+$ satisfying \eqref{BCF conditions for positivity of the coefficients, a_(r+1)=1}, and $\left(\mu_0,\cdots,\mu_{r-1}\right)$ the vector of measures supported on the whole positive real line if $s<r$ or on the interval $(0,1)$ if $s=r$ with densities %such that, for $\nu\in\{0,\cdots,r-1\}$,
\begin{equation}
\label{orthogonality measures involving Meijer G-functions}
\mathrm{d}\mu_k(x)=
%\frac{\Gamma\left(b_1^{(k)}\right)\cdots\Gamma\left(b_s^{(k)}\right)}{\Gamma\left(a_1^{(k)}\right)\cdots\Gamma\left(a_r^{(k)}\right)}\, 
\MeijerG{r,0}{s,r}{b_1^{(k)},\cdots,b_s^{(k)}\vspace*{0,1 cm}}{a_1^{(k)},\cdots,a_r^{(k)}}\,\frac{\mathrm{d}x}{x}
=\MeijerG{r,0}{s,r}{b_1^{(k)}-1,\cdots,b_s^{(k)}-1\vspace*{0,1 cm}}{a_1^{(k)}-1,\cdots,a_r^{(k)}-1}\mathrm{d}x
\quad\text{for   }
k\in\{0,\cdots,r-1\},
\end{equation}
where
\begin{equation}
\label{parameters of the orthogonality measures involving Meijer G-functions}
a_i^{(k)}=%a_i+\ceil{\dfrac{k+1-i}{r}}=
\begin{cases}
a_i+1 & \text{if }1\leq i\leq k \\
a_i & \text{if }k+1\leq i\leq r
\end{cases}
\quad\text{and}\quad
b_j^{(k)}=%b_j+\ceil{\dfrac{k+1-j}{s}}=
\begin{cases}
b_j+1 & \text{if }1\leq\lambda_j\leq k \\
b_j & \text{if }k+1\leq\lambda_j\leq r.
\end{cases}
\end{equation}

Then, the $r$-orthogonal polynomial sequence $\seq{P_n(x)}$ with respect to $\left(\mu_0,\cdots,\mu_{r-1}\right)$ is given by
%\begin{subequations}
\begin{equation}
\label{hypergeometric r-OP - explicit formula as a s+1Fr, r>=s}
P_n(x)%=P_n^{\left[(r,s),\left(\lambda_1,\cdots,\lambda_s\right)\right]}\left(x\left|\begin{matrix}a_1,\cdots,a_r \\ b_1,\cdots,b_s\end{matrix}\right.\right)
=\frac{(-1)^n\pochhammer{a_1}\cdots\pochhammer{a_r}}{\pochhammer{b_1^{(n-1)}}\cdots\pochhammer{b_s^{(n-1)}}} \,\Hypergeometric{s+1}{r}{-n,b_1^{(n-1)},\cdots,b_s^{(n-1)}}{a_1,\cdots,a_r}
\quad\text{with}\quad
b_j^{(n-1)}=b_j+\ceil{\dfrac{n-\lambda_j}{r}}.
\end{equation}
%with
%\begin{equation}
%\label{b_j parameters in the type II MOP wrt Meijer G-functions}
%b_j^{(n-1)}=b_j+\ceil{\dfrac{n-\lambda_j}{r}}.
%\end{equation}
\end{theorem}
See \cite[Eq.~5.4.4]{LukeSpecialFunctionsVolI} for the equality of the two densities involving Meijer G-functions in \eqref{orthogonality measures involving Meijer G-functions}. 

Recall that when $a_1,\cdots,a_r,b_1,\cdots,b_s\in\R^+$ satisfy the conditions \eqref{BCF conditions for positivity of the coefficients, a_(r+1)=1}, the coefficients $\mathbf{\alpha}=\seq[k\in\N]{\alpha_{k+m}}$ defined by \eqref{BCF coeff r+1Fs, r>=s} with $a_{r+1}=1$ are all positive.
Therefore, based on Theorem \ref{location and interlacing of the zeros, general case with positive BCF coeff}, the zeros of the corresponding $r$-orthogonal polynomials $\seq{P_n(x)}$ determined by \eqref{hypergeometric r-OP - explicit formula as a s+1Fr, r>=s} are all simple, real, and positive, and the zeros of consecutive polynomials interlace.
Furthermore, when $s=r$, the orthogonality measures are all supported on the interval $(0,1)$, so it is natural to conjecture that the zeros of $\seq{P_n(x)}$ are all located on that interval.
This conjecture is clearly true when $r=1$ and $\seq{P_n(x)}$ are the Jacobi polynomials orthogonal with respect to the positive measure $\mu$ supported on the interval $(0,1)$ with density $\mathrm{d}\mu(x)=x^{a-1}(1-x)^{b-a-1}\,\mathrm{d}x$ for $a,b\in\R^+$ such that $a<b$; it is also true for $r=2$, because the orthogonality measures form a Nikishin system on the interval $(0,1)$ (see \cite[Th.~1]{PaperHypergeometricWeights}).
In fact, we show that this conjecture is true for any positive integer $r$, as we give the corresponding asymptotic zero distribution in Theorem \ref{asymptotic zero distribution, s=r, th.}.

When $r=1$ and $s\leq r$ (that is, $s\in\{0,1\}$), the conditions \eqref{BCF conditions for positivity of the coefficients, a_(r+1)=1} for positivity of the branched-continued-fraction coefficients defined by \eqref{BCF coeff r+1Fs, r>=s} with $a_{r+1}=1$ correspond to the necessary and sufficient conditions for positivity of the orthogonality measure of the corresponding Laguerre and Jacobi polynomials, respectively.
Moreover, when $r=2$ and $s\leq r$ (that is, $s\in\{0,1,2\}$), the conditions in \eqref{BCF conditions for positivity of the coefficients, a_(r+1)=1} correspond to the conditions for the orthogonality measures defined by \eqref{orthogonality measures involving Meijer G-functions}-\eqref{parameters of the orthogonality measures involving Meijer G-functions} to form a Nikishin system.
It would be interesting to find out whether the conditions in \eqref{BCF conditions for positivity of the coefficients, a_(r+1)=1} also imply that the orthogonality measures defined by \eqref{orthogonality measures involving Meijer G-functions}-\eqref{parameters of the orthogonality measures involving Meijer G-functions} form a Nikishin system when $r\geq 3$.

\subsection{Recurrence coefficients and location of the zeros}
Throughout the rest of this section, we use the notation $[k]_r$, with $k,r\in\Z$ and $r\geq 1$, for the unique element of $\{1,\cdots,r\}$ congruent with $k\hspace*{-0,2 cm}\mod r$, as we have done in Section \ref{BCF for ratios of hypergeometric series}. 
When $r\geq s$ and $a_1,\cdots,a_r,b_1,\cdots,b_s\in\R^+$ satisfy \eqref{BCF conditions for positivity of the coefficients, a_(r+1)=1}, the coefficients defined by \eqref{BCF coeff r+1Fs, r>=s} have asymptotic behaviour
\begin{equation}
	\label{BCF coeff first m+1Fs asymptotic behaviour}
	\alpha_{k+r}%^{[(r,s),\left(\lambda_1,\cdots,\lambda_s\right)]}
	\sim
	\begin{cases}
		\dfrac{r^s}{(r+1)^r}\,k^{r-s} & \text{if } [k]_r\not\in\{\lambda_1,\cdots,\lambda_s\}, \vspace*{0,2 cm} \\
		\dfrac{r^s}{(r+1)^{r+1}}\,k^{r-s} & \text{if } [k]_r\in\{\lambda_1,\cdots,\lambda_s\},
	\end{cases}
	\quad\text{as }k\to\infty,
\end{equation}
which implies that
\begin{equation}
	\label{BCF coeff first m+1Fs asymptotic behaviour*}
	\alpha_{(r+1)(n+j)+\ell_j}%^{[(r,s),\left(\lambda_1,\cdots,\lambda_s\right)]}
	\sim
	\begin{cases}
		\left(\dfrac{r}{r+1}\right)^s\,n^{r-s} & \text{if } [n+j+\ell_j]_r\not\in\{\lambda_1,\cdots,\lambda_s\}, \vspace*{0,2 cm} \\
		\dfrac{r^s}{(r+1)^{s+1}}\,n^{r-s} & \text{if } [n+j+\ell_j]_r\in\{\lambda_1,\cdots,\lambda_s\},
	\end{cases}
	\quad\text{as }n\to\infty.
\end{equation}

Therefore, recalling Theorem \ref{recurrence coefficients for the MOP and link to the BCF for ratios of hypergeometric series}, we obtain the following result.
\begin{proposition}
\label{asymptotic behaviour of the recurrence coefficients r-OPS}
For $r,s\in\N$ such that $s\leq r\neq 0$, $1\leq\lambda_1<\cdots<\lambda_s\leq r$, and $a_1,\cdots,a_r,b_1,\cdots,b_s\in\R^+$ satisfying \eqref{BCF conditions for positivity of the coefficients, a_(r+1)=1}, let $\seq{P_n(x)}$ be the $r$-orthogonal polynomial sequence defined by \eqref{hypergeometric r-OP - explicit formula as a s+1Fr, r>=s}.
Then, $\seq{P_n(x)}$ satisfies the recurrence relation
\begin{equation}
\label{recurrence relation for a r-OPS}
P_{n+1}(x)=x\,P_n(x)-\sum_{k=0}^{r}\gamma_{n-k}^{\,[k]}\,P_{n-k}(x),
\end{equation}
where the recurrence coefficients are given by \eqref{recurrence coefficients for our m-OPS}, are all positive, and have asymptotic behaviour
\begin{equation}
\label{recurrence coefficients r-OPS asymptotic behaviour, s<=r}
\gamma_n^{[k]}\sim C_{[n]_r}^{[k]}\,n^{(r-s)(k+1)}
\quad\text{as }n\to\infty,
\end{equation}
where, for any $1\leq j\leq r$,
\begin{equation}
C_j^{[k]}=\left(\frac{r}{r+1}\right)^{s(k+1)}\,\mathlarger{\sum}_{r\geq\ell_0>\cdots>\ell_k\geq 0}(r+1)^{-f_j\left(\ell_0,\cdots,\ell_k\right)},
\end{equation}
with %$\dis f_j\left(\ell_0,\cdots,\ell_k\right)=\big|\left\{0\leq i\leq k\left|\left[j+\ell_i+i\right]_r\in\{\lambda_1,\cdots,\lambda_s\}\right.\right\}\big|$.
\begin{equation}
f_j\left(\ell_0,\cdots,\ell_k\right)=\big|\left\{0\leq i\leq k\left|\left[j+\ell_i+i\right]_r\in\{\lambda_1,\cdots,\lambda_s\}\right.\right\}\big|.
\end{equation}
\end{proposition}

When $s=0$, $\{\lambda_1,\cdots,\lambda_s\}=\emptyset$ and we have (cf. \cite[Lemma~4.3]{KuijlaarsZhang14})
\begin{equation}
\label{recurrence coefficients r-OPS asymptotic behaviour, s=0}
\gamma_n^{[k]}\sim\binom{r+1}{k+1}\,n^{r(k+1)}
\quad\text{as }n\to\infty,
\quad\text{for any }0\leq k\leq r.
\end{equation}
When $s=r$, $\{\lambda_1,\cdots,\lambda_s\}=\{1,\cdots,r\}$ and, as a result, we find that
\begin{equation}
\label{recurrence coefficients r-OPS asymptotic behaviour, s=r}
\gamma_n^{[k]}\to\binom{r+1}{k+1}\left(\frac{r^r}{(r+1)^{r+1}}\right)^{k+1}
\quad\text{as }n\to\infty,
\quad\text{for any }0\leq k\leq r.
\end{equation}
When $0<s<r$, the recurrence coefficients have an asymptotic behaviour of period $r$.
The simplest example of this periodic asymptotic behaviour is the case $(r,s)=(2,1)$, for which the recurrence coefficients have an asymptotic behaviour of period $2$ given in \cite[Th.~3.4]{PaperTricomiWeights}.

Combining the periodic asymptotic behaviour of the recurrence coefficients in \eqref{recurrence coefficients r-OPS asymptotic behaviour, s<=r} with Theorem \ref{asymptotic behaviour of the largest zero - general theorem for d-OPS}, we obtain an upper bound for the largest zeros of the $r$-orthogonal polynomials $\seq{P_n(x)}$ defined by \eqref{hypergeometric r-OP - explicit formula as a s+1Fr, r>=s}.
Moreover, we can relate the asymptotic behaviour of the zeros near the origin with the location of the zeros of the hypergeometric function $\dis\HypergeometricOneLine[-z]{0}{r}{-}{a_1,\cdots,a_r}$, which are in infinite number and are all real and positive 
(see \cite[\S 4]{WalterMehlerHeineAsymptoticsForMOPs}), as a consequence of the following Mehler-Heine-type asymptotic formula. %satisfied by $\seq{P_n(x)}$.
\begin{proposition}
\label{Mehler-Heine formula r-OPS prop.}
For $r,s\in\N$ such that $s\leq r\neq 0$, $1\leq\lambda_1<\cdots<\lambda_s\leq r$, and $a_1,\cdots,a_r,b_1,\cdots,b_s\in\R^+$ satisfying \eqref{BCF conditions for positivity of the coefficients, a_(r+1)=1}, let $\seq{P_n(x)}$ be the $r$-orthogonal polynomial sequence defined by \eqref{hypergeometric r-OP - explicit formula as a s+1Fr, r>=s}.
Then, 
\begin{equation}
\label{Mehler-Heine formula r-OPS}
\lim_{n\to\infty}\frac{(-1)^n\pochhammer{b_1^{(n-1)}}\cdots\pochhammer{b_s^{(n-1)}}}{\pochhammer{a_1}\cdots\pochhammer{a_r}}\,P_n\left(\frac{z}{n^{s+1}}\right)
=\Hypergeometric[-\frac{z}{r^s}]{0}{r}{-}{a_1,\cdots,a_r}.
\end{equation}
uniformly on compact subsets of $\C$.
\end{proposition}

\begin{proof}
Recalling \eqref{hypergeometric r-OP - explicit formula as a s+1Fr, r>=s},
\begin{equation}
\label{Mehler-Heine formula LHS}
\frac{(-1)^n\pochhammer{b_1^{(n-1)}}\cdots\pochhammer{b_s^{(n-1)}}}{\pochhammer{a_1}\cdots\pochhammer{a_r}}\,P_n\left(\frac{z}{n^{s+1}}\right)
=\Hypergeometric[\frac{z}{n^{s+1}}]{s+1}{r}{-n,b_1^{(n-1)},\cdots,b_s^{(n-1)}}{a_1,\cdots,a_r}.
\end{equation}
Furthermore, observe that $b_j^{(n-1)}\sim\dfrac{n}{r}$ for each $1\leq j\leq s$.
Hence, successively applying the first confluent relation in \eqref{confluent relations for generalised hypergeometric series} to the formula above, we obtain \eqref{Mehler-Heine formula r-OPS}.
\end{proof}

Our results on the location of the zeros of $\seq{P_n(x)}$ are summarised in the following theorem.
\begin{theorem}
\label{results about the zeros of the r-OP}
For $r,s\in\N$ such that $s\leq r\neq 0$, $1\leq\lambda_1<\cdots<\lambda_s\leq r$, and $a_1,\cdots,a_r,b_1,\cdots,b_s\in\R^+$ satisfying \eqref{BCF conditions for positivity of the coefficients, a_(r+1)=1}, let $\seq{P_n(x)}$ be the $r$-orthogonal polynomial sequence defined by \eqref{hypergeometric r-OP - explicit formula as a s+1Fr, r>=s}.
Then:
\begin{itemize}[leftmargin=*]
\item
all the zeros of $P_n(x)$ are simple, real, and positive, and the zeros of consecutive polynomials interlace;

\item
if we denote the zeros of $P_n(x)$ and $\dis\HypergeometricOneLine[-z]{0}{r}{-}{a_1,\cdots,a_r}$ in increasing order, by $\left(x_k^{(n)}\right)_{k=1}^{n}$ and $\left(f_k\right)_{k=1}^{\infty}$, respectively, we have
\begin{equation}
\label{behaviour of the zeros near the origin}
\lim_{n\to\infty}n^{s+1}\,x_k^{(n)}=r^s\,f_k
\quad\text{for all } k\geq 1;
\end{equation}
 
\item
there exists a constant $\mathrm{K}(r,s)\in\R^+$ such that the largest zero of $P_n(x)$ satisfies
\begin{equation}
\label{upper bound largest zero r-OP}
x_n^{(n)}<\mathrm{K}(r,s)\,n^{r-s}+o\left(n^{r-s}\right)
\quad\text{as }n\to\infty.
\end{equation}
\end{itemize}
\end{theorem}

%\textbf{If we could use Theorem \ref{asymptotic behaviour of the largest zero - general theorem for d-OPS} when $s=r$ it would give the desired upper bound, $1$. This is more evidence in favour of the conjecture that the zeros are all on the interval $(0,1)$ when $s=r$.}
When $s=r$, \eqref{upper bound largest zero r-OP} is equivalent to say that the zeros of $\seq{P_n(x)}$ are all located on a bounded interval $\left(0,\mathrm{K}(r,r)\right)$.
Based on \cite[Th.~1.1]{AptKalLagoRochaLimitBehaviour}, this is a corollary of the boundedness of the recurrence coefficients in \eqref{recurrence coefficients r-OPS asymptotic behaviour, s=r}.
In fact, we show in Theorem \ref{asymptotic zero distribution, s=r, th.} that this bounded interval is $(0,1)$.

When $s=0$, we can find a simple expression for the constant $\mathrm{K}(r,0)$ in \eqref{upper bound largest zero r-OP}.
Based on Theorem \ref{asymptotic behaviour of the largest zero - general theorem for d-OPS} and recalling the asymptotic behaviour of the recurrence coefficients given in \eqref{recurrence coefficients r-OPS asymptotic behaviour, s=0}, we have
\begin{equation}
x_n^{(n)}
<\min_{t\in\R^+}\left(t+\sum_{k=0}^{r}\binom{r+1}{k+1}t^{-k}\right)n^r+o\left(n^r\right)
=\min_{t\in\R^+}\left(\frac{(t+1)^{r+1}}{t^r}\right)n^r+o\left(n^r\right).
\end{equation}
The minimum appearing in the latter formula is obtained when $t=r$.
Therefore, we find that, when $s=0$, the largest zero of $P_n(x)$ defined by \eqref{hypergeometric r-OP - explicit formula as a s+1Fr, r>=s} satisfies
\begin{equation}
x_n^{(n)}<\frac{(r+1)^{r+1}}{r^r}\,n^r+o\left(n^r\right)
\quad\text{as }n\to\infty.
\end{equation}
The asymptotic zero distribution of $P_n\left(n^r\,x\right)$ on the interval $\left(0,\frac{(r+1)^{r+1}}{r^r}\right)$ is given in \cite[Th.~3.2]{Neuschel2014}.
%\color{red}
%When $s=r$, $\seq{P_n(x)}$ satisfies the recurrence relation \eqref{recurrence relation for a r-OPS} with
%\begin{equation}
%\gamma_n^{[k]}\to\binom{r+1}{k+1}\left(\frac{r^r}{(r+1)^{r+1}}\right)^{k+1}
%\quad\text{as }n\to\infty,
%\quad\text{for each }0\leq k\leq r.
%\end{equation}
%Therefore, based on Theorem \ref{asymptotic behaviour of the largest zero - general theorem for d-OPS}, we find that the zeros of $P_n(x)$ are all located on the interval $(0,1)$, the support of the orthogonality measures, because
%\begin{equation}
%x_n^{(n)}<\min_{t\in\R^+}\left(t+\sum_{k=0}^{r}\binom{r+1}{k+1}\,t^{-k}\right)\frac{r^r}{(r+1)^{r+1}}=1
%\quad\text{for all } n\in\Z^+.
%\end{equation}
%\color{black}

When $0<s<r$, the asymptotic behaviour of the recurrence coefficients is more convoluted and, as a consequence, the constant $\mathrm{K}(r,s)$ in the upper bound for the largest zero becomes more complicated to compute and less sharp.
For instance, see \cite[Cor.~3.6]{PaperTricomiWeights} for an upper bound for the largest zero of the $2$-orthogonal polynomials corresponding to the case $(r,s)=(2,1)$. 

%\newpage
\subsection{A $r$-orthogonal polynomial sequence with constant recurrence coefficients}
\label{r-OP w/ constant rec coeff}
Here we prove that, for $s=r$ and a particular choice of parameters $a_1,\cdots,a_r,b_1,\cdots,b_r\in\R^+$, the $r$-orthogonal polynomials given by \eqref{hypergeometric r-OP - explicit formula as a s+1Fr, r>=s} satisfy a recurrence relation with constant coefficients as follows. 
This is equivalent to say that each diagonal of the corresponding unit-lower-Hessenberg matrix is constant, i.e. that matrix is Toeplitz.
We give explicit formulas for these $r$-orthogonal polynomials and for densities and moments of their orthogonality measures.
\begin{theorem}
\label{r-OP w/ constant rec coeff Th.}
For $r\in\Z^+$, let $\seq{P_n(x)}$ be the polynomial sequence defined by \eqref{hypergeometric r-OP - explicit formula as a s+1Fr, r>=s} with $s=r$,
\begin{equation}
\label{parameters for r-OP with constant recurrence coefficients}
a_i=1+\frac{i}{r+1}=\frac{r+1+i}{r+1},
\quad\text{and}\quad
b_i=1+\frac{i+1}{r}=\frac{r+1+i}{r}
\quad\text{for each }1\leq i\leq r.
\end{equation}
Then:
\begin{enumerate}[label=(\alph*),leftmargin=*]
\item $\seq{P_n(x)}$ satisfies the recurrence relation with constant coefficients
\begin{equation}
\label{recurrence relation with constant coefficients}
P_{n+1}(x)=x\,P_n(x)-\sum_{k=0}^{\min(n,r)}\binom{r+1}{k+1}\left(\frac{r^r}{(r+1)^{r+1}}\right)^{k+1}\,P_{n-k}(x).
\end{equation}

\item $\seq{P_n(x)}$ can be explicitly written by
\begin{equation}
\label{r-OP with constant recurrence coefficients}
P_n(x)=\binom{n+r}{r}\left(\frac{-r^r}{(r+1)^{r+1}}\right)^n\,
\Hypergeometric{r+1}{r}{-n,\frac{n+r+1}{r},\cdots,\frac{n+2r}{r}\vspace*{0,1 cm}}{\frac{r+2}{r+1},\cdots,\frac{2r+1}{r+1}}
\quad\text{for any  }n\in\N.
\end{equation}

\item $\seq{P_n(x)}$ is $r$-orthogonal with respect to the vector of measures $\left(\mu_0,\cdots,\mu_{r-1}\right)$ supported on $(0,1)$ with densities
\begin{equation}
\label{density orthogonality measures r-OP with constant recurrence coefficients}
\mathrm{d}\mu_k(x)
=\MeijerG{r,0}{r,r}{\left.\frac{k+i+1}{r}\right|_{i=1}^{r}\vspace*{0,15 cm}}{\left.\frac{k+i}{r+1}\right|_{\begin{subarray}{l}i=1\\i\neq r+1-k\end{subarray} }^{r+1}}\mathrm{d}x
\quad\text{for  }0\leq k\leq r-1
\end{equation}
and moments
\begin{equation}
\label{moments of the orthogonality measures for r-OP with constant recurrence coefficients}
\int_{0}^{1}x^n\mathrm{d}\mu_k(x)=\left(\frac{r^r}{(r+1)^{r+1}}\right)^n\,\frac{1}{n+1}\,\binom{(r+1)(n+1)+k}{n}
\quad\text{for any  }0\leq k\leq r-1\text{  and  }n\in\N.
\end{equation}
\end{enumerate}
\end{theorem}

Observe that: 
\begin{itemize}[leftmargin=*]
\item
the parameters in \eqref{parameters for r-OP with constant recurrence coefficients} satisfy the conditions in \eqref{BCF conditions for positivity of the coefficients, a_(r+1)=1} for positivity of the branched-continued-fraction coefficients,

\item
the hypergeometric function in \eqref{r-OP with constant recurrence coefficients} is $\left(-\frac{1}{2}\right)$-balanced because
\begin{equation}
\sum_{i=1}^{r}\left(1+\frac{i}{r+1}\right)=r+\frac{1}{r+1}\sum_{i=1}^{r}i=\frac{3r}{2}
\quad\text{and}\quad
-n+\sum_{i=1}^{r}\left(1+\frac{n+i}{r}\right)=r+\frac{1}{r}\sum_{i=1}^{r}i=\frac{3r+1}{2},
\end{equation}

\item
the moments in \eqref{moments of the orthogonality measures for r-OP with constant recurrence coefficients} are related to the Fuss-Catalan numbers.
\end{itemize}

For $r=1$, the polynomials defined in Theorem \ref{r-OP w/ constant rec coeff Th.} are orthogonal with respect to the measure $\mu_0$ on $(0,1)$ with density $\sqrt{x(1-x)}\mathrm{d}x$ and correspond to the Chebyshev polynomials of the second kind $\seq{U_n(x)}$, up to a linear transformation of the variable (see \cite[Eqs.~4.5.15,~4.5.22]{IsmailBook}):
\begin{equation}
U_n(x)=(n+1)\,\Hypergeometric[\frac{1-x}{2}]{2}{1}{-n,n+2}{\frac{3}{2}}=(-4)^n\,P_n\left(\frac{1-x}{2}\right).
\end{equation}
For $r=2$, the polynomials defined in Theorem \ref{r-OP w/ constant rec coeff Th.} reduce to the $2$-orthogonal polynomial sequence with constant recurrence coefficients introduced in \cite[\S 4.5]{PaperHypergeometricWeights}.
The densities of the corresponding orthogonality measures $\left(\mu_0,\mu_1\right)$ are rational functions given in \cite[Eq.~126]{PaperHypergeometricWeights}.
It is natural to ask whether the Meijer G-functions in \eqref{density orthogonality measures r-OP with constant recurrence coefficients} can also be expressed as rational functions for $r\geq 3$.

\begin{proof}
Firstly, we prove (a). %that \eqref{recurrence relation with constant coefficients} holds.
Recalling Theorem \ref{recurrence coefficients for the MOP and link to the BCF for ratios of hypergeometric series}, $\seq{P_n(x)}$ satisfies the recurrence relation
\begin{equation}
\label{recurrence relation for a r-OPS 2}
P_{n+1}(x)=x\,P_n(x)-\sum_{k=0}^{r}\gamma_{n-k}^{\,[k]}\,P_{n-k}(x),
\end{equation}
with coefficients
\begin{equation}
\label{recurrence coefficients for our r-OPS}
\gamma_n^{[k]}=\sum_{r\geq\ell_0>\cdots>\ell_k\geq 0}\,\prod_{j=0}^{k}\alpha_{(r+1)(n+j)+\ell_j}
\quad\text{for all } n\in\N \text{ and } 0\leq k\leq r,
\end{equation}
where $\alpha_j=0$ for $0\leq j\leq r-1$ and, recalling \eqref{BCF coeff m+1Fm},
\begin{equation}
\label{BCF coeff r+1Fr simplified}
\alpha_{k+r}=\frac{\left(b'_k-a'_k\right)\prod\limits_{i=1}^{r}a'_{k-i}}{\prod\limits_{i=0}^{r}b'_{k-i}}
\quad\text{   for any   }k\in\N,
\end{equation}
with
\begin{equation}
\label{parameters a_k and b_k, BCF r+1Fr}
a'_k=a_{[k]_{r+1}}+\ceil{\frac{k}{r+1}}
\quad\text{and}\quad
b'_k=b_{[k]_r}+\ceil{\frac{k}{r}}.
\end{equation}
	
To prove (a) is equivalent to show that
\begin{equation}
\label{constant rec coeff}
\gamma_n^{[k]}=\binom{r+1}{k+1}\left(\frac{r^r}{(r+1)^{r+1}}\right)^{k+1}
\quad\text{for all } n\in\N \text{ and } 0\leq k\leq r.
\end{equation}
	
Taking the parameters in \eqref{parameters for r-OP with constant recurrence coefficients}, we have
	\begin{subequations}
		\begin{itemize}
			\item
			$\dis b'_k=b_{[k]_r}+\ceil{\frac{k}{r}}=1+\frac{[k]_r+1}{r}+\frac{k-[k]_r}{r}+1=2+\frac{k+1}{r}=\frac{k+2r+1}{r}$ for any $k\in\N$,
			\vspace*{0,1 cm}
			\hfill\refstepcounter{equation}\textup{(\theequation)}\label{b'_n constant rec coeff parameters}
			
			\item
			$\dis a'_k=a_{[k]_{r+1}}+\ceil{\frac{k}{r+1}}=1+\frac{[k]_{r+1}}{r+1}+\frac{k-[k]_{r+1}}{r+1}+1=2+\frac{k}{r+1}=\frac{k+2r+2}{r+1}$ if $(r+1)\nmid k$,
			\vspace*{0,1 cm}
			\hfill\refstepcounter{equation}\textup{(\theequation)}\label{a'_n constant rec coeff parameters n not multiple of r+1}
			
			\item
			$\dis a'_{(r+1)n}=a_{r+1}+\ceil{\frac{(r+1)n}{r+1}}=n+1$ for any $n\in\N$.
			\hfill\refstepcounter{equation}\textup{(\theequation)}\label{a'_n constant rec coeff parameters n multiple of r+1}
		\end{itemize}
	\end{subequations}
	
	%\newpage
	We can now compute the values of $\alpha_{k+r}$ for $k\in\N$.
	
	Firstly, we take $k=(r+1)n$ with $n\in\N$ to compute $\alpha_{(r+1)n+r}$.
	Then,
	%\begin{equation}
	%\alpha_{(r+1)n+r}=\frac{\left(b'_{(r+1)n}-a'_{(r+1)n}\right)\prod\limits_{i=1}^{r}a'_{(r+1)n-i}}{\prod\limits_{i=0}^{r}b'_{(r+1)n-i}}.
	%\end{equation}
	%
	%Moreover,
\begin{subequations}
	\begin{itemize}
		\item
		$\dis b'_{(r+1)n}-a'_{(r+1)n}=\frac{(r+1)n+2r+1}{r}-(n+1)=\frac{n+r+1}{r}$,
		\vspace*{0,1 cm}
		\hfill\refstepcounter{equation}\textup{(\theequation)}		
		
		\item
		$\dis\prod\limits_{i=0}^{r}b'_{(r+1)n-i}
		=\prod\limits_{j=0}^{r}b'_{(r+1)(n-1)+j+1}
		=\prod\limits_{j=0}^{r}\frac{(r+1)(n+1)+j}{r}
		=\frac{\pochhammer[r+1]{(r+1)(n+1)}}{r^{r+1}}$,
		\vspace*{0,1 cm}
		\hfill\refstepcounter{equation}\textup{(\theequation)}		
		
		\item
		$\dis\prod\limits_{i=1}^{r}a'_{(r+1)n-i}
		=\prod\limits_{j=1}^{r}a'_{(r+1)(n-1)+j}
		=\prod\limits_{j=1}^{r}\frac{(r+1)(n+1)+j}{r+1}
		=\frac{\pochhammer[r]{(r+1)(n+1)+1}}{(r+1)^r}$.
		\hfill\refstepcounter{equation}\textup{(\theequation)}		
	\end{itemize}
\end{subequations}	

	Therefore, recalling \eqref{BCF coeff r+1Fr simplified}, we have
	\begin{equation}
		\label{alpha_(k+r), k multiple of r+1}
		\alpha_{(r+1)n+r}
		=\frac{\left(b'_{(r+1)n}-a'_{(r+1)n}\right)\prod\limits_{i=1}^{r}a'_{(r+1)n-i}}{\prod\limits_{i=0}^{r}b'_{(r+1)n-i}}
		=\frac{r^r(n+r+1)}{(r+1)^{r+1}(n+1)}
		\quad\text{for all  }n\in\N.
	\end{equation}
	Next, we compute $\alpha_{(r+1)n+j}$, with $n\geq 1$ and $0\leq j\leq r-1$. 
	We take $k=(r+1)(n-1)+(j+1)$, so that $k+r=(r+1)n+j$.
	%\begin{equation}
	%\alpha_{(r+1)n+j}=\frac{\left(b'_k-a'_k\right)\prod\limits_{i=1}^{r}a'_{k-i}}{\prod\limits_{i=0}^{r}b'_{k-i}}
	%\quad\text{with }k=(r+1)(n-1)+(j+1).
	%\end{equation}
	Note that $1\leq j+1\leq r$, so $(r+1)\nmid k$.
	Therefore,
\begin{subequations}
	\begin{itemize}
		\item
		$\dis b'_k-a'_k=\frac{k+1}{r}-\frac{k}{r+1}=\frac{k+r+1}{r(r+1)}=\frac{(r+1)n+(j+1)}{r(r+1)}$,
		\hfill\refstepcounter{equation}\textup{(\theequation)}
		\vspace*{0,1 cm}
		
		\item
		$\dis\prod\limits_{i=0}^{r}b'_{k-i}
		=\prod\limits_{i=0}^{r}\frac{k-i+2r+1}{r}
		=\prod\limits_{i=0}^{r}\frac{k+r+1+i}{r}
		=\frac{\pochhammer[r+1]{k+r+1}}{r^{r+1}}
		=\frac{\pochhammer[r+1]{(r+1)n+(j+1)}}{r^{r+1}}$,
		\hfill\refstepcounter{equation}\textup{(\theequation)}
		\vspace*{0,1 cm}
		
		\item
		$\dis\prod\limits_{i=1}^{r}a'_{k-i}
		=\prod\limits_{i=0}^{r-1}a'_{(r+1)(n-1)+(j-i)}
		=n\prod\limits_{i=0,\,i\neq j}^{r-1}\frac{(r+1)(n+1)+(j-i)}{r+1}
		%=\frac{n\pochhammer[r]{(r+1)(n+1)+j+1-r}}{(r+1)1^r(n+1)}
		=\frac{n\pochhammer[r]{(r+1)n+j+2}}{(r+1)^r(n+1)}$.
		\hfill\refstepcounter{equation}\textup{(\theequation)}
	\end{itemize}
\end{subequations}	
	
	Hence, recalling again \eqref{BCF coeff r+1Fr simplified}, we find that
	\begin{equation}
		\label{alpha_(k+r), k not multiple of r+1}
		\alpha_{(r+1)n+j}=\frac{r^r\,n}{(r+1)^{r+1}(n+1)}
		\quad\text{for all  }n\geq 1\text{  and  }0\leq j\leq r-1.
	\end{equation}
This formula is also valid for $n=0$, because then it reduces to $\alpha_j=0$ for all $0\leq j\leq r-1$.
	
	Now we can compute $\gamma_n^{[k]}$ for any $n\in\N$ and $0\leq k\leq r$ by inputting \eqref{alpha_(k+r), k multiple of r+1} and \eqref{alpha_(k+r), k not multiple of r+1} in \eqref{recurrence coefficients for our r-OPS}.
	If $\ell_0=r$, %then
	\begin{equation}
		\prod_{j=0}^{k}\alpha_{(r+1)(n+j)+\ell_j}
		=\frac{r^r(n+r+1)}{(r+1)^{r+1}(n+1)}\prod_{j=1}^{k}\frac{r^r(n+j)}{(r+1)^{r+1}(n+j+1)}
		=\left(\frac{r^r}{(r+1)^{r+1}}\right)^{k+1}\frac{n+r+1}{n+k+1},
	\end{equation}
	and, if $\ell_0\leq r-1$, %then
	\begin{equation}
		\prod_{j=0}^{k}\alpha_{(r+1)(n+j)+\ell_j}
		=\prod_{j=0}^{k}\frac{r^r(n+j)}{(r+1)^{r+1}(n+j+1)}
		=\left(\frac{r^r}{(r+1)^{r+1}}\right)^{k+1}\frac{n}{n+k+1}.
	\end{equation}
	In \eqref{recurrence coefficients for our r-OPS}, there are $\binom{r}{k}$ summands with $\ell_0=r$ and $\binom{r}{k+1}$ summands with $\ell_0\leq r-1$, corresponding, respectively, to choosing $k$ elements of $\{0,\cdots,r-1\}$ to be $\ell_1,\cdots,\ell_{r-1}$ and to choosing $k+1$ elements of $\{0,\cdots,r-1\}$ to be $\ell_0,\cdots,\ell_{r-1}$.
	Therefore, we have
	\begin{equation}
		\gamma_n^{[k]}
		=\left(\frac{r^r}{(r+1)^{r+1}}\right)^{k+1}\left(\binom{r}{k}\frac{n+r+1}{n+k+1}+\binom{r}{k+1}\frac{n}{n+k+1}\right)
		=\left(\frac{r^r}{(r+1)^{r+1}}\right)^{k+1}\binom{r+1}{k+1},
	\end{equation}
	which means that \eqref{constant rec coeff} and, consequently, \eqref{recurrence relation with constant coefficients} hold.
	
Now we prove (b). %that \eqref{r-OP with constant recurrence coefficients} holds as well.
	Combining \eqref{hypergeometric r-OP - explicit formula as a s+1Fr, r>=s} with $s=r$ and \eqref{a,b} with $m=r$, we get
	\begin{equation}
		P_n(x)=\frac{(-1)^n\pochhammer{a_1}\cdots\pochhammer{a_r}}{\pochhammer{b'_{n-r}}\cdots\pochhammer{b'_{n-1}}}
		\,\Hypergeometric{r+1}{r}{-n,b'_{n-r},\cdots,b'_{n-1}}{a_1,\cdots,a_r}.
	\end{equation}
	Therefore, recalling \eqref{parameters for r-OP with constant recurrence coefficients} and \eqref{b'_n constant rec coeff parameters}, we have
	\begin{equation}
		\label{r-OP with constant recurrence coefficients*}
		P_n(x)=\frac{(-1)^n\pochhammer{\frac{r+2}{r+1}}\cdots\pochhammer{\frac{2r+1}{r+1}}\vspace*{0,1 cm}}{\pochhammer{\frac{n+r+1}{r}}\cdots\pochhammer{\frac{n+2r}{r}}}
		\,\Hypergeometric{r+1}{r}{-n,\frac{n+r+1}{r},\cdots,\frac{n+2r}{r}}{\frac{r+2}{r+1},\cdots,\frac{2r+1}{r+1}}.
	\end{equation}
	Furthermore,
	\begin{equation}
	\label{prod of pochhammers 1}
	\prod_{i=1}^{r}\pochhammer{\frac{n+r+i}{r}}
		=\prod_{i=1}^{r}\prod_{j=1}^{n}\frac{n+rj+i}{r}
		%=r^{-rn}\prod_{i=1}^{r}\prod_{j=1}^{n}(n+rj+i)
		=r^{-rn}\prod_{k=r+1}^{r(n+1)}(n+k)
		=\frac{\pochhammer[rn]{n+r+1}}{r^{rn}},
	\end{equation}
	and
	\begin{equation}
	\label{prod of pochhammers 2}
		\prod_{i=1}^{r}\pochhammer{\frac{r+i+1}{r+1}}
		=\prod_{i=1}^{r}\prod_{j=1}^{n}\frac{(r+1)j+i}{r+1}
		%=(r+1)^{-rn}\prod_{i=1}^{r}\prod_{j=1}^{n}(rj+i+1)
		=(r+1)^{-rn}\prod_{\substack{k=r+2\\(r+1)\nmid k}}^{(r+1)n+r}k
		%=\frac{\left((r+1)n+r\right)!}{(r+1)^{rn}(r+1)!}\left(\prod_{l=2}^{n}l(r+1)l\right)^{-1}
		=\frac{\left((r+1)n+r\right)!}{(r+1)^{rn+n-1}\,n!\,(r+1)!}
		=\frac{\pochhammer[r(n+1)]{n+1}}{r!(r+1)^{(r+1)n}}.
	\end{equation}
	Hence,
	\begin{equation}
		\frac{(-1)^n\pochhammer{\frac{r+2}{r+1}}\cdots\pochhammer{\frac{2r+1}{r+1}}\vspace*{0,1 cm}}{\pochhammer{\frac{n+r+1}{r}}\cdots\pochhammer{\frac{n+2r}{r}}}
		=\frac{(-1)^n\,r^{rn}\pochhammer[r(n+1)]{n+1}}{r!(r+1)^{(r+1)n}\pochhammer[rn]{n+r+1}}
		%=\frac{\pochhammer[r]{n+1}}{r!}\left(\frac{-r^r}{(r+1)^{r+1}}\right)^n
		=\binom{n+r}{r}\left(\frac{-r^r}{(r+1)^{r+1}}\right)^n,
	\end{equation}
	and \eqref{r-OP with constant recurrence coefficients*} implies \eqref{r-OP with constant recurrence coefficients}.
	
Finally, we prove (c). 
Considering \eqref{orthogonality measures involving Meijer G-functions}-\eqref{parameters of the orthogonality measures involving Meijer G-functions} with the parameters in \eqref{parameters for r-OP with constant recurrence coefficients}, we obtain \eqref{density orthogonality measures r-OP with constant recurrence coefficients} and we find that, for any $0\leq k\leq r-1$ and $n\in\N$,
\begin{equation}
\int_{0}^{1}x^n\mathrm{d}\mu_k(x)
=\frac{\prod\limits_{i=1}^{r+1}\pochhammer{\frac{r+1+k+i}{r+1}}}{(n+1)!\prod\limits_{i=1}^{r}\pochhammer{\frac{r+1+k+i}{r}}}.
%\quad\text{for any  }0\leq k\leq r-1\text{  and  }n\in\N.
\end{equation}
Analogously to \eqref{prod of pochhammers 1}-\eqref{prod of pochhammers 2}, we have
\begin{equation}
\prod_{i=1}^{r}\pochhammer{\frac{r+1+k+i}{r}}=\frac{\pochhammer[rn]{r+k+2}}{r^{rn}},
\quad\text{and}\quad
\prod_{i=1}^{r+1}\pochhammer{\frac{r+1+i+k}{r+1}}=\frac{\pochhammer[(r+1)n]{r+k+2}}{(r+1)^{(r+1)n}}.
\end{equation}
Therefore,
\begin{equation}
\int_{0}^{1}x^n\mathrm{d}\mu_k(x)=\left(\frac{r^r}{(r+1)^{r+1}}\right)^n\frac{\pochhammer{r(n+1)+k+2}}{(n+1)!},
%=\left(\frac{r^r}{(r+1)^{r+1}}\right)^n\,\frac{1}{n+1}\,\binom{r(n+1)+r+1+k}{n}.
\end{equation}
which is equivalent to \eqref{moments of the orthogonality measures for r-OP with constant recurrence coefficients}.
\end{proof}

%\newpage
\subsection{Connection with Jacobi-Pi\~neiro polynomials and asymptotic zero distribution}
\label{Jacobi-Pineiro poly}
Here we show that, for $s=r$ and parameters $a_1,\cdots,a_r,b_1,\cdots,b_r\in\R^+$ satisfying a certain set of relations, the $r$-orthogonal polynomials given by \eqref{hypergeometric r-OP - explicit formula as a s+1Fr, r>=s} correspond to Jacobi-Pi\~neiro polynomials on the step-line.

For $a_1,\cdots,a_r\in\R^+$, let $b_r=a_1+1$ and $b_i=a_{i+1}$ for $1\leq i\leq r-1$.
Note that these parameters satisfy the conditions in \eqref{BCF conditions for positivity of the coefficients, a_(r+1)=1} if and only if $a_1<\cdots<a_r<a_1+1$.
Then, for any $1\leq j\leq r-1$:
\begin{subequations}
\begin{itemize}
\item
$b_i^{(j)}=b_i+1=a_{i+1}+1=a_{i+1}^{(j)}$ for $1\leq i\leq j-1$,
\hfill\refstepcounter{equation}\textup{(\theequation)}\label{parameters for J.-P. poly 1}\vspace*{0,1 cm}

\item
$b_j^{(j)}=b_j+1=a_{j+1}+1=a_{j+1}^{(j)}+1$,
\hfill\refstepcounter{equation}\textup{(\theequation)}\label{parameters for J.-P. poly 2}\vspace*{0,1 cm}

\item
$b_i^{(j)}=b_i=a_{i+1}=a_{i+1}^{(j)}$ for $j+1\leq i\leq r-1$, and
\hfill\refstepcounter{equation}\textup{(\theequation)}\label{parameters for J.-P. poly 3}\vspace*{0,1 cm}

\item
$b_r^{(j)}=b_r=a_1+1=a_1^{(j)}$.
\hfill\refstepcounter{equation}\textup{(\theequation)}\label{parameters for J.-P. poly 4}
\end{itemize}
\end{subequations}

Hence, the orthogonality measures in \eqref{orthogonality measures involving Meijer G-functions}, with $s=r$ and the choice of parameters above, reduce to
\begin{equation}
\label{Jacobi Pineiro orthogonality measures}
\mathrm{d}\mu_j(x)
=\frac{\Gamma\left(a_{j+1}+1\right)}{\Gamma\left(a_{j+1}\right)}\,\frac{1}{x}\,\MeijerG{1,0}{1,1}{a_{j+1}+1}{a_{j+1}}\mathrm{d}x
=a_{j+1}\,x^{a_{j+1}-1}\mathrm{d}x
\quad\text{for all }0\leq j\leq r-1.
\end{equation}

The multiple orthogonal polynomials with respect to $\left(\mu_0,\cdots,\mu_{r-1}\right)$ determined by \eqref{Jacobi Pineiro orthogonality measures} are a particular case of the Jacobi-Pi\~neiro polynomials originally introduced by Pi\~neiro in \cite{Pineiro}. %, with the exponent of $1-x$ in the orthogonality weights equal to $0$.
In fact, the formula \eqref{hypergeometric r-OP - explicit formula as a s+1Fr, r>=s} with $b_r=a_1+1$ and $b_i=a_{i+1}$ for $1\leq i\leq r-1$ reduces to the explicit formula for the Jacobi-Pi\~neiro polynomials given in \cite[Eq.~23.3.5]{IsmailBook}, with $\beta=0$ and the multi-index $\vec{n}=\left(n_1,\cdots,n_r\right)$ on the step-line.
Furthermore, it is clear from \cite[Eq.~23.3.5]{IsmailBook} that the Jacobi-Pi\~neiro polynomials with $\beta\neq 0$ cannot be a particular case of the polynomials given by \eqref{hypergeometric r-OP - explicit formula as a s+1Fr, r>=s}.

%\textbf{Reread from here!}
This connection with the Jacobi-Pi\~neiro polynomials suggests that our polynomials share the asymptotic behaviour of the recurrence coefficients, the ratio asymptotics, and the asymptotic zero distribution with the Jacobi-Pi\~neiro polynomials. 
We show that this is true and, as a result, that we can obtain the asymptotic zero distribution of our polynomials from \cite[Th.~1.1]{WalterVAandT.Neuschel-AsymptoticsJ.-P.}.

The asymptotic zero distribution $\nu$ of $\seq{P_n(x)}$ is the limit (if it exists) for the normalised zero counting measure of $P_n(x)$ in the sense of the weak convergence of measures, that is,
\begin{equation}
\label{asymptotic zero distribution definition}
\int f\mathrm{d}\nu(t)
%=\lim_{n\to\infty}\frac{1}{n}\sum_{P_n(x)=0}f(x)
=\lim_{n\to\infty}\frac{1}{n}\sum_{k=1}^{n}f\left(x_k^{(n)}\right),
\quad\text{where  } x_1^{(n)},\cdots,x_n^{(n)} \text{  are the zeros of  }P_n(x),
\end{equation} 
for all bounded and continuous functions $f$ on $(0,1)$.%, where $x_1^{(n)},\cdots,x_n^{(n)}$ are the zeros of $P_n(x)$.

If the zeros of $\seq{P_n(x)}$ are all real and simple and the zeros of consecutive polynomials interlace and the limit of the ratio of two consecutive polynomials,
%\begin{equation}
%\label{ratio asymptotics definition}
%F(z)=\lim_{n\to\infty}\frac{P_{n+1}(z)}{P_n(z)},
%\end{equation}
which we refer to as the ratio asymptotics of $\seq{P_n(x)}$, exists and converges uniformly on compact subsets of $\C\backslash(0,1)$, the Stieltjes transform of the asymptotic zero distribution $\nu$ is (see \cite[\S 4]{WalterVAandT.Neuschel-AsymptoticsJ.-P.} for more details)
\begin{equation}
\label{asymptotic zero distribution and ratio asymptotics}
\int\frac{\dd\nu(t)}{x-t}
=\lim_{n\to\infty}\frac{1}{n}\sum_{k=1}^{n}\frac{1}{x-x_k^{(n)}}
=\lim_{n\to\infty}\frac{P_n^{\,\prime}(x)}{nP_n(x)}
=\frac{F'(z)}{F(z)}
\quad\text{with  }F(z)=\lim_{n\to\infty}\frac{P_{n+1}(z)}{P_n(z)}.
\end{equation}
%Therefore, determining the ratio asymptotics \eqref{ratio asymptotics definition} is crucial to find the asymptotic zero distribution $\nu$, which may be obtainable via the Stieltjes inverse transform.

Let $a_1,\cdots,a_r,b_1,\cdots,b_r\in\R^+$ satisfying \eqref{BCF conditions for positivity of the coefficients, a_(r+1)=1} with $s=r$ and $\lambda_j=j$ for all $1\leq j\leq r$ and $\seq{P_n(x)}$ be the $r$-orthogonal polynomials defined by \eqref{hypergeometric r-OP - explicit formula as a s+1Fr, r>=s}.
Then, the asymptotic behaviour of the recurrence coefficients of $\seq{P_n(x)}$ is given by \eqref{recurrence coefficients r-OPS asymptotic behaviour, s=r} and it does not depend on the choice of parameters.
As a result, the ratio asymptotics and the asymptotic zero distribution do not depend on the parameters either, because the ratio asymptotics is determined by the asymptotic behaviour of the recurrence coefficients (see \cite[Lemma~3.2]{AptKalLagoRochaLimitBehaviour}) and, recalling \eqref{asymptotic zero distribution and ratio asymptotics}, the asymptotic zero distribution is determined by the ratio asymptotics. 

On the other hand, the ratio asymptotics and the asymptotic zero distribution of the Jacobi-Pi\~neiro polynomials on the step-line were obtained in \cite[Th.~1.1]{WalterVAandT.Neuschel-AsymptoticsJ.-P.}.
In particular, the results in \cite[Th.~1.1]{WalterVAandT.Neuschel-AsymptoticsJ.-P.} hold for our polynomials $\seq{P_n(x)}$ when $b_r=a_1+1$ and $b_i=a_{i+1}$ for $1\leq i\leq r-1$, because we have seen that these polynomials are a particular case of the Jacobi-Pi\~neiro polynomials on the step-line. 
Consequently, those results also hold for any $a_1,\cdots,a_r,b_1,\cdots,b_r\in\R^+$ satisfying \eqref{BCF conditions for positivity of the coefficients, a_(r+1)=1}.
Therefore, we will show that the ratio asymptotics and the asymptotic zero distribution of our polynomials are the same as for the Jacobi-Pi\~neiro polynomials for multi-indices near the diagonal. 

Using \cite[Remark~3.1]{AptKalLagoRochaLimitBehaviour} (with $F_i(x)=F(x)$ for all $i\in\N$) and the asymptotic behaviour \eqref{recurrence coefficients r-OPS asymptotic behaviour, s=r} of the recurrence coefficients, we find that the ratio asymptotics $F(z)$ defined in \eqref{asymptotic zero distribution and ratio asymptotics} satisfies the algebraic equation
\begin{equation}
F(x)=x-\sum_{k=0}^{r}\binom{r+1}{k+1}\left(\frac{r^r}{(r+1)^{r+1}}\right)^{k+1}\left(F(x)\right)^{-k}
\Leftrightarrow
\left(F(x)+\frac{r^r}{(r+1)^{r+1}}\right)^{r+1}=x\left(F(x)\right)^{r}.
\end{equation}
This algebraic equation is equivalent to \cite[Eq.~2.11]{WalterVAandT.Neuschel-AsymptoticsJ.-P.}, with $F(x)=z-p$ and $p=\left(\frac{r}{r+1}\right)^r$.
Therefore, we can replicate the method in \cite[\S 3-4]{WalterVAandT.Neuschel-AsymptoticsJ.-P.} to show that the ratio asymptotics and the asymptotic zero distribution of $\seq{P_n(x)}$ are the same as for the Jacobi-Pi\~neiro polynomials for multi-indices near the diagonal. 
The asymptotic zero distribution is presented in the following result.
\begin{theorem}
\label{asymptotic zero distribution, s=r, th.}
(cf. \cite[Th.~1.1]{WalterVAandT.Neuschel-AsymptoticsJ.-P.})
For $r\in\Z^+$, let $\seq{P_n(x)}$ be the $r$-orthogonal polynomial sequence defined by \eqref{hypergeometric r-OP - explicit formula as a s+1Fr, r>=s} with $s=r$ and $a_1,\cdots,a_r,b_1,\cdots,b_r\in\R^+$ satisfying \eqref{BCF conditions for positivity of the coefficients, a_(r+1)=1} and denote by $x_1^{(n)},\cdots,x_n^{(n)}$ the zeros of $P_n(x)$ for $n\geq 1$.
Then, the asymptotic zero distribution $\nu_r$ of $P_n(x)$,
\begin{equation}
\label{asymptotic zero distribution r-OP}
\int_{0}^{1} f(t)\nu_r(t)\mathrm{d}t
=\lim_{n\to\infty}\frac{1}{n}\sum_{k=1}^{n}f\left(x_k^{(n)}\right),
\end{equation} 
is supported on $[0,1]$ with density
\begin{equation}
\nu_r(x)=\frac{(r+1)\sin\phi\,\sin(r\phi)\sin((r+1)\phi)}{\phi\,x\,\left((r+1)^2\sin^2(r\phi)-2r(r+1)\sin((r+1)\phi)\sin(r\phi)\cos\phi+r^2\sin^2((r+1)\phi)\right)},
\end{equation}
after the change of variables
\begin{equation}
x=\frac{r^r\,\sin^{r+1}\left((r+1)\phi\right)}{(r+1)^{r+1}\,\sin\phi\,\sin^r(r\phi)}
\quad\text{with }\phi\in\left(0,\frac{\pi}{r+1}\right).
\end{equation}
\end{theorem}
This asymptotic zero distribution is related to the Fuss-Catalan numbers (see \cite[\S 3]{WalterVAandT.Neuschel-AsymptoticsJ.-P.} for more details).
The asymptotic zero distribution of $\seq{P_n(x)}$ when $s=0$ presented in \cite[Th.~3.2]{Neuschel2014} is also related to the Fuss-Catalan numbers.
It would be of interest to find the asymptotic zero distribution of $\seq{P_n(x)}$ when $0<s<r$ and check if it is also related to the Fuss-Catalan numbers.

For $r=1$, Theorem \ref{asymptotic zero distribution, s=r, th.} recovers the arcsin density of the asymptotic zero distribution of the Jacobi polynomials on the interval $(0,1)$:
\begin{equation}
\label{asymptotic zero distribution Jacobi poly}
v_1(x)=\frac{1}{\pi\sqrt{x(1-x)}}
\quad\text{for }x\in(0,1).
\end{equation}
For $r=2$, Theorem \ref{asymptotic zero distribution, s=r, th.} reduces to \cite[Th.~10~(b)]{PaperHypergeometricWeights} and the density of the corresponding asymptotic zero distribution was found in \cite[Th.~2.1]{WalterCoussementX2AsymptoticZeroDistribution}:
\begin{equation}
\label{asymptotic zero distribution, s=r=2}
v_2(x)=\frac{\sqrt{3}}{4\pi}\,\frac{\left(1+\sqrt{1-x}\right)^{\frac{1}{3}}+\left(1-\sqrt{1-x}\right)^{\frac{1}{3}}}{x^{\frac{2}{3}}\sqrt{1-x}}
\quad\text{for }x\in(0,1).
\end{equation}

Reciprocally, if the ratio asymptotics $F(z)$ defined in \eqref{asymptotic zero distribution and ratio asymptotics} exists and converges uniformly for a $r$-orthogonal polynomial sequence $\seq{P_n(x)}$, then the asymptotic behaviour of the recurrence coefficients is uniquely determined by this limit (see the proof of \cite[Lemma~3.1]{AptKalLagoRochaLimitBehaviour} for an algorithm to compute the asymptotic behaviour of the recurrence coefficients from the ratio asymptotics).
Therefore, the asymptotic behaviour of the recurrence coefficients for the Jacobi-Pi\~neiro polynomials on the step-line is the same as for the $r$-orthogonal polynomial sequence $\seq{P_n(x)}$ in Theorem \ref{asymptotic zero distribution, s=r, th.} and it is given by \eqref{recurrence coefficients r-OPS asymptotic behaviour, s=r}.

%\newpage
\section*{Final remarks}
%\label{Final remarks}
The investigation presented in this paper provides an excellent example on how the study of the connection between multiple orthogonal polynomials and branched continued fractions leads to considerable advances on both topics.
We will produce analogous investigations for other examples of this connection in future work, giving rise to new multiple orthogonal polynomials and branched continued fractions.

The focus in this paper was only on the type II multiple orthogonal polynomials.
In forthcoming work, we will investigate the corresponding type I multiple orthogonal polynomials and study the connection between type I multiple orthogonal polynomials and  branched continued fractions in a more general setting.
%Other questions remain unanswered about the multiple orthogonal polynomials discussed here and some of them were mentioned throughout the paper.

We are also interested in exploring the applications of the multiple orthogonal polynomials studied here to other fields of Mathematics.
Particular cases of these polynomials have direct applications in the analysis of singular values of products of random matrices (see \cite{KuijlaarsZhang14}), are known to be random walk polynomials (see \cite{HypergeometricMOPandRandomWalks}), and are connected to the study of rational solutions of Painlev\'e equations (see \cite{PeterCLizM2003}).
We expect that these applications can be extended for the general class of multiple orthogonal polynomials investigated in this paper and we intend to explore those extensions in future work.
%under the same conditions as in Section \ref{Type II MOP w.r.t. Meijer G-functions}.

%\newpage
\noindent\textbf{Acknowledgements:}
I am very grateful to Alan Sokal for many enlightening discussions about all the topics of the investigation presented here, especially about the topics concerning lattice paths, branched continued fractions, and total positivity, for numerous pertinent suggestions that considerably improved this paper, and for kindly sharing draft versions of \cite{AlanSokalMOPd-opProdMatBCF} and \cite{AlanSokalBishalDebM.PetroelleB.ZhuBCF3Laguerre},
to Ana Loureiro for illuminating discussions on several results presented here, particularly related to multiple orthogonal polynomials,
and to LMS and EPSRC for their support of this work through grants ECF-1920-18 and EP/W522454/1, respectively.

\bibliographystyle{plain}
\bibliography{References}

\end{document}